\documentclass[3p,preprint,review]{elsarticle}
\usepackage[utf8]{inputenc}
\usepackage{setspace}

\usepackage{graphicx}%
\usepackage{multirow}%
\usepackage{amsmath,amssymb,amsfonts}%
\usepackage{amsthm}%
\numberwithin{equation}{section}
\numberwithin{figure}{section}
\usepackage[bb=dsserif,bbscaled=1.2]{mathalpha}
\usepackage{bm}
\usepackage{mathrsfs}%
\usepackage{mathtools}
\usepackage[title]{appendix}%
\usepackage{xcolor,xspace}%
\usepackage{textcomp}%
\usepackage{manyfoot}%
\usepackage{url}
\usepackage{algorithm}
\usepackage{algpseudocode}
\usepackage{relsize}
\usepackage{placeins}
\usepackage{morefloats}

\newcommand{\ordersHistory}{\ensuremath{\mathcal{P}_{\text{history}}}\xspace}
\newcommand{\orders}{\ensuremath{\mathcal{P}}\xspace}
\newcommand{\layout}{\ensuremath{\mathcal{H}}\xspace}
\newcommand{\interfaces}{\ensuremath{\mathcal{I}}\xspace}
\newcommand{\movers}{\ensuremath{\mathcal{M}}\xspace}
\newcommand{\drugs}{\ensuremath{\mathcal{G}}\xspace}
\newcommand{\packedDispensers}{\ensuremath{\mathcal{D}}\xspace}

\newcommand{\numTiles}{\ensuremath{n_{\text{tiles}}}\xspace}
\newcommand{\numMovers}{\ensuremath{n_{\text{movers}}}\xspace}
\newcommand{\maxMovers}{\ensuremath{m_{\text{max}}}\xspace}
\newcommand{\numDispensers}{\ensuremath{n_{\text{disp}}}\xspace}
\newcommand{\numInterfaces}{\ensuremath{n_{\text{inter}}}\xspace}

\newcommand{\drugUtil}{\ensuremath{u_g}\xspace}
\newcommand{\dispenserCount}{\ensuremath{z_g}\xspace}
\newcommand{\dispenserLoad}{\ensuremath{\pi_g}\xspace}
\newcommand{\tileLoad}{\ensuremath{\mu_k}\xspace}
\newcommand{\maxTileLoad}{\ensuremath{\mu_{\max}}\xspace}
\newcommand{\hatMaxLoad}{\ensuremath{\hat{\mu}_{\max}}\xspace}
\newcommand{\hatDispenserCount}[1][g]{\ensuremath{\hat{z}_{#1}}\xspace}
\newcommand{\hatDispenserLoad}{\ensuremath{\hat{\pi}_g}\xspace}
\newcommand{\fillDuration}{\ensuremath{\Delta^g_p}\xspace}
\newcommand{\placement}{\ensuremath{\varphi}\xspace}
\newcommand{\placementExt}{\ensuremath{\placement: \packedDispensers \cup \interfaces \to \layout}\xspace}
\newcommand{\permElement}[1][]{\ensuremath{{\sigma_{#1}}}\xspace}
\newcommand{\tilesDistance}[1]{\ensuremath{\| \placement(\permElement_{#1}) - \placement(\permElement_{#1+1}) \|_1}\xspace}

\newcommand{\distMatrix}{\ensuremath{\mathcal{L}_{\text{tiles}}}\xspace}
\newcommand{\schedule}{\ensuremath{\mathbf{s}}\xspace} 
\newcommand{\tasks}{\ensuremath{\mathcal{T}}\xspace}
\newcommand{\sites}{\ensuremath{\mathcal{R}}\xspace}
\newcommand{\transits}{\ensuremath{\mathcal{W}}\xspace}
\newcommand{\taskAlternatives}{\ensuremath{\mathcal{A}}\xspace}

\newcommand{\tasksForPatient}[1]{\ensuremath{\mathcal{S}_{#1}}\xspace}

\newcommand{\trayOperation}{\ensuremath{\eta^{\interfaces}}\xspace}
\newcommand{\shortestPathPatient}{\ensuremath{\kappa_p}\xspace}
\newcommand{\patientPathTime}[1][p]{\ensuremath{\eta'_{#1}}\xspace}

\newcommand{\dispensersOnTile}[1]{\ensuremath{\varphi^{-1}(#1)}\xspace}

\usepackage{tabularx}
\newcolumntype{C}{>{\centering\arraybackslash}X}
\newcolumntype{L}{>{\raggedright\arraybackslash}X}
\newcolumntype{R}{>{\raggedleft\arraybackslash}X}
\usepackage{booktabs}%
\usepackage{listings}%
\usepackage{todonotes}
\usepackage{pgfplots}

\usepackage{caption}
\usepackage{subcaption}
\usepackage{soul}
\usepackage{pgfgantt}
\usepackage{pifont}% http://ctan.org/pkg/pifont
\usepackage{csvsimple}
\usepackage{float}
\usepackage{pgfplots}
\usepackage{pgfplotstable}
\usepgfplotslibrary{statistics}
\usetikzlibrary{decorations.pathreplacing,calc,positioning,calligraphy,arrows.meta,matrix}
\pgfplotsset{compat=1.18}
\usepackage[hidelinks]{hyperref}
\usepackage{longtable}
\usepackage{booktabs}
\biboptions{authoryear}

\makeatletter
\newcommand\resetstackedplots{
\makeatletter
\pgfplots@stacked@isfirstplottrue
\makeatother
\addplot [forget plot,draw=none] coordinates{(1,0) (2,0) (3,0) (4,0) (5,0) (6,0) (7,0) (8,0) (9,0) (10,0) (11,0) (12,0) (13,0) (14,0) (15,0)};
}
\makeatother

\newcommand{\m}[1]{\boldsymbol{#1}}

\newcommand{\tn}[2][inline]{\todo[#1, color=green!20]{\textbf{TN:} #2}}

\newtheorem{definition}{Definition}%
\newtheorem{proposition}{Proposition}

\newcommand{\eqnleft}[1]{%
    \mathllap{\text{#1}\quad\quad\quad} 
}
\usepackage{chngcntr}
\counterwithout{figure}{section}
\allowdisplaybreaks

\title{Integrated packing, placement, scheduling, and routing of personalized production: a pharmaceutical Industry 4.0 use-case with a planar transport system}

\makeatletter
\def\ps@pprintTitle{%
 \let\@oddhead\@empty
 \let\@evenhead\@empty
 \def\@oddfoot{\centerline{\thepage}}%
 \let\@evenfoot\@oddfoot}
\makeatother

\begin{document}

\author[aff1,aff2]{Viktor Emil Korladinov\texorpdfstring{\corref{cor1}}{}}
\ead{korlavik@fel.cvut.cz}
\author[aff1]{Antonín Novák}
\ead{antonin.novak@cvut.cz}
\author[aff1]{Zdeněk Hanzálek}
\ead{zdenek.hanzalek@cvut.cz}
\author[aff3]{Erik Sonntag}
\ead{Erik.Sonntag@vscht.cz}
\author[aff3]{František Štěpánek}
\ead{Frantisek.Stepanek@vscht.cz}

\cortext[cor1]{Corresponding author}
%% Author affiliation
\affiliation[aff1]{organization={Czech Institute of Informatics, Robotics and Cybernetics, Czech Technical University in Prague}, addressline={Jugoslávských partyzánů 1580/3}, postcode={160 00}, city={Prague 6}, country={CZ}}
\affiliation[aff2]{organization={Faculty of Electrical Engineering, Czech Technical University in Prague}, addressline={Karlovo Náměstí 13}, postcode={120 00}, city={Prague 2}, country={CZ}}
\affiliation[aff3]{organization={Department of Chemical Engineering, University of Chemistry and Technology}, addressline={Technická 3}, postcode={166 28}, city={Prague}, country={CZ}}

% Department of Chemical Engineering, University of Chemistry and Technology, Prague, Technická 3, Prague, Praha, 166 28, Czech Republic
\begin{abstract}
The recent emergence of planar transport systems necessitates re-evaluation of Flexible Manufacturing Systems (FMS) to address the simultaneous scheduling of internal logistics and production operations. By operating on a tile-based planar grid, these systems allow independent movers full two-dimensional freedom, mitigating inefficiencies inherent to traditional sequential lines. This paper applies a planar FMS framework to a real-world use case in the pharmaceutical industry: the automated production of personalized drugs.

Implementing this system requires solving optimization problems at both tactical and operational levels. The tactical level involves decisions regarding production line layout and the positioning of drug dispensers. A Mixed-Integer Quadratic Programming model is utilized for the packing problem to exploit drug co-occurrence patterns found in historical patient data. Subsequently, we solve the placement problem — a bi-level problem combining an assignment problem with Shortest Hamiltonian paths with neighborhoods — to arrange dispensers in a layout minimizing expected travel distances.

The operational level is encountered daily, scheduling individual movers to process new orders as quickly as possible. This scheduling problem is formulated using Constraint Programming, modeling movers as reservoir resources to ensure order completeness, complemented by a routing phase using an iterative conflict-resolution mechanism and DAG-based reasoning to convert schedules into conflict-free paths.

Evaluation using real-world prescription data for 40 drugs shows the framework scales efficiently across several layout topologies for up to 500 orders, with schedules that are highly effective and computationally tractable for daily operations.
\end{abstract}

\maketitle

\section{Introduction}

The manufacturing sector is currently undergoing a structural transformation driven by the demand for mass personalization. Unlike traditional mass production, which relies on rigid, linear workflows to achieve efficiency, modern production environments must be able to handle diverse manufacturing procedures with minimal setup and transportation times. 
Flexible Manufacturing Systems (FMS) have already been heavily studied as a solution, e.g., see survey paper by \cite{Yadav03042018}. However, the recent emergence of planar transport systems (e.g., Beckhoff XPlanar\footnote{Please see XPlanar system showcase ($\approx 1$min long) \url{https://www.youtube.com/watch?v=yIiZ_DzV0c8&t=60s}}) creates new opportunities and necessitates a re-evaluation of FMS as it represents a significant technological leap. While two-dimensional transport is not new to FMS, earlier vehicle-based systems usually required robotic arms' assistance to interact with stations. Modern tile-based systems achieve sufficient positioning precision for movers to interact directly with dispensers, eliminating such overhead along with  the inefficiencies inherent to sequential lines.

While this hardware offers unprecedented flexibility, it introduces a massive combinatorial design space exploration that we tackle in this paper. Unlike traditional lines, where the sequence of operations is dictated by the layout, a planar grid allows for independent routing and station sequencing for each order. This creates a complex multi-stage optimization challenge that must integrate tactical decisions (such as resource packing and spatial placement) with operational decisions (including scheduling and conflict-free routing). As shown by \cite{maenhout2013integrated}, decoupling these tactical decisions from daily operations yields good results and robust performance, provided the tactical plans anticipate operational constraints.

However, existing frameworks often treat these problems in isolation. Standard Multi-Agent Path Finding (MAPF) approaches (reviewed in depth by \cite{stern2019mapf}) typically prioritize collision avoidance for fixed layouts, ignoring the broader impact of station placement on scheduling decisions. Engineering intuition is insufficient for such coupled systems, as minor adjustments in station placement can drastically alter throughput. Consequently, there is a critical need for formal optimization frameworks that can mathematically model these interactions and determine cost-effective configurations prior to physical implementation. A simulation tool to help visualize the framework's work and output is available on YouTube\footnote{Simulation tool showcase ($\approx 2$min long) \url{https://www.youtube.com/watch?v=ukVYjog60Js}}.

\subsection{Motivation: Personalized Pharmaceutical Production}

We apply this planar FMS framework to a real-world use case in the pharmaceutical industry: the automated production of personalized medication. The efficient management of polypharmacy is a critical healthcare challenge, with an increasing number of patients prescribed five or more concurrent medications~\citep{wastesson2018update}. Although the pharmaceutical landscape is surprisingly narrow (39 medications account for approximately 50\% of all prescriptions~\citep{quinn2017dataset}), the vast number of potential drug combinations and dosage strengths creates immense combinatorial variability.

This variability renders the traditional mass-production model of Fixed-Dose Combinations economically unfeasible. As noted by \cite{hao2015fixed} and \cite{sutherland2015coprescription}, regimens containing more than three active ingredients become effectively unique to the patient, leading to excessive Stock Keeping Unit proliferation if manufactured via traditional means. Current solutions are polarized between rigid industrial platforms, which lack flexibility \citep{vanhoorne2020recent}, and manual pharmacy compounding, which lacks scalability \citep{carvalho2022role}.

To bridge this gap, the ``Lékobot'' platform has been proposed as a modular approach in which individualized doses are assembled from prefabricated subunits (mini-tablets or pellets) into capsules (illustrated in Figure~\ref{capsules_pic}, motivated by~\cite{sonntag2023method}). This moves dose adjustment from the formulation level to an automated assembly process. To achieve industrial throughput, the system employs planar transport (Figure~\ref{plane_cad}) to route patient-specific batches through a series of dispensers (shown in Figure~\ref{tile_cad}). This application provides a rich instance of the planar FMS problem, requiring the rigorous optimization of layout, resource packing, routing, and scheduling to balance the flexibility of compounding with the efficiency of industrial manufacturing.

\begin{figure}[H]
    \centering
    \begin{subfigure}[t]{0.4\textwidth}
    \centering
    \includegraphics[width=0.8\textwidth]{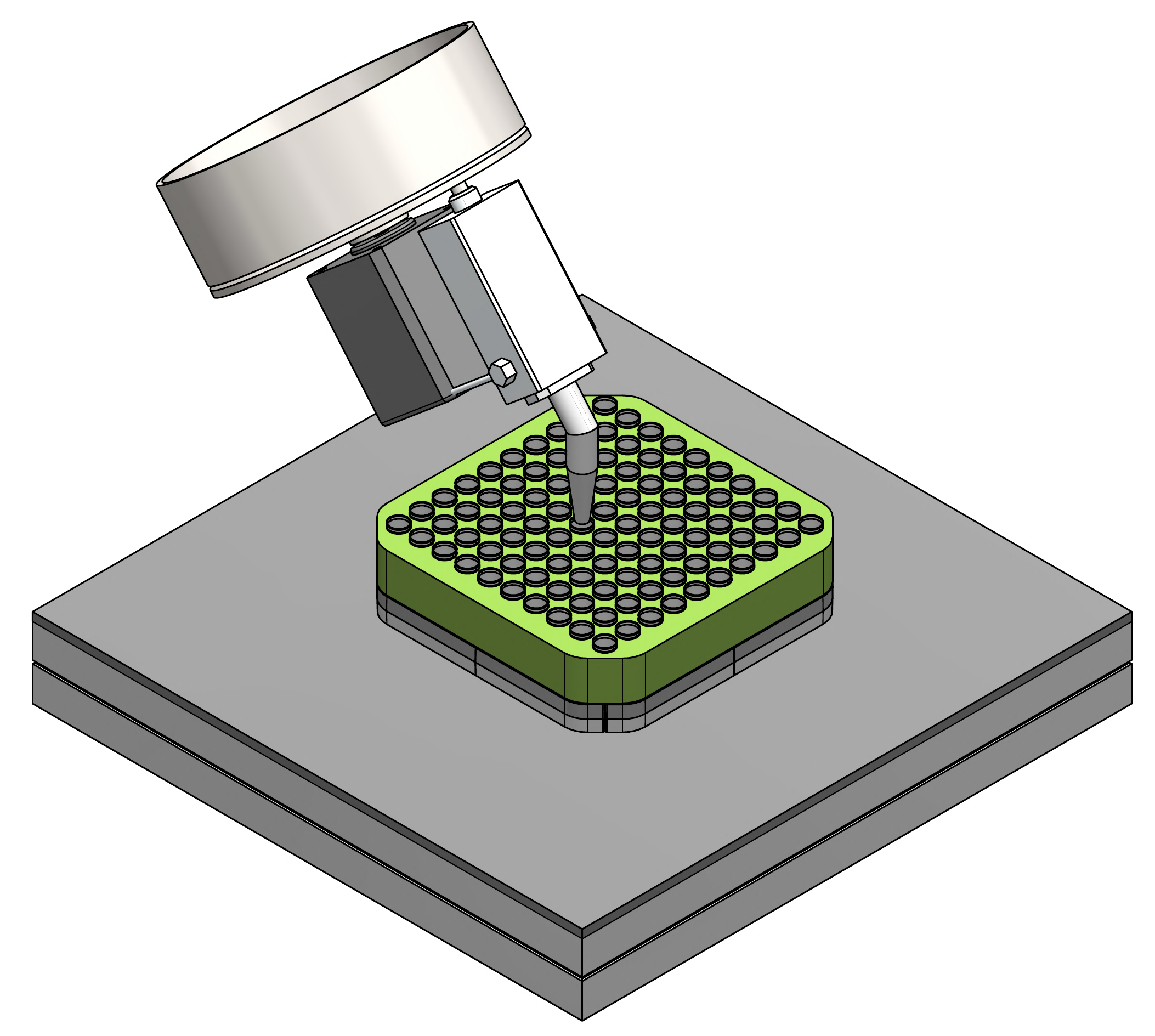}
    \caption{A tile, a mover carrying a 100-capsule cartridge, and an overhead dispenser.}
    \label{tile_cad}
    \end{subfigure}
    ~~~~~
    \begin{subfigure}[t]{0.5\textwidth}
    \centering
    \includegraphics[width=0.75\textwidth]{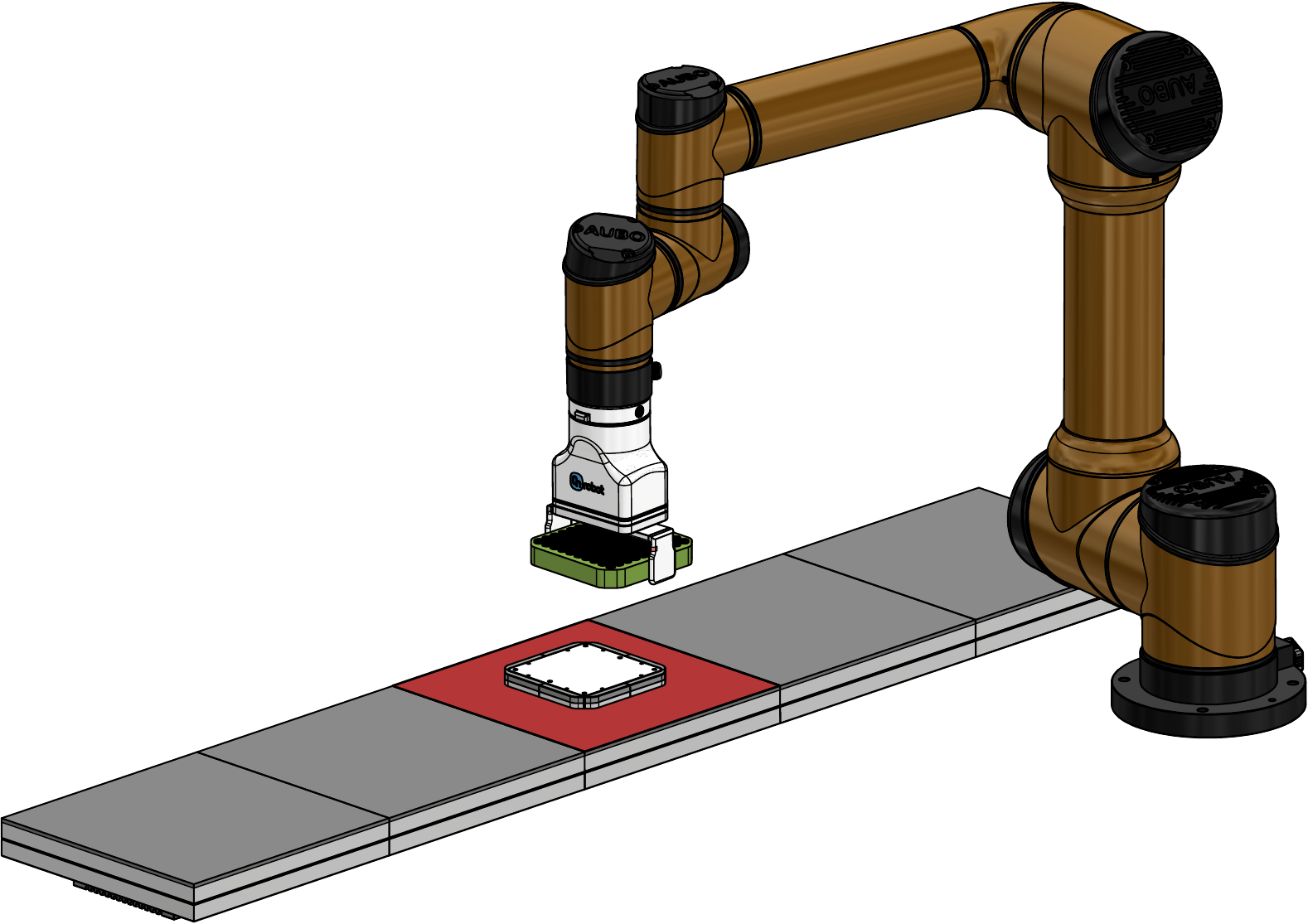}
    \caption{Robot arm interface designed for the loading and unloading of capsule cartridges.}
    \label{interface_cad}
    \end{subfigure}
    \begin{subfigure}[t]{0.34\textwidth}
    \centering
    \includegraphics[width=0.9\textwidth]{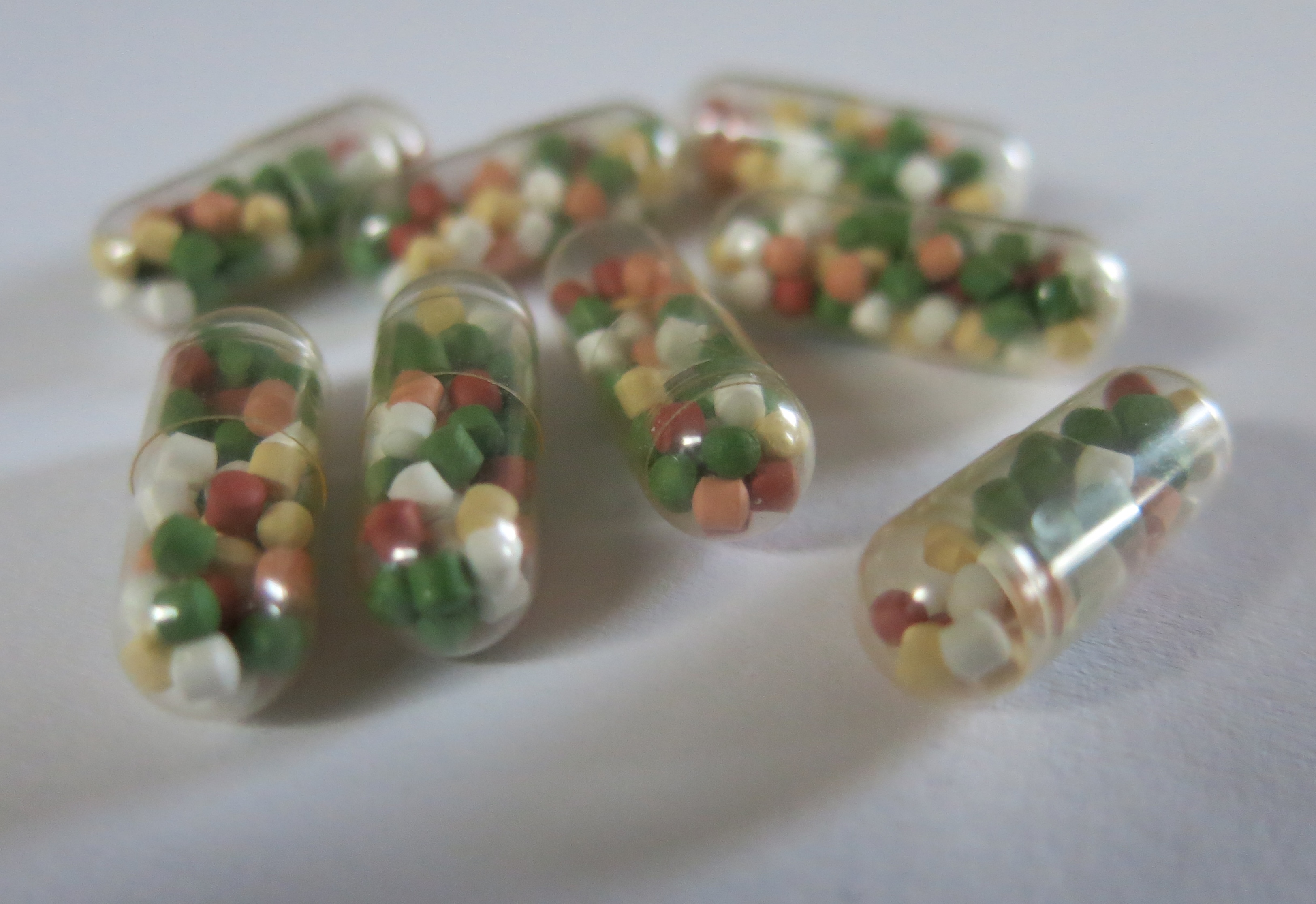}
    \caption{Capsules filled with mini-tablets.}
    \label{capsules_pic}
    \end{subfigure}
    % \begin{subfigure}[t]{0.42\textwidth}
    % \includegraphics[width=1\textwidth]{figs/introduction/platicko_crop.jpg}
    % \caption{platicko - zkusit zmenit za prazdne}
    % \end{subfigure}
    % \caption{real}
    \begin{subfigure}[t]{0.6\textwidth}
    \centering
    \includegraphics[width=0.75\textwidth]{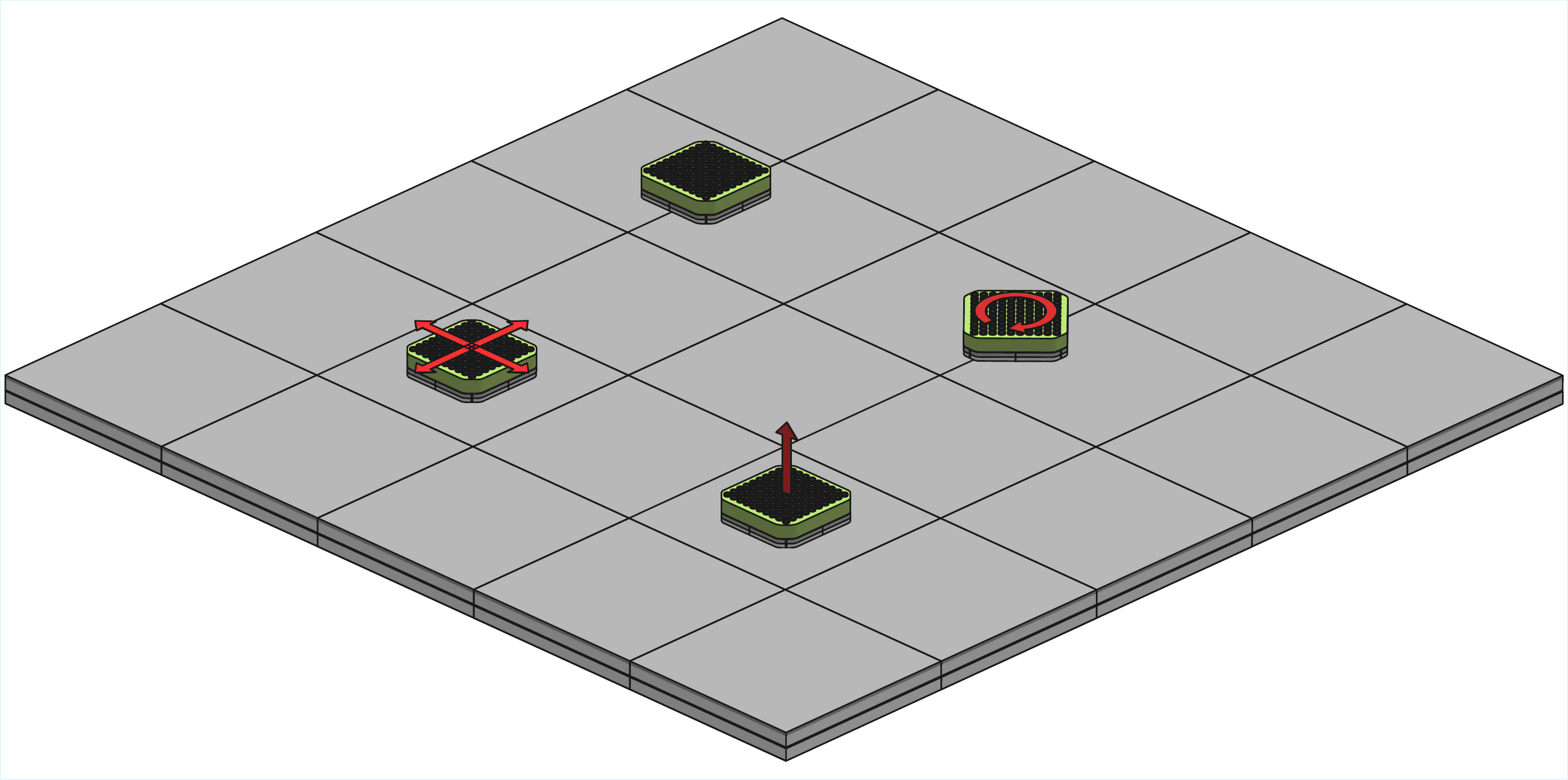}
    \caption{Planar system with 25 tiles and 4 movers.}
    \label{plane_cad}
    \end{subfigure}
    \caption{Components of personalized drug production.}
    \label{system_components}
\end{figure}

\subsection{Contributions and Paper Outline}

The main contributions of this paper are as follows:
\begin{itemize}
    \item We formalize the automated assembly of personalized medication on a planar transport system.
    \item We introduce a comprehensive, multi-stage optimization framework that decouples tactical, long-term decisions (resource packing and spatial placement) from operational, daily execution (scheduling and routing).
    \item At the tactical level, we formulate a Mixed-Integer Quadratic Programming (MIQP) model to exploit drug co-occurrence patterns, and we apply a Genetic Algorithm (GA) with the inner Shortest Hamiltonian path with neighborhood problem to optimize the physical placement of dispensers to minimize expected travel distances.
    \item At the operational level, we develop a Constraint Programming (CP) model that treats movers as take-give reservoir resources, coupled with a Directed Acyclic Graph (DAG) based routing logic to obtain collision-free paths.
    \item We prove the complexity of the placement problem and introduce a lower bound on the makespan that combines parallel machine scheduling with the Noon-Bean transform. 
    \item We evaluate the proposed approach using synthetic data derived from real-world prescription patterns, demonstrating that our framework consistently produces near-optimal makespan while effectively scaling through a batch-merging heuristic.
\end{itemize}

The remainder of this paper is organized as follows. Section 2 formally defines the problem at both the tactical and operational levels. Section 3 reviews related work in flexible manufacturing and multi-agent pathfinding. The tactical level is addressed in Section~\ref{sec:packing} (\textsf{Packing}) and Section~\ref{sec:placement} (\textsf{Placement}). The operational level is covered in Section~\ref{sec:scheduling} (\textsf{Scheduling}) and Section~\ref{sec:routing} (\textsf{Routing}). Section~\ref{sec:experiments} presents the experimental setup and discusses the results. Finally, Section~\ref{sec:conclusion} concludes the paper.

\section{Problem Statement}

As described above, we are faced with several optimization problems at two different levels. 
The tactical level is where long-term decisions, such as the layout of the production line and the number and positions of dispensers and interfaces, are made.
For example, the production line will be constructed in a certain way and will remain in its current configuration for many months.  
The second level---the operational, is encountered on a daily basis when a new set of orders is received, with the goal of scheduling individual movers to produce it as soon as possible.
A diagram illustrating the relationships between the two decision levels is shown in Figure~\ref{fig:diagram}.

\begin{figure}[ht]
    \centering
        \begin{tikzpicture}[scale=0.7, transform shape]

    \definecolor{vibrantblue}{RGB}{135, 206, 250}   % Sky blue
    \definecolor{vibrantgreen}{RGB}{127, 255, 170}  % Aquamarine
    \definecolor{vibrantpurple}{RGB}{218, 112, 214} % Orchid
    \definecolor{vibrantorange}{RGB}{255, 179, 71}  % Bright tangerine
    \definecolor{vibrantpink}{RGB}{255, 105, 180}   % Hot pink
    \definecolor{vibrantyellow}{RGB}{255, 236, 139} % Bright golden yellow

    \draw[green!20!black,fill=green!8,rounded corners=7,thick] (-1, -0.5) rectangle +(12,5) node[left,yshift=-1em,xshift=-0.2em] {\textbf{Tactical level}}; 
    \draw[green!20!black,fill=green!8,rounded corners=7,thick] (5, -1.3) rectangle +(12,-4.5) node[above left,xshift=-0.2em] {\textbf{Operational level}};
    
    % nodes
    \draw[thick,-stealth] (0.5,5.1) -- +(0,-2.1) node[midway, above, yshift=3em]  {dispensers \numDispensers};
    \draw[thick,-stealth] (2.4,5.7) -- +(0,-2.7) node[midway, above, yshift=3.9em]  {maximum dispensers $d_{max}$};
    \draw[thick,-stealth] (5,5.1) -- (5,4.5) node[midway, above, yshift=0.6em]  {historical orders \ordersHistory};
    \draw[thick,-stealth] (7,5.7) -- +(0,-1.2) node[midway, above, yshift=1.7em]  {tiles \numTiles};
    \draw[thick,-stealth] (9.5,5.1) -- (9.5,4.5) node[midway, above, yshift=0.6em]  {maximum movers \maxMovers};

   % \node at (3,4.5) {estimate $\m{\Sigma}$ and $\mathcal{D}$};

    \draw[blue!20!black,fill=blue!15,rounded corners=7,thick]
         (0,0) rectangle (4,3)
         node[above left] {\textsf{Packing}};

    %packing
    \begin{axis}[at={(100,480)},
            ybar stacked,
            height=0.16\textwidth,
            width=0.29\textwidth,
            bar width=3.3pt,
            ymin=170,
            yticklabels={},
            xtick={0,1,2,3,4,5,6,7},
            xticklabels={0,1,2,3,4,5,6,7},
            title={tile utilization}
        ]

        \addplot coordinates {
            (0, 200)
            (1, 220)
            (2, 192)
            (3, 191)
            (4, 186)
            (5, 191)
            (6, 193)
            (7, 187)
            };
        \addplot coordinates {
            (0, 10)
            (1,0)
            (2, 5)
            (3, 16)
            (4, 19)
            (5, 3)
            (6, 6)
            (7, 25)
            };
            
        \addplot coordinates {
            (0, 0)
            (1,0)
            (2, 14)
            (3, 0)
            (4, 0)
            (5, 13)
            (6, 7)
            (7, 0)
            };
    \end{axis}

    % placement

    \def\rows{7}     % Number of rows
    \def\cols{7}     % Number of columns
    \def\size{0.2}   % Size of each square
    \def\margin{0.07} % Margin between squares

    \draw[thick,-stealth] (12,2) -- +(-2,0) node[midway, right,xshift=3em] {layout \layout, $|\layout|=\numTiles+\numInterfaces$};
    \draw[thick,-stealth] (12,1) -- +(-2,0) node[midway, right, xshift=3em] {interfaces \interfaces, $|\interfaces|=\numInterfaces$};

    \draw[blue!20!black,fill=blue!15,rounded corners=7,thick]
         (6,0) rectangle (10,3)
         node[above left] {\textsf{Placement}};

    \draw[thick,-stealth] (4,-3) -- +(1,0) node[midway, left, xshift=-1.5em, yshift=0.1em]  {new customer orders \orders};

    \draw[thick,-stealth] (4,-4) -- +(1,0) node[midway, left, xshift=-1.5em, yshift=0.1em]  {available movers  \movers, $|\movers|=\numMovers$ };

    \draw[blue!20!black,fill=blue!15,rounded corners=7,thick]
         (6,-5) rectangle (10,-2)
         node[above left] {\textsf{Scheduling}};

    \draw[blue!20!black,fill=blue!15,rounded corners=7,thick]
         (12,-5) rectangle (16,-2)
         node[above left] {\textsf{Routing}};

    % arrows between boxes
    \draw[-stealth,thick] (4,1.5) -- +(2,0) node[midway, above] {packing \packedDispensers};
    \draw[-stealth,thick] (7.5,0) -- +(0,-2) node[midway, right, yshift=0.2em] {placement \placementExt};
    \draw[-stealth,thick] (10,-3.5) -- +(2,0) node[midway, above] {schedule $\schedule$};

       % Loop to create the grid
      \foreach \x in {0,...,\numexpr\cols-1} {
        \foreach \y in {0,...,\numexpr\rows-1} {
          \fill[gray!60] 
            ({(7.1+\x * (\size + \margin))}, {0.5+\y * (\size + \margin)}) 
            rectangle ++(\size, \size);
         \fill[gray!60] 
            ({(13.1+\x * (\size + \margin))}, {-4.5+\y * (\size + \margin)}) 
            rectangle ++(\size, \size);
        }
      }
    
      % interfaces
      \fill[darkgray!80] 
            ({(7.1+2 * (\size + \margin))}, {0.5 + 2 * (\size + \margin)}) 
            rectangle ++(\size, \size);
      \fill[darkgray!80] 
            ({7.1+4 * (\size + \margin)}, {0.5 + 3 * (\size + \margin)}) 
            rectangle ++(\size, \size);

      \fill[darkgray!80] 
            ({7.1+3 * (\size + \margin)}, {0.5 + 5 * (\size + \margin)}) 
            rectangle ++(\size, \size);

    % scheduling
        \begin{scope}[shift={(6.7,-3.9)}, scale=0.4]
            
            \foreach \y in {0,1,2} {
                \draw[darkgray!80 ] (0,\y-0.5) -- (7.5,\y-0.5);
            }

            \foreach \y in {1,2,3} {
                \node[text=black!80] at (-0.7,2.9-\y) {\huge $m_\y$};
            }
            
            % Mover 1
            \fill[vibrantblue]   (1,-0.3) rectangle (3.5,0.3);
            \fill[vibrantpurple!85]  (4,-0.3) rectangle (6.6,0.3);
            
            % Mover 2
            \fill[vibrantpurple!85] (0.5,0.7) rectangle (2,1.3);
            \fill[vibrantorange] (2.5,0.7) rectangle (4,1.3);
            
            % Mover 3
            \fill[vibrantpink!80]   (1,1.7) rectangle (3.5,2.3);
            \fill[vibrantblue] (4,1.7) rectangle (6,2.3);
        \end{scope}

    % routing
         
        %drugs
        \fill[vibrantpink] 
          ({13.1+2*(\size+\margin)}, {-4.5+2*(\size+\margin)}) 
          rectangle ++(\size, \size);
        \fill[vibrantblue] 
          ({13.1+4*(\size+\margin)}, {-4.5+3*(\size+\margin)}) 
          rectangle ++(\size, \size);
        \fill[vibrantpurple] 
          ({13.1+3*(\size+\margin)}, {-4.5+5*(\size+\margin)}) 
          rectangle ++(\size, \size);
        \fill[vibrantorange] 
          ({13.1+1*(\size+\margin)}, {-4.5+1*(\size+\margin)}) 
          rectangle ++(\size, \size);
        \fill[vibrantblue] 
          ({13.1+6*(\size+\margin)}, {-4.5+1*(\size+\margin)}) 
          rectangle ++(\size, \size);
        \fill[vibrantpurple] 
          ({13.1+6*(\size+\margin)}, {-4.5+5*(\size+\margin)}) 
          rectangle ++(\size, \size);
        \coordinate (pink)   at ({13.1+2*(\size+\margin)+0.5*\size}, {-4.5+2*(\size+\margin)+0.5*\size});
        \coordinate (blue)   at ({13.1+4*(\size+\margin)+0.5*\size}, {-4.5+3*(\size+\margin)+0.5*\size});
        \coordinate (purple) at ({13.1+3*(\size+\margin)+0.5*\size}, {-4.5+5*(\size+\margin)+0.5*\size});
        \coordinate (orange) at ({13.1+1*(\size+\margin)+0.5*\size}, {-4.5+1*(\size+\margin)+0.5*\size});
        \coordinate (blue2)   at ({13.1+6*(\size+\margin)+0.5*\size}, {-4.5+1*(\size+\margin)+0.5*\size});
        \coordinate (purple2) at ({13.1+6*(\size+\margin)+0.5*\size}, {-4.5+5*(\size+\margin)+0.5*\size});
        
        % first path
        \coordinate (pb-turn) at ([xshift=57*(\size+\margin)]pink);
        \path (pink) -- (pb-turn) coordinate[pos=0.75] (midPB);
        \draw[dashed] (pink) -- (midPB);
        \draw[ ->] (midPB) -- (pb-turn) -- (blue);
        
        % first mover
        \fill[red!60] 
          ([xshift=-0.05cm,yshift=-0.05cm]midPB) 
          rectangle ++(0.1cm,0.1cm);
        
        % second path
        \coordinate (po-turn) at ([yshift=-114*(\size+\margin)]purple);
        \path (purple) -- (po-turn) coordinate[pos=0.5] (midPO);
        \draw[dashed] (purple) -- (midPO);
        \draw[->] (midPO) -- (po-turn) -- (orange);
        
        % second mover
        \fill[red!60] 
          ([xshift=-0.05cm,yshift=-0.05cm]midPO) 
          rectangle ++(0.1cm,0.1cm);
    
        % thrid path
        \path (purple2) -- (blue2) coordinate[pos=0.3] (midPK);
        \draw[dashed] (blue2) -- (midPK);
        \draw[->] (midPK) -- (purple2);
        
        % second mover
        \fill[red!60] 
          ([xshift=-0.05cm,yshift=-0.05cm]midPK) 
          rectangle ++(0.1cm,0.1cm);
        
    \end{tikzpicture}
    
    \caption{Diagram of the proposed solution and mutual interactions of its components.}
    \label{fig:diagram}
\end{figure}
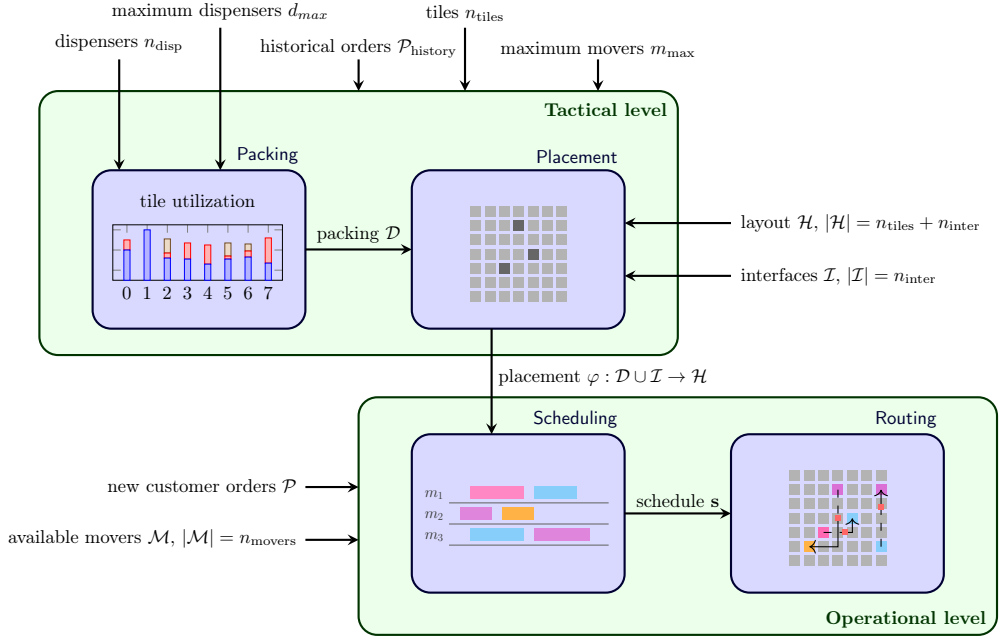

\subsection{Tactical Level}
The goal of the tactical level is to make long-term decisions related to the hardware construction of an efficient production line.
For this, we have access to a set of historical orders $\ordersHistory$ that were processed in the past, which we assume to be a representative sample for making tactical decisions. Each observation in the dataset represents an individual participant, identified by a unique sequence number. Each participant has a list of prescribed pharmacological agents and their corresponding dosage frequencies.

Namely, at the tactical level, we determine how many dispensers we should allocate to each drug and assign these dispenser locations within the given layout of the transportation system on the factory shop floor.
The layout is a subset of 2D integer lattice $\layout\subseteq \mathbb{N}\times\mathbb{N}$.
Each element $(i,j)\in \layout$ corresponds to an available tile in the system at $(x,y)$ coordinates.
For example, a square layout with total of 64 tiles as shown in Figure~\ref{fig:topology-examples:s8x8}, is represented as $\layout_\text{square 64}=\{1,\ldots,8\}^2$ and 
the ring layout with 64 tiles as shown in Figure~\ref{fig:topology-examples:r68} is given as $\layout_\text{ring 64}=\{1,\ldots,17\}^2\setminus\{2,\ldots,16\}^2$.
The tactical-level problem utilizes a set of $\numDispensers$ dispensers to be placed over $\numTiles$ dispensing tiles to provide the supply of a set of drugs $\drugs$ by using the maximum number $\maxMovers$ of available movers.
The mover starts each order on an arbitrary interface tile (to receive an empty cartridge) and ends on an arbitrary interface tile (to release a filled cartridge; the two interfaces can, but do not have to be the same). There are $\numInterfaces=|\interfaces|$ interface tiles.

Each dispenser contains one type of drug (to avoid contamination), but any drug $k\in\drugs$ can be allocated to multiple dispensers.
Further, we assume that each tile can accommodate at most $d_\text{max}$ dispensers due to space constraints (typically 4) or one interface location $i \in \interfaces$.
The upper green area in Figure~\ref{fig:diagram} represents the tactical level of decisions.

% To make estimates about the quality of the above tactical decisions, we have access to a set of historical orders $\ordersHistory$.

In this work, we make a deliberate choice to split the tactical level decisions into two problems. First, in Section~\ref{sec:packing}, we derive a \textsf{Packing} problem that (i)~decides how many dispensers should be allocated for each drug and (ii)~\textit{packs} specific dispensers into \textit{unplaced dispensing tiles} to leverage typical drug co-occurrence in individual orders. At this stage, a dispensing tile is \textit{unplaced}---it only represents a logical grouping of drugs that will eventually share a dispensing tile together, though the physical location of that dispensing tile on the grid remains undecided. Formally, the \textsf{Packing} problem associates drugs with these unplaced dispensing tiles to create a set 
$\packedDispensers = \{ (g, h) \mid g \in \drugs, h \in \{1, \dots, \numTiles\} \}$ 
denoting that drug $g$ is assigned to a specific (yet unplaced) dispensing tile $h$. 

The subsequent \textsf{Placement} step translates these logical groupings into a physical configuration. Having defined the composition of each tile, we must now determine their specific coordinates within the production line. As detailed in Section~\ref{sec:placement}, this involves mapping both the pre-packed dispensing tiles $\packedDispensers$ and the interfaces $\interfaces = \{(\textsc{interface}, h) \mid h \in \{1,\dots,\numInterfaces\}\}$ onto the physical layout $\layout$. 

We seek an assignment $\placementExt$ where each available coordinate $(x,y) \in \layout$ is occupied by either a single interface $i \in \interfaces$ or a dispensing tile. In the latter case, the tile at $(x,y)$ carries the specific set of dispensers $\{g \mid (g,h') \in \packedDispensers\}$ previously grouped into unplaced dispensing tile $h'$ during the packing stage. By optimizing this spatial arrangement, we aim to minimize the makespan of operational schedules for future orders $\orders$, assuming they follow the same distribution as the historical data $\ordersHistory$.

\subsection{Operational Level}
The operational level provides decision-making support to run daily operations, ensuring that orders planned for each day are produced.
Hence, the goal of the operational level is to determine how to operate the number of available movers $\numMovers$ on the given layout of dispensers $\placement$ so that we produce the set of current orders $\orders$.
% \todo{potrebuji nasledujici veliciny: dispensing time toho leku pro toho pacienta}
% \todo{mnozinu leku pro pacienta $p$: $\tasksForPatient{p}$ a dispensing time toho leku pro toho pacienta $\forall t \in \tasksForPatient{p}: \texttt{length}(t)$}
Each order $p \in \orders$ is given by a set of all drugs $\tasksForPatient{p}\subseteq \drugs$. 
The required dispensing quantity of the drug $g$ in order $p$ is an integer $\fillDuration\in\mathbb{N}$.
Hence, the order $p$ is represented as a set of tuples $p = \{(g,\fillDuration) \, | \, \forall g \in \tasksForPatient{p}\}$.
% Parameter $\fillDuration\in\mathbb{N}$ denotes the dispensing time (proportional to the requested amount) of drug $g$ for completing order $p$.

To complete an order $p$, the mover starts and stays for $\trayOperation\in\mathbb{N}$ time units at one of the interface locations (to mount an empty cartridge with capsules), visits all dispensers with required drugs $g\in \tasksForPatient{p}$ in an arbitrary order (i.e., in open-shop manner), spends the requested dispensing time $\fillDuration$ there, and finishes again in an interface location, spending again $\trayOperation$ time units.
Thus, at any moment, the mover processes only one order at a time, and each interface or dispenser tile serves a single mover.
The dispensing process can be preempted at any time and later resumed.
The paths of the movers are Manhattan-style (i.e., follow $\ell_1$ distance)  and are conflict-free.
The goal is to minimize the makespan, that is, the latest completion time of any mover.
See an illustration in Figure~\ref{fig:diagram} where the bottom green area displays all required inputs for making decisions at the operational level.

Similarly to the tactical problem, we deliberately decide to split the operational problem into two sub-problems here.
First, in Section~\ref{sec:scheduling}, we deduce the \textsf{Scheduling} problem that assigns orders $p\in\orders$ to the available amount of moves $\movers$ and decides on the sequencing of the individual dispensing operations.
The setup times between dispensing operations, reflecting the distance between dispensers, are enforced, but we choose to relax on mover conflicts during their transports.
To handle these conflicts, we introduce in Section~\ref{sec:routing} the \textsf{Routing} problem that finally produces  conflict-free paths for the movers by slightly extending the optimistic schedule generated from the \textsf{Scheduling} problem solver.

\section{Related Work}

Flexible Manufacturing Systems (FMS) are well-established in the literature, typically characterized as medium-volume automated systems \citep{Yadav03042018}. However, the emergence of high-precision planar transport systems, such as Beckhoff XPlanar or B\&R ACOPOS 6D, combined with the Industry 4.0 shift toward mass personalization \citep{wang_ind4}, has increased system flexibility to a point that requires updated planning strategies. Motivated by these advancements, this work revisits flexible systems, distinguishing between the tactical and operational decision levels required for effective management.

\cite{maenhout2013integrated} demonstrate that decoupling tactical decisions from operational decisions is essential for robust performance, provided that tactical plans anticipate operational constraints. 
In the context of manufacturing, this allows us to optimize the layout and dispenser configuration (tactical) separately from the daily routing of movers (operational), similar to the robust recovery policies proposed by \cite{akbarzadeh2024study} for healthcare resources.

The operational level handles day-to-day order fulfillment. The need to assign operations to one of several available dispensers mirrors the routing flexibility of the Flexible Job Shop Problem, as reviewed by \cite{DAUZEREPERES2024409}. The drug dispensing process allows for free operation ordering, a core feature of the Open Shop Problem \citep{abreu_contributions_2024}. The coordination of movers requires modeling them as renewable resources. \cite{hartmann_updated_2022} classify such complex constraints in their survey of Resource-constrained Project Scheduling Problem extensions.

What distinguishes this system from classical formulations is its reliance on internal logistics rather than machine reconfiguration. Traditional manufacturing lines are hindered by ``long line'' inefficiency; products must traverse every station regardless of necessity, as seen in customization systems like the ``Lip Picker'' \citep{Go2021}. Planar systems utilize dynamic routing to visit only required dispensers. This modularity allows the system to be reconfigured on the fly to manufacture different drugs. However, this transport-driven flexibility introduces a complex scheduling constraint: the coordination of the transport units themselves.
\citet{laborie_algorithms_2003} defines a reservoir resource, which functions like a fuel tank: it has a capacity and is only assignable when that capacity is positive. A specific subclass of this is the take-give resource proposed by \citet{hanzalek2017time}. In their model, capacity is binary (0 or 1); the resource is "taken" at the start of a task chain and released ("given") only upon completion. \citet{hartmann_updated_2022} gives structure to the different types of resources and classifies take-give resources as partially renewable cumulative resources. In this work, we model the movers as take-give resources, which ensures they remain dedicated to a single order until the entire dynamically routed sequence is complete.

Despite the complexity of the scheduling logic, the physical execution on the planar transport grid allows for some routing flexibility: with our assumed XPlanar tile configuration, two movers can share a single tile, weakening the strict vertex-collision model assumed by classical Multi-Agent Path Finding (MAPF) \citep{stern2019mapf}. Even purpose-built approaches for levitating planar systems \citep{nilsson2022path, rehme2025path} address only point-to-point navigation, and while \cite{janning2025conflict} successfully combine CBS with MPC to guarantee collision-free execution, none of these methods account for the multi-stop paths inherent to complete order fulfillment. Decoupled formulations such as \cite{zahradka2022quality} separate task assignment from routing in the Multi-Agent Multi-Item Pickup and Delivery setting, but applying exact MAPF solvers within an iterative scheduler remains computationally prohibitive: as \cite{fioravantes2024exact} demonstrate, MAPF is theoretically intractable and NP-hard even in basic planar graphs with fixed makespan. We therefore approximate routing logistics via sequence-dependent setup times and the take-give resource mechanism, augmented by an iterative conflict-resolution step, omitting MAPF frameworks from our solution.

% \tn{packing je v zasade zrevidovany}
\section{Packing Problem} \label{sec:packing}

The goal of the \textsf{Packing} problem is to decide how many dispensers will be allocated to each drug, and also which dispensers will be placed together on the same tile to take advantage of typical drug co-occurrences and to use the tile efficiently for less frequent drugs, as shown in Figure~\ref{fig:diagram}.
The input consists of a list of historical orders $\ordersHistory$ (same structure as $\orders$), the total number of tiles $\numTiles\in\mathbb{N}$, the total number of dispensers available $\numDispensers\in\mathbb{N}$, the maximum number of dispensers on a single tile $d_\text{max}\in\mathbb{N}$ and maximum number of movers $\maxMovers\in\mathbb{N}$.
The maximum number of dispensers to be placed on a tile is a parameter that follows from physical constraints imposed by the dimensions of dispensers and tile size.
%Without loss of generality, in this paper, we assume that this value is fixed to 4, which appears to be a reasonable number based on the consultation with domain experts.
We remind that the goal of the packing phase is not to assign tiles into the layout, but only to decide on the (i)~multiplicity of each drug and (ii)~co-occurrence of drugs on the tiles.
% Hence, its output is a set of tiles with assigned drug dispensers, but the tiles are not placed in the layout, which would be the goal of the placement phase.

Therefore, the problem solved by the packing phase is formulated at the capacity level, where \ordersHistory, the list of past orders, is used to estimate the future demand for each type of drug.
The problem is formulated as the following MIQP model.
We are given the following parameters: $\numTiles\in\mathbb{N}$ number of tiles available, $\numDispensers\in\mathbb{N}$ number of dispensers, $\drugUtil \in\mathbb{R}^+_0$ estimated total cumulative dispensing time for drug $g\in \drugs$.
This number essentially represents, e.g., the total time in seconds that will be spent by dispensing drug $g\in \drugs$ over all orders in \ordersHistory.

As for the decision variables, we use $y_{k,g}\in\{0,1\}$ indicator whether drug $g$ is assigned on unplaced tile $k$ (i.e., which is not placed at any particular position at the moment), $\dispenserCount\in\mathbb{N}_0$ is amount of dispensers for drug $g$, $\dispenserLoad\in\mathbb{R}^+_0$ is the estimated load for one dispenser for drug $g$, i.e., $\dispenserLoad =\drugUtil/\dispenserCount$, assuming they will be utilized uniformly.
Indeed, at the tactical stage of decision-making, we do not have access to the information about the scheduling of the specific movers (which is to be revealed at the operational level), hence we assume a uniform distribution of the load.

Furthermore, $\tileLoad\in\mathbb{R}^+_0$ is the total load of all drugs on tile $k$ and $\maxTileLoad=\max_k \tileLoad$ is the maximum total load over all tiles.
Since we have $\numInterfaces$ to be placed on separate tiles without any dispensers, they are not considered in the model, thus the total amount of tiles is decreased by $\numInterfaces$.
The full model is stated as follows:
\begin{align}
    &\min \maxTileLoad\label{eq:pack1-obj} \\
    \text{subject to}&\notag\\
    &\maxTileLoad \geq \tileLoad\quad \forall k \in \{1,\ldots, \numTiles\}\\
    & \dispenserLoad \cdot \dispenserCount = \drugUtil \quad \forall g \in \drugs \label{eq:pack1-bilin}\\
    &\sum^{\numTiles}_{k=1} y_{k,g} = \dispenserCount \quad \forall g \in \drugs\\
    % &1 \leq \dispenserCount \leq \maxMovers \quad \forall g \in \drugs \label{eq:pack1_pg}\\
    % &\dispenserLoad \geq 0.1\cdot \drugUtil \quad \forall i \in I \\
    &1\leq \sum_{g\in \drugs}y_{k,g} \leq d_\text{max} \quad \forall k \in \{1, \ldots, \numTiles\} \label{eq:pack1-max4}\\
    &\sum_{g\in \drugs} \dispenserLoad\cdot y_{k,g} =\tileLoad \quad \forall k \in  \{1, \ldots, {\numTiles}\} \label{eq:pack1-load}\\
    &\sum^{\numTiles}_{k=1} \sum_{g\in \drugs} y_{k,g} \leq \numDispensers \label{eq:pack1-dmax}\\
   % &\dispenserCount \geq 1 \quad \forall i \in I \\
    \text{where}&\notag\\
    &y_{k,g} \in \{0,1\} \quad \forall k \in \{1,\ldots, \numTiles\}, \forall g \in \drugs\\
    &\dispenserCount \in \{0,\ldots, \maxMovers\}, \dispenserLoad \in \mathbb{R}^+_0 \quad \forall g \in \drugs \label{eq:pack1_pg}\\
    % &\dispenserLoad \in \mathbb{R}^+_0 \quad \forall g \in \drugs\\
    &\tileLoad \in \mathbb{R}^+_0 \quad \forall k \in \{1,\ldots, \numTiles\}\\
    &\maxTileLoad\geq0. \label{eq:pack1-finish}
\end{align}
The output of the model is a tile-drug matrix, as shown, for example, in Figure~\ref{fig:drug-tile-d65}.
Drugs are sorted along the horizontal axis in a non-increasing order of their expected total utilization \drugUtil.
The available individual tiles are displayed along the vertical axis without any particular ordering, as they have not yet been placed in the layout.
Each dot in the matrix corresponds to a value $y_{k,g}=1$, meaning that drug $g\in \drugs$ will be placed onto tile~$k$.
For example, we can see that the most requested drug will be allocated to 4 tiles.
On the other hand, drugs along the same row share a common tile.
\begin{figure}[ht]
    \pgfplotsset{
          colormap={whiteviridis}{
            rgb(0cm)=(1,1,1)                          % 0  -> white
            rgb(1cm)=(0.267004,0.004874,0.329415)     % viridis dark purple
            rgb(2cm)=(0.229739,0.322361,0.545706)
            rgb(3cm)=(0.127568,0.566949,0.550556)
            rgb(4cm)=(0.369214,0.788888,0.382914)
            rgb(5cm)=(0.993248,0.906157,0.143936)     % viridis yellow
          }
    }

    \pgfplotsset{
          colormap={whiteblue}{
            rgb(0cm)=(1,1,1)
            rgb(1cm)=(0.85,0.92,1)  % very light blue buffer
            rgb(2cm)=(0.40,0.65,0.90)
            rgb(3cm)=(0.12,0.35,0.70)
            rgb(4cm)=(0.02,0.08,0.25) % deep navy
          }
        }
    \centering
    
    \begin{subfigure}[t]{0.42\textwidth} 
    \centering % Center the plot within its subfigure box
    \begin{tikzpicture}
          \begin{axis}[
                height=6cm,       % <--- EXACT HEIGHT APPLIED HERE
                axis on top,
                axis equal image, % <--- Width auto-adjusts to keep squares perfect
                xlabel={drug $g$ [-]},
                ylabel={tile $k$ [-]},
                enlargelimits={abs=0.5},
                point meta=explicit,
                y dir=reverse,
                yticklabels={50, 40, 30, 20, 10, 0},
                ytick={0, 10, 20, 30, 40, 50},
                colormap name=whiteblue,
                colorbar style={
                    ylabel=expected tile utilization [-]
                }
            ]
                \addplot [
                    matrix plot*,
                ] table [meta index=2] {experiments/packing_coocurrence_movers_10_dispensers_65.txt};
            \end{axis}
        \end{tikzpicture}
        \caption{Tile-drug matrix of drugs using $65$ dispensers.}
        \label{fig:drug-tile-d65}
    \end{subfigure}%
    \hfill
    \begin{subfigure}[t]{0.54\textwidth} 
    \centering % Center the plot within its subfigure box
    \begin{tikzpicture}
      \begin{axis}[
            height=6cm,           % <--- EXACT SAME HEIGHT APPLIED HERE
            axis on top,
            axis equal image,     % <--- Width auto-adjusts to keep squares perfect
            xlabel={drug $g$ [-]},
            enlargelimits={abs=0.5},
            point meta=explicit,
            y dir=reverse,
            yticklabels={50, 40, 30, 20, 10, 0},
            ytick={0, 10, 20, 30, 40, 50},
            colormap name=whiteblue,
            colorbar,
            colorbar style={
                ylabel=expected tile utilization $\mu_k$ [-]
            }
        ]
            \addplot [
                matrix plot*,
            ] table [meta index=2] {experiments/packing_coocurrence_movers_15_dispensers_100.txt};
            
        \end{axis}
    \end{tikzpicture}
    
    \caption{Tile-drug matrix of drugs using $100$ dispensers.}
    \label{fig:drug-tile-d100}
    \end{subfigure}
    
    \caption{Tile-drug matrix represents values of variables $y_{k,g}$ for different amounts of dispensers \numDispensers with $\numTiles=50$ tiles.}
    \label{fig:drug-tile-matrix}
\end{figure}
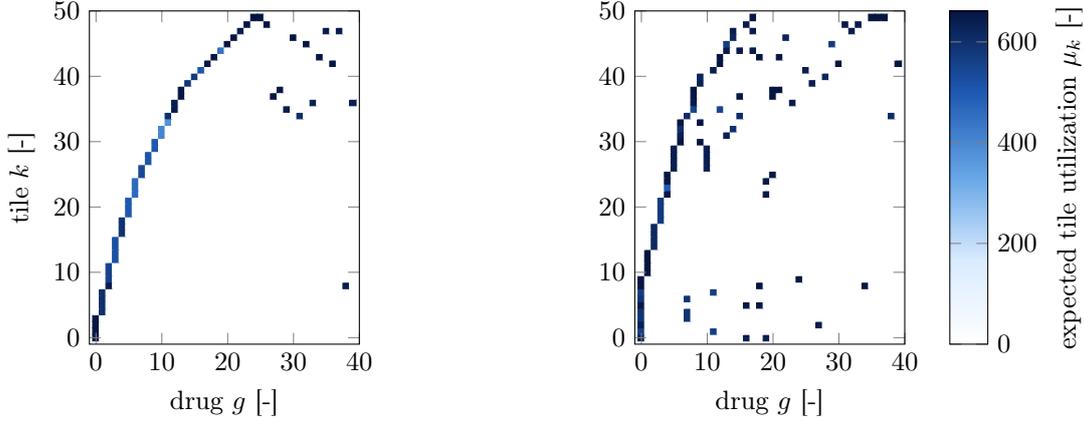

The model objective \eqref{eq:pack1-obj} aims to minimize the maximum expected tile load, as shown in the solution displayed in Figure~\ref{fig:tileutil} and also by the shades of blue in Figure~\ref{fig:drug-tile-matrix}.
Each bar in Figure~\ref{fig:tileutil-d65} corresponds to the expected tile utilization \tileLoad, which is calculated using equation~\eqref{eq:pack1-load}, which is expressed as a sum of contributions of assigned dispensers on the tile $k$.
The contribution of one dispenser \dispenserLoad for drug $g\in \drugs$ is calculated using the total demand \drugUtil for drug~$g$, divided by the number of dispensers \dispenserCount for drug~$g$.
This is implemented by a bilinear constraint~\eqref{eq:pack1-bilin}.
Hence, we assume that if \dispenserCount dispensers for drug $g$ will be used, then the overall demand of $g$ would be equally split among them, leading to their individual utilization $\drugUtil/\dispenserCount$.
This constraint expresses our assumption that, for high-quality placements, the operational-level schedule leads to balanced utilization of all tiles with drug $g$.

The rest of the constraints ensure that we can fit at most $d_\text{max}$ dispensers onto a tile by constraint~\eqref{eq:pack1-max4} and that the total amount of dispensers is not exceeded \eqref{eq:pack1-dmax}.
Furthermore, we included the constraint~\eqref{eq:pack1_pg} for practical reasons to limit the number of dispensers of drug $g$.
After initial experiments, we observed that when \numDispensers is large enough, then the model tends to generate many dispensers of the same drug to improve utilization balance, sometimes in such high numbers that they would be even larger than the movers available.

\begin{figure}[ht]
    \centering
    \begin{subfigure}[t]{0.48\textwidth}
    \centering
    \begin{tikzpicture}
        \begin{axis}[
            ybar stacked,
            height=0.45\textwidth,
            width=1\textwidth,
            xlabel={tile [-]},
            ylabel={utilization $\tileLoad$ [-]},
            grid=both,
            bar width=2pt,
            ymin=100,
            ymax=750,
            xmin=-1,
            xmax=50,
            legend pos={outer north east}, 
        ]

        \foreach \x in {0,1,2,3} {
            \addplot+ plot table [x=tile, y=utilization\x, col sep=comma] {experiments/packing_utilization_movers_10_dispensers_65.txt};
        }
            
        % \legend{1st dispenser, 2nd dispenser, 3rd dispenser}
        \end{axis}
    \end{tikzpicture}
    \caption{Expected tile utilization after packing using $\numDispensers=65$ dispensers, resulting in $\mu_{\max}=732$.}
    \label{fig:tileutil-d65}
    \end{subfigure}
    ~~~
    \begin{subfigure}[t]{0.48\textwidth}
    \centering
    \begin{tikzpicture}
        \begin{axis}[
            ybar stacked,
            height=0.45\textwidth,
            width=1\textwidth,
            % height=0.21\textwidth,
            % width=0.5\textwidth,
            xlabel={tile [-]},
            grid=both,
            bar width=2pt,
            ymin=100,
            ymax=750,
            xmin=-1,
            xmax=50,
            legend pos={outer north east}, 
            legend style={font=\footnotesize}
        ]
            \foreach \x in {0,1,2,3} {
                \addplot+ plot table [x=tile, y=utilization\x, col sep=comma] {experiments/packing_utilization_movers_15_dispensers_100.txt};
            }       
            \legend{1st dispenser, 2nd dispenser, 3rd dispenser, 4th dispenser}
        \end{axis}
    \end{tikzpicture}
    
    \caption{Expected tile utilization after packing using $\numDispensers=100$ dispensers, resulting $\mu_{\max}=661.5$.}
    \label{fig:tileutil-d100}
    \end{subfigure}
    \caption{Expected tile utilization with individual dispensers' contributions.}
    \label{fig:tileutil}
\end{figure}
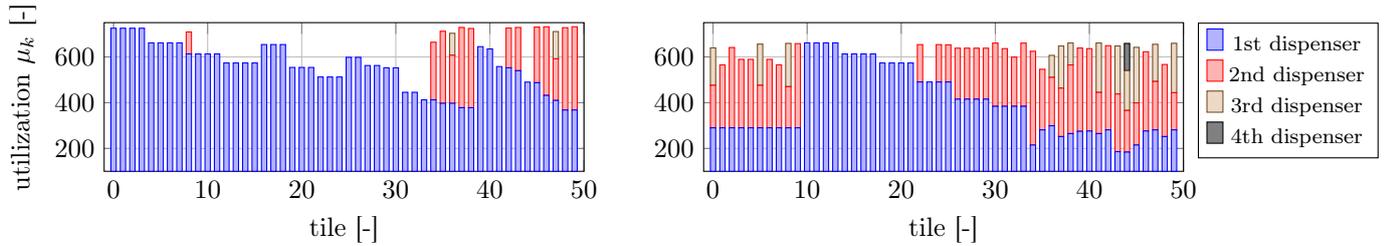

With the increasing number of dispensers \numDispensers available, we can allow more balanced utilization of the tiles.
For example, a solution to the packing problem with $\numDispensers=100$ is shown in Figure~\ref{fig:drug-tile-d100} with expected tile utilization plotted in Figure~\ref{fig:tileutil-d100}---we see that maximum utilization is decreased compared to Figure~\ref{fig:tileutil-d65}, as well as the tiles have a more balanced workload.

As the number of tiles $\numTiles=50$ remains the same in both cases, more dispensers are being packed onto a single tile, as visible in Figure~\ref{fig:drug-tile-d100}.
In such cases, it would be advantageous to assign dispensers to tiles together, not only to minimize the tile workload but also to account for mutual drug co-occurrence, thereby reducing mover movements when completing an order. 
We have formulated the maximization of pairwise correlations as another optimization step following the minimization of the maximal utilization.
Please see details of this model in the Appendix B.

% \FloatBarrier
\section{Placement problem} \label{sec:placement}
% \textbf{Given:}
% \begin{itemize}
%     \item type of topology (line/double line/ring/square)
%     \item number of tiles, interfaces, dispensers and medicines
% \end{itemize}

% \textbf{Goal:}
% \begin{itemize}
%     % \item place interfaces on the tiles
%     \item place dispensers that are already packed together on the remaining tiles
%     % \item balance connectivity (expected number of steps for a mover/patient) and congestion metric
% \end{itemize}

% \begin{figure}[ht]
%     \centering
%     \includegraphics[width=1.2\linewidth]{figs/correlation_1_1.png}
%     \caption{Correlation of placement and scheduler objective with 1:1 dosage ratio.}
%     \label{fig:enter-label}
% \end{figure}

%\subsection{Placement problem definition}
We assume that the input to the \textsf{Placement} problem consists of the union of the set of (unplaced) tiles $\packedDispensers$ with packed dispensers together with the set of interface tiles $\interfaces$ (used for extraction of cartridges from movers with completed orders and intersection of an empty cartridge for a new order), the given layout \layout, $|\layout|=\numTiles+\numInterfaces$ and a set of orders \ordersHistory that reflect an empirical data, e.g., based on past orders.
In this work, for simplicity, we assume four basic types of layouts: \textit{line} (e.g., $\mathcal{H}_\text{line 25} = \{(1,1),\ldots,(1,25)\}$), \textit{doubleline}, \textit{ring}, and \textit{square}, as shown in Figure~\ref{fig:topology-examples}.
We refer to specific layout sizes with the string $|\layout| \sim |\interfaces|$.
For example, the square $8\times8 \sim 4$ layout refers to a square with 8 rows, 8 columns, and 4 interfaces.

%$|\layout|\,(|\interfaces|)$.
%captioned with $\numTiles/\numInterfaces$ description.
%, although the presented algorithms are not particularly limited to 

 \begin{figure}[ht]
    \centering
    \begin{subfigure}[c]{0.33\textwidth}
      \centering
        \begin{tikzpicture}
      % Define parameters for the grid
      \def\rows{8}     % Number of rows
      \def\cols{8}     % Number of columns
      \def\size{0.3}   % Size of each square
      \def\margin{0.1} % Margin between squares
      
        % Loop to create the grid
      \foreach \x in {0,...,\numexpr\cols-1} {
        % Log column being processed
        \typeout{Processing column \x}
        \foreach \y in {0,...,\numexpr\rows-1} {
          % Log row being processed
          \typeout{  Processing row \y}
          % Draw each square
          \fill[gray!60] 
            ({\x * (\size + \margin)}, {\y * (\size + \margin)}) 
            rectangle ++(\size, \size);
        }
      }
    
      % interfaces
      \fill[darkgray!80] 
            ({3 * (\size + \margin)}, {5 * (\size + \margin)}) 
            rectangle ++(\size, \size);
      \fill[darkgray!80] 
            ({5 * (\size + \margin)}, {4 * (\size + \margin)}) 
            rectangle ++(\size, \size);
      \fill[darkgray!80] 
            ({2 * (\size + \margin)}, {3 * (\size + \margin)}) 
            rectangle ++(\size, \size);
      \fill[darkgray!80] 
            ({4 * (\size + \margin)}, {2 * (\size + \margin)}) 
            rectangle ++(\size, \size);
    \end{tikzpicture}
      \caption{Square $8\times8 \sim 4$ layout.}
      \label{fig:topology-examples:s8x8}
    \end{subfigure}%
    \begin{subfigure}[c]{0.33\textwidth}
      \centering
         \begin{tikzpicture}
      % Define parameters for the grid
      \def\rows{17}     % Number of rows
      \def\cols{17}     % Number of columns
      \def\size{0.14}   % Size of each square
      \def\margin{0.05} % Margin between squares
      
        % Loop to create the grid
      \foreach \x in {0,...,\numexpr\cols-1} {
          % Draw each square
          \fill[gray!60] 
            ({\x * (\size + \margin)}, {0 * (\size + \margin)}) 
            rectangle ++(\size, \size);
        \fill[gray!60] 
            ({\x * (\size + \margin)}, {16 * (\size + \margin)}) 
            rectangle ++(\size, \size);
        \fill[gray!60] 
            ({0 * (\size + \margin)}, {\x * (\size + \margin)}) 
            rectangle ++(\size, \size);
        \fill[gray!60] 
            ({16 * (\size + \margin)}, {\x * (\size + \margin)}) 
            rectangle ++(\size, \size);
        }

      % interfaces
      \fill[darkgray!80] 
            ({0 * (\size + \margin)}, {8 * (\size + \margin)}) 
            rectangle ++(\size, \size);
      \fill[darkgray!80] 
            ({1 * (\size + \margin)}, {16 * (\size + \margin)}) 
            rectangle ++(\size, \size);
    \fill[darkgray!80] 
            ({9 * (\size + \margin)}, {16 * (\size + \margin)}) 
            rectangle ++(\size, \size);
      \fill[darkgray!80] 
            ({16 * (\size + \margin)}, {12 * (\size + \margin)}) 
            rectangle ++(\size, \size);
    \end{tikzpicture}
      \caption{Ring $64 \sim 4$ layout.}
      \label{fig:topology-examples:r68}
    \end{subfigure}%
    \begin{minipage}[c]{0.33\textwidth}
      \begin{subfigure}[t]{\linewidth}
        \centering
          \begin{tikzpicture}
      % Define parameters for the grid
      \def\rows{1}     % Number of rows
      \def\cols{25}     % Number of columns
      \def\size{0.15}   % Size of each square
      \def\margin{0.04} % Margin between squares
      
        % Loop to create the grid
      \foreach \x in {0,...,\numexpr\cols-1} {
        % Log column being processed
        \typeout{Processing column \x}
        \foreach \y in {0,...,\numexpr\rows-1} {
          % Log row being processed
          \typeout{  Processing row \y}
          % Draw each square
          \fill[gray!60] 
            ({\x * (\size + \margin)}, {\y * (\size + \margin)}) 
            rectangle ++(\size, \size);
        }
      }
    
      % interfaces
      \fill[darkgray!80] 
            ({12 * (\size + \margin)}, {0 * (\size + \margin)}) 
            rectangle ++(\size, \size);
      % \fill[darkgray!80] 
      %       ({7 * (\size + \margin)}, {1 * (\size + \margin)}) 
      %       rectangle ++(\size, \size);
      
    \end{tikzpicture}
        \caption{Line $25 \sim 1$ layout.}
        \label{fig:topology-examples:l25}
      \end{subfigure}

      \vspace{2em}

      \begin{subfigure}[t]{\linewidth}
        \centering
          \begin{tikzpicture}
      % Define parameters for the grid
      \def\rows{2}     % Number of rows
      \def\cols{25}     % Number of columns
      \def\size{0.15}   % Size of each square
      \def\margin{0.04} % Margin between squares
      
        % Loop to create the grid
      \foreach \x in {0,...,\numexpr\cols-1} {
        % Log column being processed
        \typeout{Processing column \x}
        \foreach \y in {0,...,\numexpr\rows-1} {
          % Log row being processed
          \typeout{  Processing row \y}
          % Draw each square
          \fill[gray!60] 
            ({\x * (\size + \margin)}, {\y * (\size + \margin)}) 
            rectangle ++(\size, \size);
        }
      }
    
      % interfaces
      \fill[darkgray!80] 
            ({15 * (\size + \margin)}, {0 * (\size + \margin)}) 
            rectangle ++(\size, \size);
      \fill[darkgray!80] 
            ({7 * (\size + \margin)}, {1 * (\size + \margin)}) 
            rectangle ++(\size, \size);
      
    \end{tikzpicture}
        \caption{Doubleline $2\times25 \sim 2$ layout.}
        \label{fig:topology-examples:d50}
      \end{subfigure}
    \end{minipage}
    \caption{Examples of considered layout topologies. Highlighted tiles display an example of a possible placement of interface locations $\interfaces$, which is together with packed tiles $\packedDispensers$, subject to optimization.}
    \label{fig:topology-examples}
\end{figure}
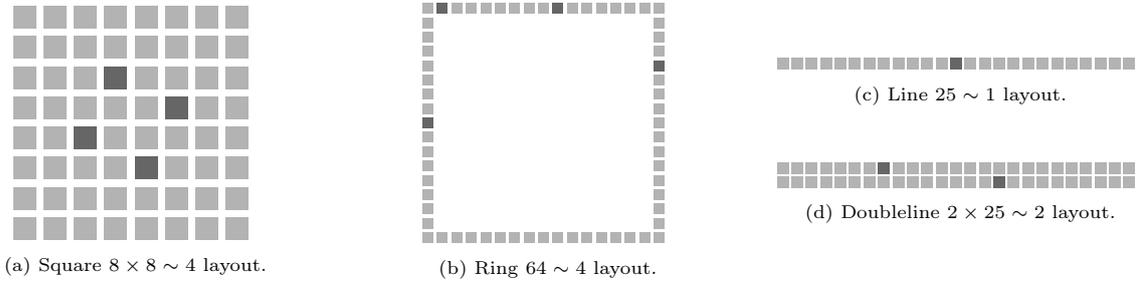

The goal of \textsf{Placement} is to find a mapping \placementExt \xspace between unplaced tiles with packed dispensers~$\packedDispensers$ (i.e., output of the \textsf{Packing} problem) and interfaces~$\interfaces$ to the layout \layout, such that the mean number of movers' steps in Manhattan distance needed to fulfill a set of historical orders \ordersHistory is minimized:
\begin{align}
     &\min_{\placementExt} \text{placement}(\placement) = \min_{\placementExt} \frac{1}{|\ordersHistory|} \sum_{p\in\ordersHistory} \min_{\permElement \in \interfaces \times \Pi(p) \times \interfaces} \sum^{|\permElement|-1}_{j=1} \tilesDistance{j}, \label{eq:placement_problem}
\end{align}
where \interfaces is the set of all interfaces and $\Pi(p)$ is the set of all permutations of dispensers \textit{relevant} to order $p$ defined as follows.
The set of dispensers providing drug $g$ is denoted as  $\packedDispensers_g=\left\{(g,k) \mid \forall (g,k) \in \packedDispensers\right\}$, i.e., the set of alternative dispensers of drug $g$.
Then, $\Pi(p)=\text{Perm}\left(\bigcup_{(g,\Delta^g_p)\in p} \packedDispensers_g\right)$ be the set of all permutations of $\packedDispensers_g$ for all drugs $g$ requested in the customer order $p$.
Finally, $\forall \permElement\in\Pi(p): \permElement{j}$ denotes $j$-th element of permutation \permElement of dispensers relevant to $p$. 

% \tn{vysvetlit strukturu: (master) assignment problem $\to$ (inner) shortest ham path with neighborhoods. koncept alternative dispensers $\approx$ neighborhood}

The \textsf{Placement} problem, as defined by \eqref{eq:placement_problem}, consists of two levels: (i)~an assignment problem that assigns packed dispensers and interfaces $\packedDispensers\cup\interfaces$ to the specific tiles placed in the given layout $\layout$, and (ii)~sum of the inner problems, where each resembles the Shortest Hamiltonian path problem with neighborhoods (SHPPN)~\citep{arkin1994approximation,deckerova2024combinatorial}. 
Each customer order $p\in\ordersHistory$ represents a single path in a graph defined by assignment $\varphi$, starting and finishing at some interface location $i\in\interfaces$.
The path must contain the required dispensers, each of which may exist in multiple alternatives, forming a so-called \textit{neighborhood}.
Due to the considered layouts and the permitted movements, the costs of the underlying graph resemble the $\ell_1$ metric on an integer lattice, i.e., Manhattan distance.

See an example in Figure~\ref{fig:placement_example}, where the solution $\varphi$ of the assignment part of the \textsf{Placement} problem is displayed.
To calculate the quality of placement $\varphi$, we have a set of three historical orders \ordersHistory given in the \textit{drugs} column of Figure~\ref{fig:example_placement_data}.
With these, the optimal solution of the corresponding three inner problems in \eqref{eq:placement_problem} is given in \textit{permutation $\sigma$} column in Figure~\ref{fig:example_placement_data}.
These correspond to the optimal sequence of dispensers together with a starting and finishing interface for each order in \ordersHistory.
Finally, the objective value of placement $\varphi$ is calculated as the arithmetic average of the mover steps in each permutation $\sigma$.
Hence, with the increasing size $|\ordersHistory|\to\infty$, the objective function approaches the expected number of steps needed to fulfill an order.

The objective function of this bi-level problem is designed to find a placement \placement such that dispensers with drugs requested in typical individual orders are placed near each other, thereby reducing the average travel time of movers.
Later in the experiments in Section~\ref{sec:experiments-correlation}, we will show that the choice of the objective function \eqref{eq:placement_problem} is justified by its correlation with the actual schedule makespan observed at the operational level.

% On the other hand, we show that this objective function leads to an \textsf{NP}-hard problem, even for the simplest layout.

\begin{figure}[ht]
    \centering
    \begin{subfigure}[t]{0.9\textwidth}
      \centering
        \begin{tikzpicture}[font=\footnotesize]
  % ===== Grid params =====
  \def\rows{4}     % Number of rows
  \def\cols{4}     % Number of columns
  \def\size{0.6}   % Size of each square
  \def\margin{0.15}% Margin between squares

  % ===== Color palette (distinct, tweak as needed) =====
\definecolor{cOMEPRAZOLE}{HTML}{1F77B4}
\definecolor{cLEVOTHYROXINE}{HTML}{59260b}
\definecolor{cLOVASTATIN}{HTML}{9acd32}
\definecolor{cVALSARTAN}{HTML}{D62728}
\definecolor{cMETFORMIN}{HTML}{9467BD}
\definecolor{cGLIPIZIDE}{HTML}{8C564B}
\definecolor{cPRAVASTATIN}{HTML}{E377C2}
\definecolor{cLISINOPRIL}{HTML}{FF7F0E}
\definecolor{cSIMVASTATIN}{HTML}{31A354}
\definecolor{cMETOPROLOL}{HTML}{17BECF}
\definecolor{cCLOPIDOGREL}{HTML}{3182BD}
\definecolor{cHYDROCHLOROTHIAZIDE}{HTML}{17BECF}
\definecolor{cLOSARTAN}{HTML}{393B79}
\definecolor{cAMLODIPINE}{HTML}{E6550D}
\definecolor{cATORVASTATIN}{HTML}{BCBD22}
\definecolor{cWARFARIN}{HTML}{636363}
\definecolor{cATENOLOL}{HTML}{9C9EDE}
\definecolor{cFUROSEMIDE}{HTML}{9ECAE1}
\definecolor{cINTERFACE}{HTML}{4DDDAD}

  % Drug -> color name lookup
  \pgfkeys{/drugcolor/.is family,/drugcolor,
    OMEPRAZOLE/.initial=cOMEPRAZOLE,
    LEVOTHYROXINE/.initial=cLEVOTHYROXINE,
    LOVASTATIN/.initial=cLOVASTATIN,
    VALSARTAN/.initial=cVALSARTAN,
    METFORMIN/.initial=cMETFORMIN,
    GLIPIZIDE/.initial=cGLIPIZIDE,
    PRAVASTATIN/.initial=cPRAVASTATIN,
    LISINOPRIL/.initial=cLISINOPRIL,
    SIMVASTATIN/.initial=cSIMVASTATIN,
    METOPROLOL/.initial=cMETOPROLOL,
    CLOPIDOGREL/.initial=cCLOPIDOGREL,
    HYDROCHLOROTHIAZIDE/.initial=cHYDROCHLOROTHIAZIDE,
    LOSARTAN/.initial=cLOSARTAN,
    AMLODIPINE/.initial=cAMLODIPINE,
    ATORVASTATIN/.initial=cATORVASTATIN,
    WARFARIN/.initial=cWARFARIN,
    ATENOLOL/.initial=cATENOLOL,
    FUROSEMIDE/.initial=cFUROSEMIDE,
    INTERFACE/.initial=cINTERFACE
  }
  \newcommand{\drugclr}[1]{\pgfkeysvalueof{/drugcolor/#1}}

  % ===== Coordinate -> drug(s) mapping (from your JSON)
  % (INTERFACE in uppercase to match the color key) =====
  \expandafter\def\csname cell@0x0\endcsname{OMEPRAZOLE}
  \expandafter\def\csname cell@0x1\endcsname{LEVOTHYROXINE}
  \expandafter\def\csname cell@0x2\endcsname{LOVASTATIN}
  \expandafter\def\csname cell@0x3\endcsname{VALSARTAN}
  \expandafter\def\csname cell@1x0\endcsname{METFORMIN,GLIPIZIDE,PRAVASTATIN}
  \expandafter\def\csname cell@1x1\endcsname{LISINOPRIL}
  \expandafter\def\csname cell@1x2\endcsname{SIMVASTATIN}
  \expandafter\def\csname cell@1x3\endcsname{INTERFACE}
  \expandafter\def\csname cell@2x0\endcsname{METOPROLOL,CLOPIDOGREL}
  \expandafter\def\csname cell@2x1\endcsname{INTERFACE}
  \expandafter\def\csname cell@2x2\endcsname{HYDROCHLOROTHIAZIDE}
  \expandafter\def\csname cell@2x3\endcsname{LOSARTAN,AMLODIPINE,ATORVASTATIN}
  \expandafter\def\csname cell@3x0\endcsname{WARFARIN}
  \expandafter\def\csname cell@3x1\endcsname{ATORVASTATIN}
  \expandafter\def\csname cell@3x2\endcsname{ATENOLOL}
  \expandafter\def\csname cell@3x3\endcsname{FUROSEMIDE}

% ===== Drawer for one square (90° clockwise rotation) =====
\newcommand{\drawCell}[2]{%
  % map (x,y) -> (cols-1 - y, x)
  \pgfmathsetmacro{\yy}{(\cols-1-#2)*(\size+\margin)}
  \pgfmathsetmacro{\xx}{#1*(\size+\margin)}

  % default background in case of gaps:
  \fill[gray!20] (\xx,\yy) rectangle ++(\size,\size);

  % read content
  \let\content\relax
  \expandafter\let\expandafter\content\csname cell@#1x#2\endcsname
  \ifx\content\relax
    \fill[gray!60] (\xx,\yy) rectangle ++(\size,\size);
  \else
    \edef\list{\content}
    % count entries
    \pgfmathtruncatemacro{\N}{0}
    \foreach \d [count=\i] in \list {\xdef\N{\i}}
    \pgfmathsetmacro{\w}{\size/\N}
    % draw vertical slices
    \foreach \d [count=\i] in \list {%
      \pgfmathsetmacro{\xstart}{\xx+(\i-1)*\w}
      \fill[\drugclr{\d}] (\xstart,\yy) rectangle ++(\w,\size);
    }
    % subtle border
    \draw[black, line width=0.15pt] (\xx,\yy) rectangle ++(\size,\size);
  \fi
}

  % ===== Render the grid =====
  \foreach \x in {0,...,\numexpr\cols-1} {
    \foreach \y in {0,...,\numexpr\rows-1} {
      \drawCell{\x}{\y}
    }
  }

  % ===== Axes: tick labels and titles (fixed with braces) =====
\pgfmathsetmacro{\gridW}{\cols*(\size+\margin)-\margin}
\pgfmathsetmacro{\gridH}{\rows*(\size+\margin)-\margin}
\def\axispad{0.18}

% x ticks (labels start at 1)
\foreach \ix in {0,...,\numexpr\cols-1} {
  \pgfmathtruncatemacro{\lab}{\ix+1}
  \node[below] at ({\ix*(\size+\margin)+0.5*\size}, {-\axispad}) {\lab};
}

% y ticks (labels start at 1)
\foreach \iy in {0,...,\numexpr\rows-1} {
  \pgfmathtruncatemacro{\lab}{\iy+1}
  \node[left] at ({-\axispad}, {\iy*(\size+\margin)+0.5*\size}) {\lab};
}

% axis titles
\node[below] at ({\gridW/2}, {-0.6}) {$x$ coordinate};
\node[rotate=90] at ({-0.7}, {\gridH/2}) {$y$ coordinate};

  % ===== Legend (multi-column; 7 rows per column as in your snippet) =====
  \pgfmathsetmacro{\legendx}{\cols*(\size+\margin)+0.35}
  \pgfmathsetmacro{\legendtop}{(\rows-1)*(\size+\margin)+\size}

  \def\rowspercol{7}                              % rows in each legend column
  \pgfmathsetmacro{\legendRowStep}{\size*0.7}     % vertical spacing between rows
  \pgfmathsetmacro{\legendColStep}{\size*5.8}     % horizontal spacing between columns
  \pgfmathsetmacro{\sw}{\size*0.5}                % color swatch size

  \def\legendList{INTERFACE,OMEPRAZOLE,LEVOTHYROXINE,LOVASTATIN,VALSARTAN,%
    METFORMIN,GLIPIZIDE,PRAVASTATIN,LISINOPRIL,SIMVASTATIN,%
    METOPROLOL,CLOPIDOGREL,LOSARTAN,AMLODIPINE,%
    ATORVASTATIN,WARFARIN,ATENOLOL,FUROSEMIDE,HYDROCHLOROTHIAZIDE}

  \foreach \lab [count=\i] in \legendList {
    \pgfmathtruncatemacro{\col}{(\i-1)/\rowspercol}
    \pgfmathtruncatemacro{\row}{mod(\i-1,\rowspercol)}
    \pgfmathsetmacro{\x}{\legendx + \col*\legendColStep}
    \pgfmathsetmacro{\y}{\legendtop - \row*\legendRowStep}

    \fill[\drugclr{\lab}] (\x,\y-\sw*1.2) rectangle ++(\sw,\sw);
    \node[anchor=west] at (\x+\sw+0.15,\y-\sw*0.6) {\textsc{\lab}};
  }
\end{tikzpicture}
        \vspace{-1.5em}
      \caption{Example placement $\varphi$ in square layout $4\times4 \sim 2$.}
      \label{fig:placement_example}
    \end{subfigure}
    
    \vspace{1em}
    
    \begin{subfigure}[t]{1\textwidth}
      \centering
      \resizebox{1\textwidth}{!}{%
        \begin{tabular}{c|c|c|c}
        \toprule
         order $p$  & drugs & permutation $\sigma$ & steps [-] \\
        \midrule
          $\#1$   &  $\{\textsc{atorvastatin}, \textsc{hydrochlorothiazide}\}$ & $(2,1)\xrightarrow{}(3,1)\xrightarrow{}(3,2)\xrightarrow{}(3,3)$ & 3 \\
          $\#2$   &  $\{\textsc{omeprazole}\}$ & $(3,3)\xrightarrow{}(2,3)\xrightarrow{}(1,3)\xrightarrow{}(1,4)\xrightarrow{}(2,4)\xrightarrow{}(3,4)\xrightarrow{}(3,3)$ & 6\\
          $\#3$    &  $\{\textsc{lisinopril}, \textsc{simvastatin}\}$& $(3,3)\xrightarrow{}(2,3)\xrightarrow{}(2,2)\xrightarrow{}(2,1)$& 3\\
       \midrule
    \multicolumn{3}{c}{$\text{placement}(\varphi) = (3 + 6 + 3)/3=4$}\\
    \bottomrule  
        \end{tabular}%
        }
        \caption{Example $\ordersHistory$ data and their optimal permutations $\sigma$.}
        \label{fig:example_placement_data}
    \end{subfigure}
    \caption{Example of the considered placement in the example square layout $4\times 4 \sim 2$.}
    \label{fig:placement-example}
\end{figure}

\subsection{Complexity of the Placement problem}
In this section, we hint at two sources of its complexity.
The first one comes from the assignment part of the problem, agnostic to the SHPPN arising in the inner problem.
We show that the decision version of \textsf{Placement} is \textsf{NP}-hard even for the \textit{line} layout.

\begin{proposition}[Complexity of \textsf{Placement}]
   Deciding whether \textsf{Placement} has a solution with the objective value less than or equal to $q$ is \textsf{NP}-hard. 
\end{proposition}
    
The reduction will be shown from the Exact 3-Cover problem:
\begin{definition}[\textsf{X3C}]
Exact 3-Cover:
    \begin{itemize}
        \item Given: Set $U=\{1,2,\ldots, 3q\}$, family of sets $F=\{S_1,S_2, \ldots, S_k\}$, $|S_i|=3$, $S_i\subseteq U$.
        \item Goal: Does exist an exact cover of $U$, i.e., exists $F^\prime\subseteq F$, $|F^\prime|=q$ such that $S_i,S_j\in F^\prime$, $i\neq j$: $S_i \cap S_j = \emptyset$ and $\bigcup_{S_i\in F^\prime} S_i = U$.
    \end{itemize}
\end{definition}

\textsf{X3C} is known to be \textsf{NP}-complete, even if each $i \in U$ is contained in at most three sets $S_i$~\citep{garey1979computers}.
% \tn{need to slightly change wording as we are not packing dispensers to tiles by packing but drugs to dispenseres, dispensers are then placed}

\begin{proof}
We construct the following instance of \textsf{Placement} problem.
Consider a line layout with $k+2$ tiles containing 2 interfaces.
The set of all drugs corresponds to the set $U$. 
The packing of dispensers is given by each $S_i\in F$ set, i.e., we have $3k$ dispensers packed into a total of $k$ (unplaced) tiles.
Finally, we have a single order in \ordersHistory that requires all drugs $U$.
We claim that \textsf{X3C} has a solution if and only if there is a placement $\placement$ with the cost equal to $q$ mover steps.

Suppose we have solution $F^\prime=\{S^\prime_1, \ldots, S^\prime_q\}$ to \textsf{X3C}.
We construct a solution to the placement problem in the following way.
Place the first interface to tile $(1,1)$ and the second interface to tile $(1,q+2)$.
The $q$ tiles that pack a set of dispensers defined by the solution of \textsf{X3C} $S^\prime_1, \ldots, S^\prime_q$ are placed between the interfaces in an arbitrary order.
The remaining $k-q$ tiles are placed in arbitrary order on tiles $(2,q+2), (1,q+3),\ldots, (1,k+2)$.
See visualization in Figure~\ref{fig:reduction_a}.

\begin{figure}[ht]
    \centering
    \begin{tikzpicture}
      % Define parameters for the grid
      \def\rows{1}     % Number of rows
      \def\cols{29}     % Number of columns
      \def\size{0.2}   % Size of each square
      \def\margin{0.06} % Margin between squares
      
        % Loop to create the grid
      \foreach \x in {0,...,\numexpr\cols-1} {
        % Log column being processed
        \typeout{Processing column \x}
        \foreach \y in {0,...,\numexpr\rows-1} {
          % Log row being processed
          \typeout{  Processing row \y}
          % Draw each square
          \fill[gray!60] 
            ({\x * (\size + \margin)}, {\y * (\size + \margin)}) 
            rectangle ++(\size, \size);
        }
      }
    
      % interfaces
      \fill[darkgray!80] 
            ({11 * (\size + \margin)}, {0 * (\size + \margin)}) 
            rectangle ++(\size, \size);
      \fill[darkgray!80] 
            ({0 * (\size + \margin)}, {0 * (\size + \margin)}) 
            rectangle ++(\size, \size);

      \draw[decoration={brace,amplitude=2mm,mirror,raise=1mm},decorate] (0.26,0) -- (2.8,0);
      \node[] at (1.6, -0.6) {\scriptsize{$S^\prime_1, \ldots, S^\prime_q$}};
    \end{tikzpicture}
    \caption{Placement solution corresponding to a feasible instance of \textsf{X3C}.}
    \label{fig:reduction_a}
\end{figure}
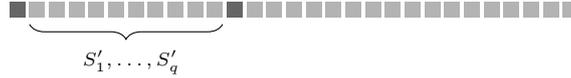
It can be seen that such a solution has cost $q$ and satisfies the customer order requirement, since a mover fulfilling the order can traverse from the left to the right interface, dispensing all 3 drugs on each visited tile.
Further, note that no better solution exists as the order requires $|U|=3q$ drugs, and each tile is packed with exactly 3 dispensers. Therefore, the mover has to visit at least $q$ tiles.

Next, consider a feasible solution to the placement problem with a cost equal to $q$, shown in Figure~\ref{fig:reduction_b}.
Let $a$ and $b$, $a<b$ be indices of tiles where interface locations are placed.
Further, let $S^\prime_{a+1}, \ldots, S^\prime_{b-1}$ be the sets defined by drugs that are dispensed at tiles $(1,a+1), \ldots, (1,b-1)$.
Since the order requires $3q=|U|$ drugs and each tile (except an interface) contains 3 dispensers, then no such solution can start and finish at the same interface location.
Therefore, in the optimal placement, there are $b-a-1=q$ tiles between the interfaces.
Since any feasible solution of placement with cost $q$ for a single order must ensure that all its $3q$ drugs can be dispensed in $q$ steps of the mover, it follows $\forall i,j\in \{a+1,\ldots, b-1\}, i<j: S^\prime_i \cap S^\prime_j = \emptyset$, hence $F^\prime=\{S^\prime_{a+1}, \ldots, S^\prime_{b-1}\}$ is an exact cover of $U$.   

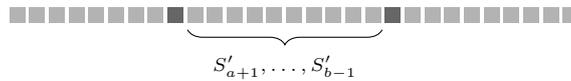
\begin{figure}[ht]
    \centering
    \begin{tikzpicture}
      % Define parameters for the grid
      \def\rows{1}     % Number of rows
      \def\cols{29}     % Number of columns
      \def\size{0.2}   % Size of each square
      \def\margin{0.06} % Margin between squares
      
        % Loop to create the grid
      \foreach \x in {0,...,\numexpr\cols-1} {
        % Log column being processed
        \typeout{Processing column \x}
        \foreach \y in {0,...,\numexpr\rows-1} {
          % Log row being processed
          \typeout{  Processing row \y}
          % Draw each square
          \fill[gray!60] 
            ({\x * (\size + \margin)}, {\y * (\size + \margin)}) 
            rectangle ++(\size, \size);
        }
      }
    
      % interfaces
      \fill[darkgray!80] 
            ({8 * (\size + \margin)}, {0 * (\size + \margin)}) 
            rectangle ++(\size, \size);
      \fill[darkgray!80] 
            ({19 * (\size + \margin)}, {0 * (\size + \margin)}) 
            rectangle ++(\size, \size);

      \draw[decoration={brace,amplitude=2mm,mirror,raise=1mm},decorate] (2.34,0) -- (4.9,0);
      \node[] at (3.62, -0.6) {\scriptsize{$S^\prime_{a+1}, \ldots, S^\prime_{b-1}$}};
    \end{tikzpicture}
    \caption{A solution of placement problem where $q$ steps of a mover fulfills the order drug requirements $U$.}
    \label{fig:reduction_b}
\end{figure}
\end{proof}
Note that when we allow the packing of at most 2 dispensers onto the same tile, the above reduction would no longer hold, as the Exact 2-Cover (covering by sets with exactly two elements) is polynomially solvable---can be reduced to 2-SAT, which is polynomially solvable~\citep{garey1979computers}.
Thus, this leaves the possibility of having lower complexity on simple topologies such as \textit{line} or \textit{ring} if the maximum number of dispensers on a tile is limited to 2.
However, as indicated above, the problem of placing dispensers and interface locations is not the only source of hardness.
Without providing a formal proof, we suspect the problem may still be hard due to the inner SHPPN problem.
Indeed, it is known that the standard Traveling Salesman Problem in the plane equipped with $\ell_1$ metric still remains strongly \textsf{NP}-hard, see \citep{itai1982hamilton} or [ND23] in \citep{garey1979computers}.

\subsection{Genetic Algorithm with Inner SHPPN Problem}
To solve the \textsf{Placement} problem represented by equation~\eqref{eq:placement_problem}, we proposed the following algorithm.
Due to the complexity of the inner problem, we proposed a meta-heuristic algorithm that assigns packed tiles \packedDispensers to the layout \layout, where the placement quality of \placement is approximated as follows.
For each customer order $p$ in \ordersHistory, we solve the inner optimization problem in the form of
\begin{equation}
    \eqnleft{[SHPPN]} \min_{\permElement \in \interfaces \times \Pi(p) \times \interfaces} \sum^{|\permElement|-1}_{j=1} \tilesDistance{j}, \label{eq:placement_inner}
\end{equation}
%\todo{tady to makro neni nejak dobre, micha se $i$ a $j$}
which involves finding the shortest path from an interface location over a set of selected vertices representing relevant dispenser locations for order $p$, finishing in some interface location as well.
The underlying graph resembles a lattice equipped with $\ell_1$ metric.
Hence, the problem effectively reduces to finding a sequence \permElement of dispensers that the mover visits, as the path between locations $\placement(\permElement_j)$ and $\placement(\permElement_{j+1})$ always exists with length exactly $\tilesDistance{j}$.
For example, an optimal solution to these inner problems is shown in Figure~\ref{fig:example_placement_data} for the placement from Figure~\ref{fig:placement_example}, for three historical customer orders.

If one would not solve SHPPN problem exactly, one of the natural ways of solving \eqref{eq:placement_inner} would be to find $\permElement$ greedily---start at some interface location $i\in\interfaces$ and consecutively proceed from $\placement(\sigma_j)$ to the following location $\placement(\sigma_{j+1})$ with the lowest distance $\tilesDistance{j}$ and so on.
However, this obviously may lead to suboptimal results, and at the operational level, it may not always be possible, for example, because the nearest required dispenser may be occupied by another mover.

Therefore, at the operational level, we would sometimes have to assign a more distant dispenser, i.e., breaking the greedy ordering.
Hence, instead of a fully greedy method, we construct multiple $\permElement$, where each sequence of dispensers is also constructed in a greedy way, but the next dispenser in the sequence is not taken greedily, but is sampled inversely proportionally with respect to their distance from the current location $\placement(\sigma_j)$.
Thus, closer alternative dispensers are more likely to be included in the calculation of \eqref{eq:placement_inner} (reflecting what an efficient scheduler at the operational level would do), while not ruling out the occasional inclusion of more distant ones.

Consider, for example, the optimal permutation for order $\#1$ in Figure~\ref{fig:example_placement_data}.
An alternative permutation $\sigma^\prime$ would be $(2,1)\xrightarrow{}(2,2)\xrightarrow{}(2,3)\xrightarrow{}(3,3)\xrightarrow{}(4,3)\xrightarrow{}(4,2)\xrightarrow{}(3,2)\xrightarrow{}(3,3)$ with length 7.
However, there might be other permutations, for example, starting at a different interface location, such as $(3,3)\xrightarrow{}(4,3)\xrightarrow{}(4,2)\xrightarrow{}(3,2)\xrightarrow{}(3,3)$ consisting of 4 just steps.
Each such attempt to construct a permutation $\permElement$ for a single customer order is called an \textit{episode}.

For each customer order $p$ in historical data \ordersHistory, we therefore run a specified number of episodes (simulations) that may yield different feasible solutions to the inner problem~\eqref{eq:placement_inner}.
Objective values from solutions across all episodes of each historical order $p\in\ordersHistory$ are summed and normalized by the number of episodes.
In this way, the permutations $\permElement$ with lower cost have a higher probability of being involved in the calculation of placement quality, while the contribution of less probable but more expensive ones serves as a regularization to derive more robust placement, reflecting the occasional need for "suboptimal" routes at the operational level.

The pseudocode for the evaluation of a single placement solution \placement is given in Algorithm~\ref{alg:placer}.
In computations where the following location to move to (line 9) is sampled, if $\delta_j=0$, then we use the convention that $1/\delta_j = 1$. 
This essentially states that if on some location we can dispense more than one drug for the considered order $p$, then all of these drugs are dispensed there (in an arbitrary order), taking advantage of the packing of dispensers that reflects a positive correlation between drug occurrences (see Appendix B). 

\begin{algorithm}[ht]
\caption{Approximate solution to inner SHPP problem: fitness computation.}\label{alg:placer}
\begin{algorithmic}[1]
\Require layout \layout, placement \placementExt, number of \#\textit{episodes}.
\Ensure $\text{placement}(\placement)$.
\State $total \gets 0$
\ForAll{$p\in\ordersHistory$}\Comment{For each order in historical data.}
\ForAll{$e\in \{1,\ldots, \#\textit{episodes}\}$}\Comment{Simulate episodes.}
    \State sample $i^\star \sim \interfaces$ \text{ uniformly} \Comment{Sample uniformly an initial interface tile.}
    \State $R\gets \{g \, | \, \forall (g,\Delta^g_p) \in p\}$  \Comment{Drugs requested for order $p$.}
    \Repeat
        \State $L \gets \bigcup_{g \in R} \packedDispensers^{-1}(g)$ \Comment{Set of all dispensers supplying remaining $g\in R$ drugs.}
        \State compute distance $\delta_j=\| \placement(i^\star) - \placement(j) \|_1$, $\forall j\in L$
        \State sample $j^\star\sim L$ proportionally to $1/\delta_j$
        \State $total \gets total + \delta_{j^\star}$ \Comment{Mover travels from location $\placement(i^\star)$ to $\placement(j^\star)$.}
        \State $R \gets R \setminus \{\dispensersOnTile{j^\star}\}$ \Comment{Remove drug that was dispensed at tile $j^\star$.}
        \State $i^\star \gets j^\star$
    \Until{$R = \emptyset$} \Comment{Until all drugs have been dispensed.}
    \State compute distance $\delta^\prime_j=\| \placement(i^\star) - \placement(j) \|_1$, $\forall j\in \interfaces$ \Comment{Distance from $\placement(i^\star)$ to all interface locations.}
    \State sample $j^\star \sim \interfaces$ proportionally to $1/\delta^\prime_j$
    \State $total \gets total +\delta^\prime_{j^\star}$ \Comment{Go to a finishing interface tile.}
\EndFor
\EndFor
\State \Return $\textit{total}/\left(|\ordersHistory|\cdot \#\textit{episodes}\right)$ \Comment{Return the expected number of mover steps in placement \placement.}
\end{algorithmic}
\end{algorithm}

This algorithm is used for the computation of the fitness function that steers the metaheuristic algorithm towards placement \placement that minimizes \eqref{eq:placement_problem}.  %implemented in \texttt{Pymoo} library 
For this, we use a standard genetic algorithm (GA) that represents placements \placementExt{} as permutations that map (indexed) packed tiles $\packedDispensers$ to specific positions in the layout \layout.
The algorithm is implemented in the \texttt{Pymoo} library~\citep{Blank_2020} and uses \textit{order crossover} with \textit{inversion mutation} operators.
% Due to the way the objective function is computed (i.e., Algorithm~\ref{alg:placer}), which essentially "simulates" how the movers would fulfill the prescribed sequence of orders \ordersHistory on the given placement \placement, we call this placement algorithm \textit{Simulation placer}.

Furthermore, we would like to highlight the fact that the Algorithm~\ref{alg:placer} relaxes on the dispensing time; thus, it acts as if it spent zero time at each dispensing site.
Hence, congestion at the dispensers is not taken into account by the placement---all orders are essentially being completed simultaneously.
However, this is handled indirectly by sampling next dispensers in the sequence $\sigma$ to visit proportionally to the inverse distance from the current location.
This aims to reflect that the nearest compatible dispenser may not always be available.

Finally, we also tested a different placement algorithm, known as the \textit{Analytical placer}, which solved the inner optimization problem~\eqref{eq:placement_inner} optimally by enumerating all permutations \permElement.
This is computationally tractable if the number of alternative dispensers \dispenserCount for all drugs $g\in\drugs$ is not too high, as the typical order consists of approximately 4 drugs.
However, with more alternatives in dispensers (larger neighborhoods), this quickly leads to prohibitive computational time.
What is perhaps more important, however, is the fact that the placement objective \eqref{eq:placement_problem} is sensitive to the \textit{best-case traversal} through the given layout due to the inner problem~\eqref{eq:placement_inner}.
This means that, e.g., the second-shortest routes are not optimized, although they might be frequently used in scheduling at the operational level, especially when more than one mover is used.

\FloatBarrier
\section{Scheduling problem}\label{sec:scheduling}

The \textsf{Scheduling} problem addresses the operational-level decisions of executing a real, given set of orders~\orders on the given layout \layout with a placement \placement of dispensers $\packedDispensers$ and interfaces $\interfaces$.
The goal is to construct a schedule \schedule of dispensing operations for the given set of orders \orders. 
The schedule $\schedule$ assigns all dispensing operations of orders to movers while adhering to constraints: the capacity of the movers must not be exceeded, and each dispenser and tile can be occupied by only one mover during dispensing operation. 
The mover realizes one order $p\in\orders$ at a time, starting and finishing at some interface location $\interfaces$.
All operations must be realized at the physical locations of the corresponding dispensers and interfaces as defined by the placement~\placement.
Travel time between two operations performed at different tiles is equal to their $\ell_1$ distance in the layout $\layout$.

In scheduling terms, each order $p\in\orders$ represents a job. Every order $p$ consists of operations $\tasks_p$. Each operation is either dispensing of drug $g$ (with duration $\Delta^g_p$ ) or a cartridge exchange (loading/collection - both with constant duration \trayOperation). Interfaces and drug dispensers are represented as resources; every dispenser for drug $g$ can realize any $g$-dispensing operation, and every interface can realize any cartridge exchange operation. Two more resources are consumed with each operation - the tile below the assigned dispenser and a mover $m$. Mover $m$ is a take-give resource, meaning all operations from an order $p$ must be executed successively on the same mover $m$ before another order $p'$ can be realized by $m$. Travel times are incorporated through a transition distance matrix (later defined as \distMatrix{} in Section~\ref{distMatrixDef}) that specifies the minimum time interval required between consecutive operations based on their relative distance. There are no precedence constraints between the operations within an order, except that cartridge loading must be the first operation and cartridge collection the last. The output of the scheduling algorithm comprises the assignment of each order to a specific mover, the start times of all operations, the allocation of each drug-dispensing operation to a particular dispenser (in cases where multiple dispensers are available), and the assignment of each \textsc{start} and \textsc{finish} (cartridge loading and collecting, respectively) operation to specific interfaces. 

Formally, we define a set of all individual dispensing operations and cartridge exchange operations $\tasks = \bigcup_{p \in \orders} \tasks_p$, where for each order $p \in \orders$ we define a set of 4-tuples, one for each drug $g \in \tasksForPatient{p}$ in order $p$, one for a \textsc{start} operation and one for a \textsc{finish} operation:
\[
\tasks_p = \{(p, g, \fillDuration, \textsc{dispensing}) \mid \forall g \in \tasksForPatient{p}\} \cup \{(p,\textsc{interface},\trayOperation,\textsc{start}),(p,\textsc{interface},\trayOperation,\textsc{finish})\}.
\]
Thus, each operation is defined by a 4-tuple $(p, g, \fillDuration, \tau)$, where $p \in \orders$ denotes the order, $g \in \drugs \cup \{\textsc{interface}\}$ is the drug or interface, and $\fillDuration \in \mathbb{N}$ is the duration (interface operations always require $\trayOperation$ time units). The type $\tau$ distinguishes between $\textsc{dispensing}$ drug $g \in \drugs$ and the loading ($\textsc{start}$) or collecting ($\textsc{finish}$) of cartridges.

\subsection{Constraint Programming Model Variables}
Generally, scheduling problems formulated as Constraint Programming models use the concept of \emph{interval variables}~\citep{heinz2022constraint}, which serve as the primary building blocks of the model. For each 4-tuple $t \in \tasks$, we create an interval variable $\texttt{task\_var}_t$ with a duration $\fillDuration$ or $\trayOperation$ representing its execution, and whose start time is determined by the scheduler.
To allow the solver to decide which dispensing operation will be performed on which mover and at which dispenser, we define the set of alternatives as $\taskAlternatives= \bigcup_{t \in \tasks} \taskAlternatives_t$, 
where for each operation $t = (p, g, \Delta, \tau) \in \tasks$, we define the  set $\mathcal{A}_t = \{(m, t, h) \mid \forall m \in \movers,\, \forall h \in \layout: g \in \placement^{-1}(h)\}$, where $\placement^{-1}(h)$ is the set of drugs dispensed at tile $h\in\layout$.

Each alternative $(m, t, h)$ corresponds to the execution of operation $t$ by mover $m$ at tile $h$, while preserving the duration $\fillDuration$ and type $\tau$ of the original operation. We then create interval variables for all alternative tuples in the same manner as with the original tuples $t \in \tasks$. We denote these interval variables $\texttt{alt\_task}_{m,t,h}$.
To ensure that two interval variables do not overlap, they must be included within the same wrapper (data structure known as a sequence variable in IBM CP Optimizer). Consequently, multiple wrappers are used to group interval variables based on shared properties. An illustration of these wrappers is provided in Figure~\ref{fig:sched:variables_showcase}. It shows how interval variables for a specific order on a specific mover are grouped (e.g., $\texttt{order\_wrap}_{m_1,p_6}$), and how these order-specific groups are then collected under a main wrapper for each mover (e.g., $\texttt{mover\_wrap}_{m_0}$).

\begin{table}[H]
    \centering
    \small % Zmenšení písma na small (nebo \footnotesize pro ještě menší)
    \renewcommand{\arraystretch}{1.1} % Mírné snížení výšky řádků (standard je 1)
    
    \begin{tabularx}{\textwidth}{>{\raggedright\arraybackslash}p{4.5cm} >{\raggedright\arraybackslash}X}
        \toprule
        Variable & Description \\
        \midrule
        $\texttt{task\_var}_t$ & Interval variable, $\forall t \in \tasks$. Represents filling of a specific drug for an order or loading/unloading its cartridge. \\
        
        $\texttt{alt\_task}_{m,t,h}$ & Optional interval variable, $\forall (m,t,h) \in \taskAlternatives$. Represents the possibility of mover $m$ executing operation $t$ on tile $h$. \\
        
        $\texttt{order\_wrap}_{m,p}$ & Optional wrapper, $\forall m \in \movers, \forall p \in \orders$. Represents order $p$ being served by mover $m$. Spans across all interval variables for order $p$. \\
        
        $\texttt{wrapper\_mover\_tasks}_{m}$ & Wrapper for all operations assignable to mover $m$. For example, for $m' \in \movers$, $\texttt{wrapper\_mover\_tasks}_{m'} = \{\texttt{alt\_task}_{m,t,h} \mid \forall (m,t,h) \in \taskAlternatives, m = m' \}$. \\
        
        $\texttt{wrapper\_identical\_tiles}_h$ & Wrapper for all operations executable to tile $h$. For example, for $h' \in \layout$, $\texttt{wrapper\_identical\_tiles}_{h'} = \{\texttt{alt\_task}_{m,t,h} \mid \forall (m,t,h) \in \taskAlternatives, h = h' \}$. \\
        
        $\texttt{mover\_wrap}_m$ & Wrapper: $\{\texttt{mover\_wrap}_{m,p} \mid \forall p \in \orders \}$, $\forall m \in \movers$. Wrapper for all orders served by mover $m$. \\
        \bottomrule
    \end{tabularx}
    \caption{Scheduling model variables.}
    \label{table:sched:scheduling_variables}
\end{table}
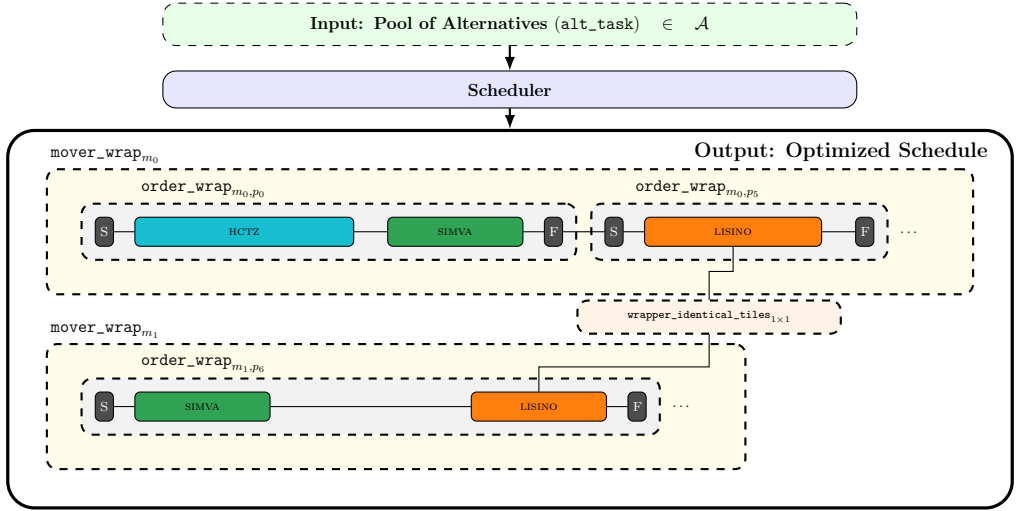
\begin{figure}[htbp]
    \centering
    \begin{tikzpicture}[
        % --- STYLES ---
        task/.style 2 args={ rectangle, draw, font=\small, minimum height=0.7cm, text width=#1, align=center, fill=#2, rounded corners=2pt, text=black },
        task_dark/.style 2 args={ task={#1}{#2}, text=white },
        wrapper/.style={ rectangle, draw, thick, dashed, align=center, fill=#1!10, rounded corners=4pt, inner sep=0.15cm },
        arrow/.style={->, >=latex, thick},
        process_box/.style={ rectangle, draw, fill=blue!10, text width=16cm, minimum height=0.8cm, align=center, font=\large\bfseries, rounded corners=5pt, inner sep=0.3cm },
        pool_box/.style={ rectangle, draw, fill=green!10, text width=16cm, minimum height=0.8cm, align=center, font=\large\bfseries, rounded corners=5pt, inner sep=0.3cm, dashed },
        overall_frame/.style={ rectangle, draw=black, very thick, rounded corners=10pt, inner sep=0.5cm },
        scale=0.55,
        transform shape
    ]
        % Load necessary TikZ libraries
        \usetikzlibrary{positioning, chains, fit, calc}
        \pgfdeclarelayer{background} \pgfdeclarelayer{inner_background} \pgfdeclarelayer{main}
        \pgfsetlayers{background, inner_background, main}

        % Define Drug Colors
        \definecolor{catHYDROCHLOROTHIAZIDE}{HTML}{17BECF}
        \definecolor{catLISINOPRIL}{HTML}{FF7F0E}
        \definecolor{catSIMVASTATIN}{HTML}{31A354}
        \definecolor{catINTERFACE}{HTML}{4D4D4D}

        % --- TOP SECTION: Input and Process (Full Width & Centered) ---
        \node[pool_box] (pool_alternatives) at (2.2,4)
        { Input: Pool of Alternatives $(\texttt{alt\_task}) \in \mathcal{A}$ };

        \node[process_box, below = 0.6cm of pool_alternatives] (scheduler_process)
        { Scheduler };

        \draw[arrow] (pool_alternatives.south) -- (scheduler_process.north);

        % --- BOTTOM SECTION: The Schedule ---
        % Mover 0
        \node[task_dark={0.2cm}{catINTERFACE}] (p1_start) at (-7.5, -1) {\textsc{S} }; 
        \node[task={5cm}{catHYDROCHLOROTHIAZIDE}, right=0.5cm of p1_start] (p1_hctz) {\textsc{hctz} }; 
        \node[task={3cm}{catSIMVASTATIN}, right=0.8cm of p1_hctz] (p1_simva) {\textsc{simva} }; 
        \node[task_dark={0.2cm}{catINTERFACE}, right=0.5cm of p1_simva] (p1_end) {\textsc{F} };

        \node[task_dark={0.2cm}{catINTERFACE}, right=1cm of p1_end] (p2_start) {\textsc{S} };
        \node[task={4cm}{catLISINOPRIL}, right=0.5cm of p2_start] (p2_lisino) {\textsc{lisino} }; 
        \node[task_dark={0.2cm}{catINTERFACE}, right=0.8cm of p2_lisino] (p2_end) {\textsc{F} };
        \node[right=0.5cm of p2_end] (m0_end) {$\dots$};

        % Mover 1
        \node[task_dark={0.2cm}{catINTERFACE}, below=3.5cm of p1_start] (p3_start) {\textsc{S} };
        \node[task={3cm}{catSIMVASTATIN}, right=0.5cm of p3_start] (p3_simva) {\textsc{simva} }; 
        \node[task={3cm}{catLISINOPRIL}, right=4.8cm of p3_simva] (p3_lisino) {\textsc{lisino} }; 
        \node[task_dark={0.2cm}{catINTERFACE}, right=0.5cm of p3_lisino] (p3_end) {\textsc{F} };
        \node[right=0.5cm of p3_end] (m1_end) {$\dots$};

        % MOVER WRAPPERS (Yellow) - Labels moved ABOVE
        \begin{pgfonlayer} {background} 
            \node[wrapper=yellow, fit=(p1_start) (m0_end), inner sep=18pt] (wrap1) {}; 
            \node[anchor=south west, font=\large] at (wrap1.north west) {$\texttt{mover\_wrap}_{m_0}$};
            
            \node[wrapper=yellow, fit=(p3_start) (m1_end), inner sep=18pt] (wrap2) {};
            \node[anchor=south west, font=\large] at (wrap2.north west) {$\texttt{mover\_wrap}_{m_1}$};
        \end{pgfonlayer}

        % ORDER WRAPPERS (Gray)
        \begin{pgfonlayer} {inner_background} 
            \node[wrapper=gray, fit=(p1_start) (p1_end), inner sep=5pt, label={[font=\large, xshift=-3cm]north:$\texttt{order\_wrap}_{m_0,p_0}$}] (mvar1) {}; 
            \node[wrapper=gray, fit=(p2_start) (p2_end), inner sep=5pt, label={[font=\large, xshift=-1cm]north:$\texttt{order\_wrap}_{m_0,p_5}$}] (mvar2) {}; 
            \node[wrapper=gray, fit=(p3_start) (p3_end), inner sep=5pt, label={[font=\large, xshift=-4cm]north:$\texttt{order\_wrap}_{m_1,p_6}$}] (mvar3) {};
        \end{pgfonlayer}

        % Identical Tiles Sequence (Orange)
        \node[wrapper=orange, minimum height=0.8cm, text width=6cm, fill=orange!10, font=\footnotesize] (identical_tiles_1)
        at ($(p1_end.south east)!0.5!(p2_end.south west) + (0,-1.7cm)$) {$\texttt{wrapper\_identical\_tiles}_{1\times1}$};

        % Sequence linkages
        \draw (p1_start.east) -- (p1_hctz.west);
        \draw (p1_hctz.east) -- (p1_simva.west); 
        \draw (p1_simva.east) -- (p1_end.west);
        \draw (p1_end.east) -- (p2_start.west);
        \draw (p2_start.east) -- (p2_lisino.west);
        \draw (p2_lisino.east) -- (p2_end.west); 
        \draw (p3_start.east) -- (p3_simva.west);
        \draw (p3_simva.east) -- (p3_lisino.west); 
        \draw (p3_lisino.east) -- (p3_end.west);

        % Identical Tiles Lines
        \draw (p2_lisino.south) -- ++(0,-0.6) -| (identical_tiles_1.north);
        \draw (p3_lisino.north) -- ++(0,0.6) -| (identical_tiles_1.south);

        % --- Overall Frame for the Schedule part ---
        \node[overall_frame, fit=(wrap1) (wrap2)] (output_frame) {};
        
        % Output Caption - Moved ABOVE the yellow wrappers inside the frame
        \node[anchor=north east, font=\Large\bfseries] at (output_frame.north east) [xshift=-0.5cm, yshift=-0.2cm] {Output: Optimized Schedule};

        % Arrow from Scheduler to Output Frame
        \draw[arrow] (scheduler_process.south) -- (output_frame.north);

    \end{tikzpicture}
    \caption[Illustration of the scheduler’s optimized schedule]{The scheduler maps input operations (left) to an optimized assigned schedule (right). Depicted intervals represent chosen operation alternatives for movers $m_0$ and $m_1$. This schedule fulfills three orders via $\texttt{order\_wrap}_{m_0,p_0}$, $\texttt{order\_wrap}_{m_0,p_5}$, and $\texttt{order\_wrap}_{m_1,p_6}$. Both \textsc{lisinopril} operations are grouped under $\texttt{wrapper\_identical\_tiles}_{1 \times 1}$, indicating both were dispensed from the same dispenser at tile (1,1). The two \textsc{simvastatin} operations are not grouped in such a way, indicating each was filled at a different dispenser.}
    \label{fig:sched:variables_showcase}
\end{figure}
 
\subsection{Constraint Programming Model Constraints}\label{sec:scheduling_constraints}
The optimal allocation of \emph{dispensing operations} $t \in \tasks$ to movers lies at the core of this scheduling problem. The objective is to assign operations to movers so that the total execution time is minimal while adhering to several vital constraints.

\subsubsection{Alternative Movers and Dispensers}
For every operation $t \in \tasks$, we define $\texttt{task\_var}_t$ as a \textit{master} interval variable. For each feasible (mover, tile) execution option, we have defined optional \textit{alternatives} $\texttt{alt\_task}_{m,t,h}$. For every $t' \in \tasks$, we link the \textit{master} variable $\texttt{task\_var}_{t'}$ to its alternatives using
\begin{equation}
\textsc{alternative}\left(\texttt{task\_var}_{t'}, \{\texttt{alt\_task}_{m,t,h}\mid \forall (m,t,h) \in \taskAlternatives, t= t'\}\right).
\label{eq:alternative}
\end{equation}
This guarantees that exactly one alternative is chosen for the given $\texttt{task\_var}_{t'}$.

\subsubsection{No Overlap of Mover Activities}\label{distMatrixDef}
Our initial step in model construction is to prohibit movers from using multiple tiles at the same time, i.e., a single mover cannot be in multiple physical locations at once. 

We enforce no overlap among all interval variables targeting mover $m \in \movers$. Furthermore, we enforce travel times (i.e., gaps in the schedule) between any two interval variables $\texttt{alt\_task}_{m,t,h}$ and $\texttt{alt\_task}_{m',t',h'}$, $\forall (m,t,h), (m',t',h') \in \taskAlternatives$, equal to the physical distance between the corresponding dispensers' positions. The distance matrix \distMatrix{} captures the travel time between tiles $h,h' \in \layout$, defined as:
\begin{equation}
    \distMatrix[(m,t,h),(m',t',h')] = \|h - h'\|_1.
\end{equation}

We pass \distMatrix{} to the model as part of \textsc{no\_overlap} constraint to enforce the travel times:
\begin{equation}
    \textsc{no\_overlap}\left(\{\texttt{alt\_task}_{m,t,h} \mid \forall (m,t,h) \in \taskAlternatives, m = m' \}\}, \distMatrix \right) \quad \forall m' \in \movers.
\end{equation}

\subsubsection{Order Completion}
Ensuring non-overlapping execution of dispensing operations is a necessary condition but not sufficient. 
We also have to enforce that no two dispensing operations of different orders interleave; an order must be completed before starting a new one.

For each order $p' \in \orders$ and for each mover $m' \in \movers$, we group the corresponding optional interval variables into a wrapper $\texttt{order\_wrap}_{p',m'}$, defined as:
\[
\texttt{order\_wrap}_{p',m'} = \{ \texttt{alt\_task}_{m,t,h} \mid (m,t,h) \in \taskAlternatives,m=m', p' \in t \}.
\]

Then, for each mover $m' \in \movers$, we group all order wrappers $\texttt{order\_wrap}_{p,m'} $ for all $p \in \orders$ into a $\texttt{mover\_wrap}_{m'}$:
\[
\texttt{mover\_wrap}_{m} = \{ \texttt{order\_wrap}_{p,m} \mid \forall p \in \orders \} \quad \forall m \in \movers.
\]
This hierarchical structure is visualized in Figure~\ref{fig:sched:variables_showcase}. For instance, mover $m_0$ is assigned orders $p_0$ and $p_5$, thus the top-level wrapper $\texttt{mover\_wrap}_{m_0}$ contains the individual order wrappers $\texttt{order\_wrap}_{m_0,p_0}$ and $\texttt{order\_wrap}_{m_0,p_5}$.

As seen in Figure~\ref{fig:sched:variables_showcase}, each $\texttt{order\_wrap}$ wrapper begins with a $\textsc{start}$ operation and concludes with a $\textsc{finish}$ operation. To ensure a feasible ordering, we must guarantee that all dispensing interval variables occur strictly between the $\textsc{start}$ and $\textsc{finish}$ intervals.
To formalize this, we first define three subsets of alternative tasks for each order $p \in \orders$ and mover $m \in \movers$. Let $\mathcal{A}\mathit{lt}^{\textsc{start}}_{p,m}$ be the set of operations representing loading and empty cartridge, $\mathcal{A}\mathit{lt}^{\textsc{finish}}_{p,m}$ be the set of operations representing unloading a completed cartridge, and let $\mathcal{A}\mathit{lt}^{\textsc{dispensing}}_{p,m}$ be the set of operations representing drug dispensing:
\begin{align*}
    \mathcal{A}\mathit{lt}^{\tau}_{p,m} &= \left\{ \texttt{alt\_task}_{m,t,h} \mid (m,t,h) \in \taskAlternatives, \ p \in t \land \tau \in t \right\}\quad \forall \tau \in \{\textsc{start},\textsc{dispensing}, \textsc{finish}\}.
\end{align*}
We then apply a constraint that ensures every possible \textsc{start} operation begins before any dispensing operation. This guarantees that regardless of which interface is selected, the initialization always precedes any dispensing:
\begin{equation}
    \begin{aligned}
        \textsc{start\_before\_start}\left( a_{start}, a_{disp} \right), \\
        \forall p \in \orders, \forall m \in \movers,
        \forall a_{start} \in \mathcal{A}\mathit{lt}^{\textsc{start}}_{p,m}, \forall a_{disp} \in \mathcal{A}\mathit{lt}^{\textsc{dispensing}}_{p,m}
    \end{aligned}
\end{equation}

Similarly, we must enforce that every dispensing operation ends before the \textsc{finish} operation ends. This prevents any drug filling task from occurring after the order wrapper has supposedly finished:
\begin{equation}
    \begin{aligned}
        \textsc{end\_before\_end}\left( a_{disp}, a_{finish} \right), \\
        \forall p \in \orders, \forall m \in \movers, 
        \forall a_{disp} \in \mathcal{A}\mathit{lt}^{\textsc{dispensing}}_{p,m}, \forall a_{finish} \in \mathcal{A}\mathit{lt}^{\textsc{finish}}_{p,m}
    \end{aligned}
\end{equation}

This yields a result comparable to what we would obtain using a \textsc{span} constraint. 
Finally, we make sure that \emph{dispensing operations} for different orders assigned to the same mover do not overlap:
\begin{equation}
    \textsc{no\_overlap}\left(\texttt{mover\_wrap}_m\right)\quad \forall m \in \movers.
\end{equation}

This not only enforces non-overlapping execution, but also guarantees full separation between orders: a mover must complete all operations for an order $p$ before beginning any loading/dispensing operations for another order $p'$.

Additional mechanisms are required to ensure that the previous constraint operates as intended. In its current formulation, each mover wrapper contains order wrappers corresponding to all orders. The following technique ensures that an order wrapper is instantiated only when the corresponding mover–order assignment alternative is selected (i.e., if order 3 is assigned to mover 2, then only the wrapper $\texttt{order\_wrap}_{3,2}$ is instantiated and placed within $\texttt{mover\_wrap}_2$).

For each order $p \in \orders$, we define
\[
u_{t',m'} = \sum_{a \in B} \textsc{presence}(a) \qquad \forall t' \in \tasks_{p},\forall m' \in \movers, B = \left\{\texttt{alt\_task}_{t,m,h}\mid (t,m,h) \in \taskAlternatives, t=t', m=m'\right\}.
\]

For each operation $t \in \tasks_p$ and mover $m \in \movers$, the variable $u_{t,m}$ reflects whether mover $m$ executes operation $t$: it equals $1$ only for the mover assigned to that operation and $0$ otherwise.

Let $t_1$ be the first operation in $T_p$. By enforcing $u_{p,m} = u_{t_1,m}$ we propagate the existing mover assignment information from the operation-level variables $u_{t,m}$ to a single representative variable $u_{p,m}$ for the entire order $p$ --- $u_{p,m}$ will be equal to $1$ only for the mover executing the order and $0$ for all others.

We then make sure $order\_wrap_{p,m}$ is present only when $u_{p,m}$ is $1$:
\[
\textsc{presence}(\texttt{order\_wrap}_{p,m}) = u_{p,m}.
\]

\subsubsection{No Overlap over Dispensing Activities}
Every dispenser must also allow at most one operation at any given moment. To achieve that, for any given tile $h' \in \layout$, we create a wrapper $\texttt{sequence\_identical\_tiles}_{h'}$ of $\texttt{alt\_tasks}_{m,t,h}, \forall m \in \movers, \forall t \in \tasks, h=h'$. Figure~\ref{fig:sched:variables_showcase} shows an example of this with $\texttt{sequence\_identical\_tiles}_{1 \times 1}$. This wrapper groups interval variables across different orders ($p_5, p_6$) and movers ($m_0, m_1$) that all require the same dispenser tile. We ensure the tile is not used concurrently by multiple movers with another \textsc{no\_overlap} constraint:
\begin{equation}
    \textsc{no\_overlap}\left(\texttt{wrapper\_identical\_tiles}_h\right)\quad \forall h \in \layout.
\end{equation}
Finally, the model's objective is to minimize the latest finish time across all interval variables:
\begin{equation}
    \min \max \left\{\textsc{end\_of}(\texttt{task\_var}_t) \mid \forall t \in \tasks\right\}.
\end{equation}
This completes the Constraint Programming formulation for the core scheduling problem, which assigns dispensing operations to movers and determines their order. 
The model captures all hard constraints---mover, tile, and dispenser exclusivity, travel times between dispensers, while intentionally relaxing conflict resolution among movers sharing pathways.
In the following subsection, we present a calculation  of a lower bound on \textsf{Scheduling}.

\subsection{Lower Bound: Parallel Movers Using SHPPN} \label{sec:lowerbound}
We propose a lower bound on the optimal makespan with the two goals: (a)~to estimate the quality of the solution of \textsf{Scheduling}, and (b)~to provide a warm start solution for the solver.
The main idea of the lower bound consists of relaxing the overlaps on dispensers and the travel times between interfaces.
On the other hand, we include travel times between the dispensers themselves.
This is achieved by constructing an instance of a TSP with neighborhoods (TSPN) for each order $p\in\orders$.
Subsequently, the results from TSPN would be used as constants in the parallel machine scheduling problem $P||C_\text{max}$, with jobs having processing times set to the lengths of the optimal tours found by TSPN and a set of resources corresponding to the available movers.
Let
\begin{align}
    \shortestPathPatient =  \min_{\permElement \in \interfaces \times \Pi(p) \times \interfaces} \sum^{|\permElement|-1}_{j=1} \tilesDistance{j} \label{eq:gtsp_lb}
\end{align}
be the length of the shortest path between some two interface locations that visits all dispensers $\tasksForPatient{p}$ required for order $p$.
Then we define
\begin{equation}
    \patientPathTime = 2\trayOperation + \shortestPathPatient + \sum_{(g,\Delta^p_g)\in p} \Delta^p_g,
\end{equation}

denoting the lower bound on the total processing time for order $p$, where \trayOperation is the time of releasing/collecting an empty cartridge, \shortestPathPatient is the shortest possible traveling time of the mover performing order $p$, and $\Delta^p_g$ is the time to dispense drug $g$ in customer order $p$.

% \tn{trosku sjednotit terminologii s tim co bude v placementu}
The values $\patientPathTime$ can be computed by solving one instance of the SHPPN for each customer order $p$ separately.
% We speak about SHPN rather than TSPN at this point, since the mover can start and finish at different interface locations $\interfaces$.
However, SHPPN can be reduced to the TSPN problem and subsequently reduced to standard TSP using the reduction of~\cite{noon1993efficient} with the introduction of only a single additional vertex.
Since the instance of the SHPPN typically involves only small graphs (e.g., $\approx 30$ vertices since a typical order might consist of 4 drugs plus additional start and finish interface locations, each with approximately 5 alternatives), the resulting instances of standard TSP can be solved efficiently by, e.g., lazy subtour elimination in a fraction of a second.

With the values of \patientPathTime, the lower bound is computed as follows.
Let us define an instance of the $P||C_{max}$ problem (i.e., scheduling on parallel identical machines with makespan minimization) with \numMovers resources and a set of $n$ jobs with processing times $\{\patientPathTime[1], \ldots, \patientPathTime[|\orders|]\}$.
We claim that $P,\numMovers|\{\patientPathTime\}^n_{i=1}|C_{max}$ acts as a lower bound on schedule makespan.

% \tn{tady jsem skoncil}
\begin{proposition}[SHPPN lower bound]
    Optimal solution of $P,\numMovers|\{\patientPathTime\}^n_{i=1}|C_{max}$ is a lower bound on \textsf{Scheduling}.
\end{proposition}
Without detailing the proof, the main idea of SHPPN lower bound is that the optimal solution of $P,\numMovers|\{\patientPathTime\}^n_{i=1}|C_{max}$ satisfies all constraints on \textsf{Scheduling}, detailed in Section~\ref{sec:scheduling_constraints}, except it relaxes on: (i)~the overlap at dispensers, tiles and interfaces locations, and (ii)~the travel time between two interface locations where a mover finishes with a completed order and the following interface where it starts with the next order (i.e., instantaneous teleportation between interface locations). 

Regarding the complexity of computing the SHPPN lower bound, it is known that $P||C_{max}$ is (weakly) \textsf{NP}-hard already for two machines (i.e., movers) and strongly \textsf{NP}-hard for an unbounded number of machines.
Computing values of $\patientPathTime$ for the realistic scale of the problem can be done efficiently in practice with an existing TSP solver.
Therefore, we deem that the lower bound is computationally practical.
Additionally, we also use the solution computed by the lower bound procedure to warm-start the Constraint Programming model for the \textsf{Scheduling} problem by supplying mover-order assignments, which accelerate its solution.

\FloatBarrier
\section{Routing Problem}\label{sec:routing}

The objective of the \textsf{Routing} problem is to construct feasible $\ell_1$ paths for each mover $m \in \movers$ based on the pre-determined schedule of dispensing operations. Our routing focuses on resolving \emph{conflicts}, defined here as simultaneous tile occupancy. Unlike classical MAPF, our routing has to precisely follow the schedule. Furthermore, spatial overlap of two movers on one tile is permissible provided no dispensing occurs. If during dispensing on one mover another has to travel through the same tile, the dispensing is paused. Since standard MAPF solvers cannot accommodate these context-dependent rules, we employ a tailored iterative heuristic.

The algorithm begins by generating fixed \emph{resting sites} --- locations where movers can wait without obstructing access to dispensers (Section~\ref{routing:resting_sites_generation}). In each iteration, an ILP assigns idle movers to these sites (Section~\ref{routing:resting_sites_assignment}), and a polynomial-time algorithm constructs least-congested shortest paths between operations and resting spots. These paths are used to detect conflicts. Although rare, if any arise, the schedule is revised and the process repeats if needed.

\subsection{Resting Sites Generation}\label{routing:resting_sites_generation}
Movers may experience idle times between operation executions due to dispenser unavailability. To manage these times efficiently and maintain system throughput, we assign \emph{resting sites} --- specific locations where movers can remain stationary without interfering with dispensing operations on adjacent tiles.

We consider a fixed dispenser nozzle at the center of each tile, which enables cooperative operation by placing resting sites at tile intersections. This layout allows a resting mover (represented in gray in Figure~\ref{fig:routing:mover_rotation}) to remain idle while another mover utilizes an adjacent tile for transit or dispensing. To maintain clearance for these resting movers at the intersections, the strategy leverages the mover’s rotational capabilities: after filling the capsules on one side of the cartridge (purple side, Figure~\ref{fig:routing_mover_rotation_a}), the mover shifts slightly to rotate $180^\circ$ (depicted in Figure~\ref{fig:routing_mover_rotation_b}) and fill the other side (green side, Figure~\ref{fig:routing_mover_rotation_c}), which together comprise the full cartridge. By utilizing this rotation instead of traversing the entire tile to fill every capsule, the mover minimizes its footprint (the patterned area in Figure~\ref{fig:routing:mover_rotation}), thereby preserving the necessary clearance  for half a resting site at the boundary.
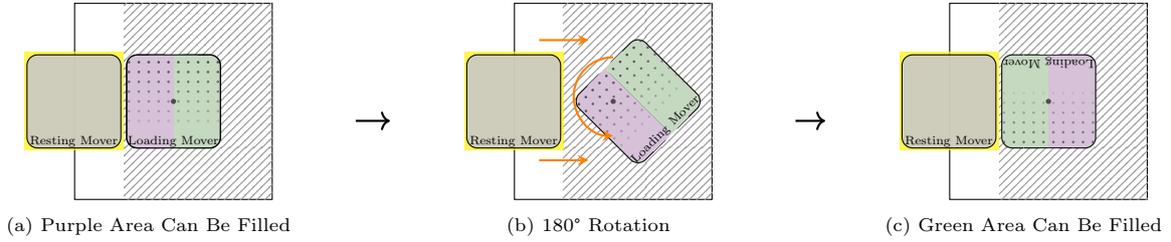
\begin{figure}[H]
    \centering
    % DEFINE SHARED STYLES AND COLORS HERE
    \tikzset{
        every node/.style={font=\small},
        reachafilled/.style={draw,rectangle,minimum width=3cm,minimum height=4cm, draw=none, pattern=north east lines, pattern color=ptrn},
        rest_site/.style={draw,rectangle,minimum width=2cm,minimum height=2cm,fill=yellow, opacity=0.7, draw=none},
        mover/.style={draw,rectangle,minimum width=1.9cm,minimum height=1.9cm,fill=mvr, fill opacity=0.9,rounded corners,},
        disp_tile/.style={draw,rectangle,minimum width=4cm,minimum height=4cm},
        top and bottom/.style={
        path picture={
            \draw (path picture bounding box.north west) -- (path picture bounding box.north east);
            \draw (path picture bounding box.south west) -- (path picture bounding box.south east);
        }
    }
    }
    \definecolor{unfilled}{HTML}{D6C1D9}
    \definecolor{filled}{HTML}{c5d9c1}
    \definecolor{ptrn}{HTML}{969696}
    \definecolor{mvr}{HTML}{c9c9c9}

    % ================= SUBFIGURE 1: Initial State =================
    \begin{subfigure}[c]{0.3\textwidth}
        \centering
        % Scale is applied here
        \begin{tikzpicture}[scale=0.65, transform shape]
            % Tile
            \node (tile_1) [disp_tile] {};
            
            % Reachable Zone
            \node (reach) [reachafilled, xshift=0.5cm] at (tile_1){};
            
            % Resting site & Mover
            \node (rest_site_1) [rest_site, left = -1cm of tile_1] {};
            \node (resting_mover) [mover, label={[xshift=0cm, yshift=-2cm, font=\scriptsize,black]Resting Mover}] at (rest_site_1) {};

            % Loader Mover (Straight)
            \node[mover, fill=filled, opacity=1,minimum width=0.95cm,xshift=0.48cm] (load_1_l) at (tile_1) {};
            \node[mover, fill=unfilled, opacity=1, minimum width=0.95cm,xshift=-0.47cm] (load_1_r) at (tile_1) {};
            % Cover middle
            \node[draw, rectangle, fill=filled, minimum height=1.9cm, minimum width = 0.5cm,draw=none,xshift=0.25cm] at (tile_1){};
            \node[draw, rectangle, fill=unfilled, minimum height=1.9cm, minimum width = 0.5cm,draw=none,xshift=-0.25cm] at (tile_1){};
            \node[mover, fill=none,label={[xshift=0cm, yshift=-2cm, font=\scriptsize,black]Loading Mover}] (load_1) at (tile_1){};

            % Capsules
            \foreach \i in {0,...,9} { \foreach \j in {0,...,6} {
                \pgfmathsetmacro{\opacity}{1 - \j * 0.13}
                \fill[black!70, opacity=\opacity] ($(tile_1.center) + (-0.9cm + \i*0.2cm, 0.8cm - \j*0.2cm)$) circle (0.75pt);
            }}

            % Nozzle
            \fill[black!70] ($(tile_1.center)$) circle (1.5pt);
            
        \end{tikzpicture}
        \caption{Purple Area Can Be Filled}
        \label{fig:routing_mover_rotation_a}
    \end{subfigure}
    \hfill
    \begin{tikzpicture}[scale=0.75, transform shape, baseline=-0.5ex]
         \draw[->,thick] (0,0) -- (0.6cm,0cm);
    \end{tikzpicture}%
    \hfill
    % ================= SUBFIGURE 2: Middle State (Rotated) =================
    \begin{subfigure}[c]{0.3\textwidth}
        \centering
        \begin{tikzpicture}[scale=0.65, transform shape]
            \node (tile_2) [disp_tile] {};
            \node (reach2) [reachafilled, xshift=0.5cm] at (tile_2){};
            \node (rest_site_2) [rest_site, left = -1cm of tile_2] {};
            \node (resting_mover_2) [mover, label={[xshift=0cm, yshift=-2cm, font=\scriptsize]Resting Mover}] at (rest_site_2) {};

            % Center point for rotation
            \coordinate (rot_center) at ($(tile_2.center) + (0.5cm, 0cm)$);

            % --- ROTATED SECTION ---
            % FIX: Removed 'rotate around' from a surrounding scope.
            % Applied rotation directly to nodes relative to (rot_center).

            % Loader parts rotated 45 degrees
            \node[mover, fill=filled,opacity=1, rotate around={45:(rot_center)}, minimum width=0.95cm] (load_2_l) at ($(rot_center) + (0.48cm,0)$) {};
            \node[mover, fill=unfilled,opacity=1, rotate around={45:(rot_center)}, minimum width=0.95cm] (load_2_r) at ($(rot_center) + (-0.47cm,0)$) {};

            % Covers rotated 45 degrees
            \node[draw, rectangle, rotate around={45:(rot_center)}, fill=filled, minimum height=1.9cm, minimum width = 0.5cm,draw=none] at ($(rot_center) + (0.25cm,0)$){};
            \node[draw, rectangle, rotate around={45:(rot_center)}, fill=unfilled, minimum height=1.9cm, minimum width = 0.5cm,draw=none] at ($(rot_center) + (-0.25cm,0)$){};

            % Label rotated 45 degrees
            \node[font=\scriptsize,opacity=1,rotate around={45:(rot_center)},yshift=-0.8cm] (load_2) at (rot_center){Loading Mover};

           % --- ROTATED CAPSULES ---
            % We use a scope to rotate the coordinate system around the same center
            \begin{scope}[rotate around={46:($(rot_center) + (-0.87cm, 0.82cm)$) }]
            \foreach \i in {0,...,9} { 
                \foreach \j in {0,...,6} {
                     \pgfmathsetmacro{\opacity}{1 - \j * 0.13}
                     % Note: We calculate positions relative to (rot_center) inside the rotated scope
                     \fill[black!70, opacity=\opacity] 
                         ($(rot_center) + (-0.87cm + \i*0.2cm, 0.82cm - \j*0.2cm)$) 
                         circle (0.75pt);
                }
            }
            \end{scope}

            % Rotation arrows
            \draw[->,thick,>=stealth,orange]  ($(tile_2.center) + (-1.5cm,1.25cm)$) -- ($(tile_2.center) + (-0.5,1.25cm)$) ;
            \draw[->,thick,>=stealth,orange]  ($(tile_2.center) + (-1.5cm,-1.2cm)$) -- ($(tile_2.center) + (-0.5,-1.2cm)$) ;
            \draw[->,thick,>=stealth,orange]  ($(tile_2.center) + (0,0.9cm)$) arc (90:270:0.8);

            % Nozzle
            \fill[black!70] ($(tile_2.center)$) circle (1.5pt);
            
        \end{tikzpicture}
        \caption{180° Rotation}
        \label{fig:routing_mover_rotation_b}
    \end{subfigure}
    \hfill
    % ARROW 1
    % Added baseline=-0.5ex to center vertically, removed yshift
    \begin{tikzpicture}[scale=0.65, transform shape, baseline=-0.5ex]
         \draw[->,thick] (0,0) -- (0.6cm,0cm);
    \end{tikzpicture}%
    \hfill
    % ================= SUBFIGURE 3: Final State (Flipped) =================
    \begin{subfigure}[c]{0.3\textwidth}
        \centering
        \begin{tikzpicture}[scale=0.65, transform shape]
            \node (tile_3) [disp_tile] {};
            \node (reach3) [reachafilled, xshift=0.5cm] at (tile_3){};
            \node (rest_site_3) [rest_site, left = -1cm of tile_3] {};
            \node (resting_mover_3) [mover, label={[xshift=0cm, yshift=-2cm, font=\scriptsize]Resting Mover}] at (rest_site_3) {};

            \coordinate (rot_center_3) at (tile_3.center);

            % --- ROTATED SECTION (180 degrees) ---
            % FIX: Applied rotation directly to nodes.

            % Loader parts rotated 180 degrees
            % Note: xshift signs flip compared to subfigure 1 because of 180 rotation convention in TikZ relative to placement
            \node[mover, fill=unfilled,opacity=1, rotate around={180:(rot_center_3)}, minimum width=0.95cm] (load_3_l) at ($(rot_center_3) + (-0.48cm,0)$) {};
            \node[mover, fill=filled,opacity=1, rotate around={180:(rot_center_3)}, minimum width=0.95cm] (load_3_r) at ($(rot_center_3) + (0.47cm,0)$) {};

            % Covers rotated 180 degrees

            \node[top and bottom, rectangle, rotate around={180:(rot_center_3)}, fill=unfilled, minimum height=1.9cm, minimum width = 0.5cm] at ($(rot_center_3) + (-0.25cm,0)$){};
            \node[top and bottom, rectangle, rotate around={180:(rot_center_3)}, fill=filled, minimum height=1.9cm, minimum width = 0.5cm] at ($(rot_center_3) + (0.25cm,0)$){};

            % Label rotated 180 degrees
            \node[font=\scriptsize,rotate around={180:(rot_center_3)},yshift=-0.8cm] (load_3) at (rot_center_3){Loading Mover};

             % Capsules
            \foreach \i in {0,...,9} { \foreach \j in {0,...,6} {
                \pgfmathsetmacro{\opacity}{0 + \j * 0.13}
                \fill[black!70, opacity=\opacity] ($(tile_3.center) + (-0.9cm + \i*0.2cm, 0.37cm - \j*0.2cm)$) circle (0.75pt);
            }}

            % Nozzle
            \fill[black!70] ($(tile_1.center)$) circle (1.5pt);
            
        \end{tikzpicture}
        \caption{Green Area Can Be Filled}
        \label{fig:routing_mover_rotation_c}
    \end{subfigure}

    \caption[Visualization of dispensing maneuver]{Visualization of dispensing maneuver. Geometrically, a mover with dimensions $a\times a$  rotates within a disk of radius $R = a/\sqrt{2}$. On a $2a\times 2a$ tile, this leaves free space of $2a - 2R = a(2-\sqrt{2}) > 0.5a$ along the edge, which is sufficient to accommodate half of a resting mover.}
    \label{fig:routing:mover_rotation}
\end{figure}
However, the rotation depicted in Figure~\ref{fig:routing:mover_rotation} imposes a constraint: each dispensing tile supports only one resting site to preserve rotation space for the dispensing mover. Interface tiles are an exception, accommodating two resting sites since movers remain centered during cartridge swaps.

We formulate the problem of finding viable resting sites layout as a Maximum Independent Set problem over a graph (see Figure~\ref{fig:MIS_graph})  constructed from layout $\layout$. In this graph, edges (depicted as red line segments) represent mutual exclusion constraints --- that is, two nodes (potential resting sites, depicted as blue dots) are connected if resting at both simultaneously would cause dispensing inoperability on their common tile. The resulting solution, shown in Figure~\ref{fig:MIS_result}, highlights the selected resting sites \sites using yellow rectangles.
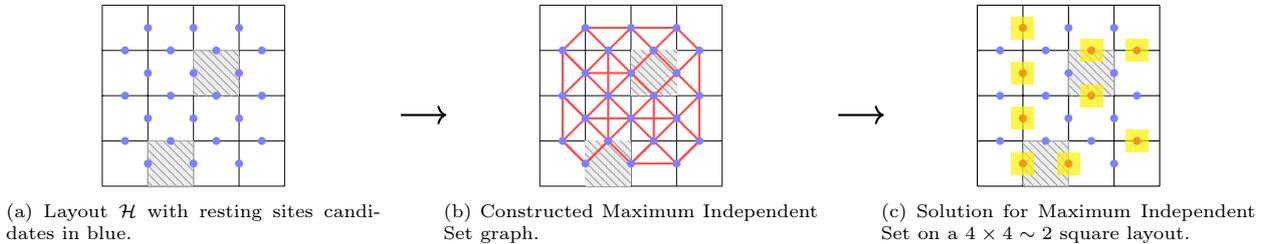
\begin{figure}[H]
    \centering
    % --- GRID 1: Clean grid with interfaces ---
    \begin{subfigure}[b]{0.3\textwidth}
        \centering
        \begin{tikzpicture}[
            dot/.style={circle, fill=blue!50, inner sep=0pt, minimum size=3pt},
            hatched_cell/.style={fill=gray!20, opacity=0.7},
            hatched_pattern/.style={pattern=north west lines, pattern color=gray!70},
            scale=0.6
        ]
        \def\gridmax{4}
        \pgfmathtruncatemacro{\imax}{\gridmax-1}
        
        \draw[step=1, black!90] (0,0) grid (\gridmax,\gridmax);
        
        \fill[hatched_cell] (1,0) rectangle (2,1);
        \fill[hatched_pattern] (1,0) rectangle (2,1);
        \fill[hatched_cell] (2,2) rectangle (3,3);
        \fill[hatched_pattern] (2,2) rectangle (3,3);

        \foreach \x in {0,...,\imax} {
            \foreach \y in {1,...,\imax} {
                \node[dot] at (\x+0.5, \y) {};
            }
        }
        \foreach \x in {1,...,\imax} {
            \foreach \y in {0,...,\imax} {
                \node[dot] at (\x, \y+0.5) {};
            }
        }
        \end{tikzpicture}
        \caption{Layout $\layout$ with resting sites candidates in blue.}
    \end{subfigure}
    \hfill
    % --- ARROW 1 ---
    \begin{minipage}[b]{0.02\textwidth}
        \centering
        \raisebox{1.6cm}{% 
        \begin{tikzpicture}
             \draw[->,thick] (0,0) -- (0.6cm,0cm);
         \end{tikzpicture}
         }%
    \end{minipage}
    \hfill
    % --- GRID 2: With red lines ---
    \begin{subfigure}[b]{0.3\textwidth}
        \centering
        \begin{tikzpicture}[
            dot/.style={circle, fill=blue!50, inner sep=0pt, minimum size=3pt},
            hatched_cell/.style={fill=gray!20, opacity=0.7},
            hatched_pattern/.style={pattern=north west lines, pattern color=gray!70},
            scale=0.6
        ]
        \def\gridmax{4}
        \pgfmathtruncatemacro{\imax}{\gridmax-1}
        
        \draw[step=1, black!90] (0,0) grid (\gridmax,\gridmax);
        
        % Red Graph Lines
        \foreach \x in {1,...,\gridmax} {
            \draw[red!70,thick] (\x - 0.5,1) -- (\x - 0.5,\gridmax-1);
        }
        \foreach \y in {1,...,\gridmax} {
            \draw[red!70,thick] (1,\y-0.5) -- (\gridmax-1,\y-0.5);
        }
        
        \fill[white] (1,0) rectangle (2,1);
        \fill[white] (2,2) rectangle (3,3);
        \fill[hatched_cell] (1,0) rectangle (2,1);
        \fill[hatched_pattern] (1,0) rectangle (2,1);
        \fill[hatched_cell] (2,2) rectangle (3,3);
        \fill[hatched_pattern] (2,2) rectangle (3,3);

        \foreach \x in {1,...,\imax} {
            \foreach \y in {1,...,\imax} {
                \draw[red!70, thick] (\x-0.5, \y) -- (\x, \y-0.5);
                \draw[red!70, thick] (\x, \y+0.5) -- (\x+0.5, \y);
                \draw[red!70, thick] (\x, \y+0.5) -- (\x-0.5, \y);
                \draw[red!70, thick] (\x+0.5, \y) -- (\x, \y-0.5);
            }
        }
        
        \foreach \y in {1,...,\imax} { \draw[red!70, thick] (0.5, \y) -- (1, \y-0.5); }
        \foreach \x in {1,...,\imax} { \draw[red!70, thick] (\x, 0.5) -- (\x+0.5, 1); }
        
        \foreach \x in {0,...,\imax} {
            \foreach \y in {1,...,\imax} { \node[dot] at (\x+0.5, \y) {}; }
        }
        \foreach \x in {1,...,\imax} {
            \foreach \y in {0,...,\imax} { \node[dot] at (\x, \y+0.5) {}; }
        }
        \end{tikzpicture}
        \caption{Constructed Maximum Independent Set graph.}
        \label{fig:MIS_graph}
    \end{subfigure}
    \hfill
    % --- ARROW 2 ---
    \begin{minipage}[b]{0.02\textwidth}
        \centering
        \raisebox{1.6cm}{% 
        \begin{tikzpicture}
             \draw[->,thick] (0,0) -- (0.6cm,0cm);
         \end{tikzpicture}
         }%
    \end{minipage}
    \hfill
    % --- GRID 3: Selected sites ---
    \begin{subfigure}[b]{0.3\textwidth}
        \centering
        \begin{tikzpicture}[
            dot/.style={circle, fill=blue!50, inner sep=0pt, minimum size=3pt},
            yellow_square/.style={fill=yellow, opacity=0.7, draw=none},
            hatched_cell/.style={fill=gray!20, opacity=0.7},
            hatched_pattern/.style={pattern=north west lines, pattern color=gray!70},
            scale=0.6
        ]
        \def\gridmax{4}
        \pgfmathtruncatemacro{\imax}{\gridmax-1}
        \draw[step=1, black] (0,0) grid (\gridmax,\gridmax);
        
        \fill[hatched_cell] (1,0) rectangle (2,1);
        \fill[hatched_pattern] (1,0) rectangle (2,1);
        \fill[hatched_cell] (2,2) rectangle (3,3);
        \fill[hatched_pattern] (2,2) rectangle (3,3);
        
        \foreach \x in {0,...,\imax} {
            \foreach \y in {1,...,\imax} { \node[dot] at (\x+0.5, \y) {}; }
        }
        \foreach \x in {1,...,\imax} {
            \foreach \y in {0,...,\imax} { \node[dot] at (\x, \y+0.5) {}; }
        }
        
        \begin{scope}[yellow_square]
            \foreach \x/\y in {1/3.5, 3.5/3, 1/2.5, 2.5/2, 2.5/3, 1/1.5, 1/0.5, 2/0.5, 3.5/1}
            {
                \fill (\x-0.25, \y-0.25) rectangle (\x+0.25, \y+0.25);
                \node[circle, fill=orange, inner sep=0pt, minimum size=3pt] at (\x, \y) {};
            }
        \end{scope}
        \end{tikzpicture}
        \caption{Solution for Maximum Independent Set on a $4\times 4 \sim 2$ square layout.}
        \label{fig:MIS_result}
    \end{subfigure}

    \caption{The yellow areas represent selected resting sites $\sites$ for the movers out of the candidate positions depicted by the blue dots. Interface tiles (gray, hatched tiles) may contain up to 2 resting sites.}
\end{figure}

\subsection{Resting Sites Assignment}\label{routing:resting_sites_assignment}

The problem of assigning each mover's idle period to a resting site in the schedule $\schedule$ is formulated as a Mixed-Integer Programming model. The model takes as input the set of precomputed resting sites $\sites = \{r_1, r_2, \ldots, r_k\}$, the set of movers $\movers = \{m_1, m_2, \ldots, m_{\numMovers}\}$, and the set of \emph{transits} $\transits = \{w_1, w_2, \ldots, w_l\}$ derived from the schedule \schedule.

A \emph{transit} represents the time gap between two consecutive operations assigned to the same mover in the schedule. Among these, \emph{transits with idle periods} are those whose duration exceeds the travel time required to move between the corresponding operation locations (see Figure~\ref{fig:transits_w}). We populate set $\transits$ only with \emph{transits with idle periods}, as these represent the intervals during which a mover needs a resting site. The objective of the model is to assign each such transit to a resting site in a way that minimizes the additional travel cost for movers to reach and leave the resting sites.

\input{figs/rest_site_gen}

The problem is modeled using binary decision variables $x_{ij} \in \{0, 1\}$, where $x_{ij} = 1$ if transit $w_i \in \transits$ is assigned to resting site $r_j \in \sites$, and $c_{ij}$ denotes the round-trip travel time. The objective is to minimize $\sum_{i,j} c_{ij} x_{ij}$ while assigning each transit to a resting site: $\sum_{j \in \sites} x_{ij} = 1, \quad \forall i \in \transits$ and ensuring at most one mover uses a resting site at any given time: $x_{ij} + x_{i'j} \leq 1, \quad \forall (a_{ij}, a_{i'j}) \in F$. Set $F$ consists of pairs of operations with overlapping time windows.

\subsection{Conflict Resolution}\label{routing:conflicts}

Unlike standard MAPF collisions, our system allows movers to coexist on a tile. A \emph{conflict} occurs specifically when a mover transits through a tile where another is actively dispensing, forcing the operation to pause (see Figure~\ref{fig:routing:mover_rotation}, loading mover has to yield so there is space for transit). Our algorithm detects these interruptions and extends the schedule to ensure the required dispensing duration is met. 

We model the schedule as a directed acyclic graph (DAG) $\Omega = (K \cup \{k_{0}\}, E)$, where vertices $K$ represent operations and $k_{0}$ is a virtual source node. Edges $E$ represent start-to-start precedence constraints:
\begin{enumerate}
    \item Same-mover constraint: Connects sequential operations of a single mover. We include edges $(k_{0}, k)$ for the first operation $k$ of each mover with a weight equal to the start time of $k$. For all other mover edges, the weight $w(k,k') = l_{i} + \Delta^{g'}_{p'} +\| h' - h'' \|_1$ accounts for the interruption time $l_i$, dispensing duration $\Delta^{g'}_{p'}$, and travel time between locations $h'$ and $h''$.
    \item Same-dispenser constraint: Connects operations sharing a dispenser, ordered by their original start times, with a weight of 1.
\end{enumerate}

To adjust the schedule, we calculate the longest path from the source vertex to every vertex. Setting the start time of the source node $S(k_{0}) = 0$, the start time $S(k)$ for all other operations is determined recursively: $S(k) = \max_{(k',k)\in E} \{ S(k') + w(k',k) \}.$
The global makespan is then determined by $\max_{k \in K}\{S(k) + \Delta^{g'}_{p'} \}$. This process is iterative. We recalculate the interruptions and update the new interruption counts using $l_i = \max\{l_i, \text{current interruptions}\}$ and re-optimize the start times if any values have changed. If not, we have reached a conflict-free solution.

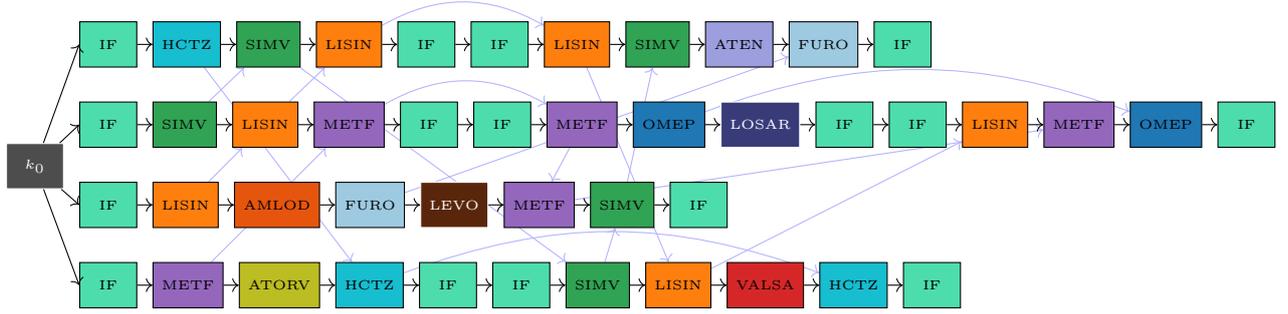
\begin{figure}[t]
	\centering
    \pgfdeclarelayer{background}
    \pgfdeclarelayer{foreground}
    \pgfsetlayers{background,main,foreground}
	\begin{tikzpicture}[
		drugnode/.style={rectangle, draw, minimum width=0.75cm, minimum height=0.6cm, font=\tiny},
        drugnode_l/.style={rectangle, draw, minimum width=0.75cm, minimum height=0.6cm, font=\tiny, white},
		every edge/.style={->, >=stealth, thick}
	]
		% Color definitions
		\definecolor{catOMEPRAZOLE}{HTML}{1F77B4}
		\definecolor{catLEVOTHYROXINE}{HTML}{59260b}
		\definecolor{catVALSARTAN}{HTML}{D62728}
		\definecolor{catMETFORMIN}{HTML}{9467BD}
		\definecolor{catLISINOPRIL}{HTML}{FF7F0E}
		\definecolor{catSIMVASTATIN}{HTML}{31A354}
		\definecolor{catHYDROCHLOROTHIAZIDE}{HTML}{17BECF}
		\definecolor{catLOSARTAN}{HTML}{393B79}
		\definecolor{catAMLODIPINE}{HTML}{E6550D}
		\definecolor{catATORVASTATIN}{HTML}{BCBD22}
		\definecolor{catATENOLOL}{HTML}{9C9EDE}
		\definecolor{catFUROSEMIDE}{HTML}{9ECAE1}
		\definecolor{catVIRTUAL}{HTML}{4D4D4D}
        \definecolor{catINTERFACE}{HTML}{4DDDAD}
        
		% Mover 0 nodes
		\node[drugnode, fill=catINTERFACE] (m0n0) at (0,0) {IF};
		\node[drugnode, fill=catHYDROCHLOROTHIAZIDE, right=0.2cm of m0n0] (m0n1) {HCTZ};
		\node[drugnode, fill=catSIMVASTATIN, right=0.2cm of m0n1] (m0n2) {SIMV};
		\node[drugnode, fill=catLISINOPRIL, right=0.2cm of m0n2] (m0n3) {LISIN};
		\node[drugnode, fill=catINTERFACE, right=0.2cm of m0n3] (m0n4) {IF};
		\node[drugnode, fill=catINTERFACE, right=0.2cm of m0n4] (m0n5) {IF};
		\node[drugnode, fill=catLISINOPRIL, right=0.2cm of m0n5] (m0n6) {LISIN};
		\node[drugnode, fill=catSIMVASTATIN, right=0.2cm of m0n6] (m0n7) {SIMV};
		\node[drugnode, fill=catATENOLOL, right=0.2cm of m0n7] (m0n8) {ATEN};
		\node[drugnode, fill=catFUROSEMIDE, right=0.2cm of m0n8] (m0n9) {FURO};
		\node[drugnode, fill=catINTERFACE, right=0.2cm of m0n9] (m0n10) {IF};

        \node[drugnode_l,fill=catVIRTUAL, below left = 1cm and 0.2cm of m0n0] (virt) {$k_0$};
        
		% Mover 1 nodes
		\node[drugnode, fill=catINTERFACE, below=0.45cm of m0n0] (m1n0) {IF};
		\node[drugnode, fill=catSIMVASTATIN, right=0.2cm of m1n0] (m1n1) {SIMV};
		\node[drugnode, fill=catLISINOPRIL, right=0.2cm of m1n1] (m1n2) {LISIN};
		\node[drugnode, fill=catMETFORMIN, right=0.2cm of m1n2] (m1n3) {METF};
		\node[drugnode, fill=catINTERFACE, right=0.2cm of m1n3] (m1n4) {IF};
		\node[drugnode, fill=catINTERFACE, right=0.2cm of m1n4] (m1n5) {IF};
		\node[drugnode, fill=catMETFORMIN, right=0.2cm of m1n5] (m1n6) {METF};
		\node[drugnode, fill=catOMEPRAZOLE, right=0.2cm of m1n6] (m1n7) {OMEP};
		\node[drugnode_l, fill=catLOSARTAN, right=0.2cm of m1n7] (m1n8) {LOSAR};
		\node[drugnode, fill=catINTERFACE, right=0.2cm of m1n8] (m1n9) {IF};
		\node[drugnode, fill=catINTERFACE, right=0.2cm of m1n9] (m1n10) {IF};
		\node[drugnode, fill=catLISINOPRIL, right=0.2cm of m1n10] (m1n11) {LISIN};
		\node[drugnode, fill=catMETFORMIN, right=0.2cm of m1n11] (m1n12) {METF};
		\node[drugnode, fill=catOMEPRAZOLE, right=0.2cm of m1n12] (m1n13) {OMEP};
		\node[drugnode, fill=catINTERFACE, right=0.2cm of m1n13] (m1n14) {IF};
		
		% Mover 2 nodes
		\node[drugnode, fill=catINTERFACE, below=0.45cm of m1n0] (m2n0) {IF};
		\node[drugnode, fill=catLISINOPRIL, right=0.2cm of m2n0] (m2n1) {LISIN};
		\node[drugnode, fill=catAMLODIPINE, right=0.2cm of m2n1] (m2n2) {AMLOD};
		\node[drugnode, fill=catFUROSEMIDE, right=0.2cm of m2n2] (m2n3) {FURO};
		\node[drugnode_l, fill=catLEVOTHYROXINE, right=0.2cm of m2n3] (m2n4) {LEVO};
		\node[drugnode, fill=catMETFORMIN, right=0.2cm of m2n4] (m2n5) {METF};
		\node[drugnode, fill=catSIMVASTATIN, right=0.2cm of m2n5] (m2n6) {SIMV};
		\node[drugnode, fill=catINTERFACE, right=0.2cm of m2n6] (m2n7) {IF};
		
		% Mover 3 nodes
		\node[drugnode, fill=catINTERFACE, below=0.45cm of m2n0] (m3n0) {IF};
		\node[drugnode, fill=catMETFORMIN, right=0.2cm of m3n0] (m3n1) {METF};
		\node[drugnode, fill=catATORVASTATIN, right=0.2cm of m3n1] (m3n2) {ATORV};
		\node[drugnode, fill=catHYDROCHLOROTHIAZIDE, right=0.2cm of m3n2] (m3n3) {HCTZ};
		\node[drugnode, fill=catINTERFACE, right=0.2cm of m3n3] (m3n4) {IF};
		\node[drugnode, fill=catINTERFACE, right=0.2cm of m3n4] (m3n5) {IF};
		\node[drugnode, fill=catSIMVASTATIN, right=0.2cm of m3n5] (m3n6) {SIMV};
		\node[drugnode, fill=catLISINOPRIL, right=0.2cm of m3n6] (m3n7) {LISIN};
		\node[drugnode, fill=catVALSARTAN, right=0.2cm of m3n7] (m3n8) {VALSA};
		\node[drugnode, fill=catHYDROCHLOROTHIAZIDE, right=0.2cm of m3n8] (m3n9) {HCTZ};
		\node[drugnode, fill=catINTERFACE, right=0.2cm of m3n9] (m3n10) {IF};

         % Virtual edges
        \draw[->] (virt) -- (m0n0.west);
        \draw[->] (virt) -- (m1n0.west);
        \draw[->] (virt) -- (m2n0.west);
        \draw[->] (virt) -- (m3n0.west);
        
		% Mover 0 edges
		\draw[->] (m0n0) -- (m0n1);
		\draw[->] (m0n1) -- (m0n2);
		\draw[->] (m0n2) -- (m0n3);
		\draw[->] (m0n3) -- (m0n4);
		\draw[->] (m0n4) -- (m0n5);
		\draw[->] (m0n5) -- (m0n6);
		\draw[->] (m0n6) -- (m0n7);
		\draw[->] (m0n7) -- (m0n8);
		\draw[->] (m0n8) -- (m0n9);
		\draw[->] (m0n9) -- (m0n10);
		
		% Mover 1 edges
		\draw[->] (m1n0) -- (m1n1);
		\draw[->] (m1n1) -- (m1n2);
		\draw[->] (m1n2) -- (m1n3);
		\draw[->] (m1n3) -- (m1n4);
		\draw[->] (m1n4) -- (m1n5);
		\draw[->] (m1n5) -- (m1n6);
		\draw[->] (m1n6) -- (m1n7);
		\draw[->] (m1n7) -- (m1n8);
		\draw[->] (m1n8) -- (m1n9);
		\draw[->] (m1n9) -- (m1n10);
		\draw[->] (m1n10) -- (m1n11);
		\draw[->] (m1n11) -- (m1n12);
		\draw[->] (m1n12) -- (m1n13);
		\draw[->] (m1n13) -- (m1n14);
		
		% Mover 2 edges
		\draw[->] (m2n0) -- (m2n1);
		\draw[->] (m2n1) -- (m2n2);
		\draw[->] (m2n2) -- (m2n3);
		\draw[->] (m2n3) -- (m2n4);
		\draw[->] (m2n4) -- (m2n5);
		\draw[->] (m2n5) -- (m2n6);
		\draw[->] (m2n6) -- (m2n7);
		
		% Mover 3 edges
		\draw[->] (m3n0) -- (m3n1);
		\draw[->] (m3n1) -- (m3n2);
		\draw[->] (m3n2) -- (m3n3);
		\draw[->] (m3n3) -- (m3n4);
		\draw[->] (m3n4) -- (m3n5);
		\draw[->] (m3n5) -- (m3n6);
		\draw[->] (m3n6) -- (m3n7);
		\draw[->] (m3n7) -- (m3n8);
		\draw[->] (m3n8) -- (m3n9);
		\draw[->] (m3n9) -- (m3n10);

        \begin{pgfonlayer}{background}
    		% HYDROCHLOROTHIAZIDE: m0n1 -> m3n3 -> m3n9
    		\draw[->, color=blue!30] (m0n1) -- (m3n3);
    		\draw[->, color=blue!30,bend left=20] (m3n3) to (m3n9);
    		
    		% SIMVASTATIN: m1n1 -> m0n2 -> m2n6 -> m0n7 -> m3n6
    		\draw[->, color=blue!30] (m1n1) -- (m0n2);
    		\draw[->, color=blue!30] (m0n2) -- (m3n6);
    		\draw[->, color=blue!30] (m3n6) -- (m2n6);
    		\draw[->, color=blue!30] (m2n6) -- (m0n7);
    		
    		% LISINOPRIL: m2n1 -> m0n3 -> m1n2 -> m0n6 -> m3n7 -> m1n11
    		\draw[->, color=blue!30] (m2n1) -- (m1n2);
    		\draw[->, color=blue!30] (m1n2) -- (m0n3);
    		\draw[->, color=blue!30, bend left] (m0n3) to (m0n6);
    		\draw[->, color=blue!30] (m0n6) -- (m3n7);
    		\draw[->, color=blue!30] (m3n7) -- (m1n11);
    		
    		% METFORMIN: m3n1 -> m1n3 -> m1n6 -> m2n5 -> m1n12
            \draw[->, color=blue!30] (m3n1) -- (m1n3);
    		\draw[->, color=blue!30, bend left] (m1n3) to (m1n6);
    		\draw[->, color=blue!30] (m1n6) -- (m2n5);
    		\draw[->, color=blue!30] (m2n5) -- (m1n12);
    		
    		% FUROSEMIDE: m2n3 -> m0n9
    		\draw[->, color=blue!30] (m2n3) -- (m0n9);
    		
    		% OMEPRAZOLE: m1n7 -> m1n13
    		\draw[->, color=blue!30, bend left=20] (m1n7) to (m1n13);
		\end{pgfonlayer}

	\end{tikzpicture}
    \caption[Schedule-to-graph DAG]{DAG generated from the schedule in Figure~\ref{fig:mainExampleSchedule}. Black edges represent same-mover precedence constraints, blue edges represent same-dispenser precedence constraints. Task colors indicate dispenser assignment.}
    \label{fig:complex_dag_schedule}
\end{figure}

\subsection{DAG-based Schedule Merging}\label{batch-heuristic-merging}
For larger order sets, it might be useful to split the set into smaller batches randomly and independently generate sub-schedules for each batch. These sub-schedules are then transformed into DAGs using the same principle and connected together. Specifically, edges are added from the terminal vertices of the preceding DAG to the initial vertices of the new sub-schedule. Edge weights correspond to the minimum required transit times between the last and first tasks of adjacent batches. Additionally, edges are added between the last tasks in the earlier schedule and the first tasks in the new schedule that use the same tile, ensuring potential dispensing conflicts are accounted for. The start times of the new intervals are then initialized by propagating earliest start times through the extended graph, computed via LP or longest-path relaxation over the merged DAG.

\section{Experiments}\label{sec:experiments}
The proposed algorithms allow investors to evaluate potential facility configurations — including tile layouts, dispenser placements, and mover counts — and assess whether a given investment yields sufficient production throughput. This section pursues two aims: validating the proposed methodology and demonstrating that the decomposition introduces no significant loss in solution quality relative to the theoretical lower bound (Section~\ref{exp:lb}). At the tactical level, we evaluate robustness (Section~\ref{exp:robst}) and analyze how tactical decisions propagate to the operational level (Section~\ref{sec:experiments-correlation}). At the operational level, we examine the effect of mover count and time limits on makespan (Sections~\ref{exp:mvc_cmax} and~\ref{expt:tml_cmax}), assess the scalability of DAG-based schedule merging and its effect on solution quality for larger problem instances (Section~\ref{exp:batching}), and quantify the impact of the routing phase on overall system performance (Section~\ref{exp:rting}).

As a basis for our study, we applied data from an analysis of real-world co-prescription patterns within the NHANES drug survey dataset \citep{nchs}. The analysis provided the marginal probabilities and correlation matrix (discussed in \textsf{Packing}, illustrated in Figure B.1 for the top 40 most frequently prescribed drugs. To reflect sufficient system complexity, simulated orders were restricted to contain between 3 and 8 unique medications.

To generate realistic, typical order-sets, we implemented a synthetic data generator that produces multivariate binary patient prescriptions. This system is based on the methodology proposed by \cite{data_generator_bindata}, which transforms samples drawn from a multivariate normal distribution into binary outcomes via marginal thresholding.
The generated instances can be found on GitHub \textsc{[link removed while under review]}. 

\subsection{Hardware and Software Environment}
All computational experiments were conducted on a server equipped with dual Intel Xeon Silver 4110 CPUs @ 2.10GHz, providing a total of 16 physical cores and 32 logical threads. The system is provisioned with 187 GB of RAM. The operating environment is Debian GNU/Linux 13 (trixie).

The algorithms were implemented in Python 3. Mathematical optimization models were solved using commercial solvers, with the environment configured to support Gurobi 12.0.0 and CPLEX 22.1.1.0.
To ensure reproducibility, we provide a summary of the parameters used for our algorithms in Table~\ref{tab:experimental_settings}. The dispensing speed parameter we derived from \cite{sonntag2023method}. As future works are primarily aimed at speeding up the dispensing process, we also evaluated the system's performance with faster dispenser speeds.
\begin{table}[ht]
    \centering
    \begin{tabular}{ll ll}
        \toprule
        \multicolumn{2}{c}{Tactical Level (Placement \& Packing)} & 
        \multicolumn{2}{c}{Operational Level (Scheduling, Routing)} \\
        \cmidrule(r){1-2} \cmidrule(l){3-4} 
        GA Population Size & 150 & 
        Time Limit (Standard) & 600s \\
        GA Max Evaluations & 50,000 & 
        Time Limit (Scalability) & 300s -- 2,000s \\
        GA Episodes & 20 & 
        Standard Dispensing Speed & 100 ticks / cartridge \\
        Dispenser Count & 82 (Fig.~\ref{fig:robustness}), 62 (Fig.~\ref{fig:placement-sched-correlation}) & 
        Fast Dispensing Speed & 10 ticks / cartridge \\
        \bottomrule
    \end{tabular}
    \caption{Summary of experimental parameters.}
    \label{tab:experimental_settings}
\end{table}
\FloatBarrier

\subsection{Robustness of Tactical Level}\label{exp:robst}

To ensure that the placement and packing phases do not suffer from overfitting, where the resulting layout is optimized for a specific sequence of orders and, therefore, performs poorly on others, we evaluated the robustness of the solution. 
The packing and placement phases were run using a training set of 100 generated orders to create a fixed $8 \times 8$ layout. The packing configuration consisted of 82 dispensers, optimized by a genetic algorithm with a population size of 150 and a maximum of 50,000 evaluations. 
Subsequently, to test the robustness of the resulting layout, we executed the scheduling and routing phases using 100 distinct datasets: the original training set and 99 new, unseen validation sets, each consisting of 100 orders. 
For these validation runs, the scheduler used a warm-start initialization and was constrained to a 600-second time limit.
\begin{figure}[ht]
  \centering
  \begin{subfigure}[t]{0.3\textwidth}
    \centering
    \begin{tikzpicture}
      \begin{axis}[
        ybar,
        xlabel={relative difference [\%]},
        ylabel={count [-]},
        xmin=-10, xmax=10,
        ymin=0,ymax=20,
        width=1\linewidth,
        height=0.9\textwidth,
        ymajorgrids=true
      ]
        \addplot+[hist={bins=15, data min=-10, data max=10},fill opacity=0.6] table[y=deviation_pct] {experiments/robustness/deviation_vals_8.csv};
      \end{axis}
    \end{tikzpicture}
    \caption{Relative makespan deviation from the mean across 100 runs, for a layout with 8 movers.}
  \end{subfigure}
  \begin{subfigure}[t]{0.3\textwidth}
    \centering
    \begin{tikzpicture}
      \begin{axis}[
        ybar,
        xlabel={relative difference [\%]},
        xmin=-10, xmax=10,
        ymin=0,ymax=20,
        width=\linewidth,
       height=0.9\textwidth,
       ymajorgrids=true
      ]
        \addplot+[hist={bins=15, data min=-10, data max=10},fill opacity=0.6] table[y=deviation_pct] {experiments/robustness/deviation_vals_12.csv};
      \end{axis}
    \end{tikzpicture}
    \caption{Relative makespan deviation from the mean across 100 runs, for a layout with 12 movers.}
  \end{subfigure}
  \begin{subfigure}[t]{0.38\textwidth}
      \begin{tikzpicture}
    \begin{axis}[
        view={120}{40},
        width=165pt,
        height=165pt,
        grid=major,
        z buffer=sort,
        xmin=0,xmax=17,
        ymin=0,ymax=17,
        zmin=0,zmax=3,
        enlargelimits=upper,
        ztick=\empty,
        % Set limits wide enough to show all combined geometries
        % xmin=-32, xmax=32, 
        % ymin=-15, ymax=15,
        % zmin=0, zmax=3,
        grid=none, % Turned off the major grid so it doesn't clutter the combined plot
        z buffer=sort,
        enlargelimits=upper,
        xlabel={$y$},
        ylabel={$x$},
        ]
        
        % A common style for all our blocks to keep code clean
        \pgfplotsset{
            cube style/.style={
                only marks,
                scatter,
                mark=cube*,
                mark size=3.5pt, % Exactly half of the x/y unit (10pt) to close all gaps
                visualization depends on={value \thisrow{r} \as \myr},
                visualization depends on={value \thisrow{g} \as \myg},
                visualization depends on={value \thisrow{b} \as \myb},
                scatter/@pre marker code/.append code={
                    \definecolor{markcolor}{rgb}{\myr,\myg,\myb}
                    % draw=black!70 makes the edges highly visible!
                    \scope[draw=black!70, line width=0.2pt, fill=markcolor!85]
                },
                scatter/@post marker code/.append code={
                    \endscope
                }
            }
        }

        % Plot all 4 topologies into the same axis!
        \addplot3 [cube style] table [x=x, y=y, z=z] {interface_stability/heat_square.txt};
        \addplot3 [cube style] table [x=x, y=y, z=z] {interface_stability/heat_ring.txt};
        \addplot3 [cube style] table [x=x, y=y, z=z] {interface_stability/heat_doubleline.txt};
        \addplot3 [cube style] table [x=x, y=y, z=z] {interface_stability/heat_line.txt};

    \end{axis}
\end{tikzpicture}
\caption{Interface occurrence heatmap with 2 interfaces in each layout type. From top to bottom, the layouts doubleline, line, ring, and square are displayed.}\label{fig:interface-stability} 
  \end{subfigure}

  \caption{Robustness and stability analysis of the tactical level across varying mover configurations and layout topologies.}
  \label{fig:robustness}
  % \tn{todo recompute ring data}
\end{figure}
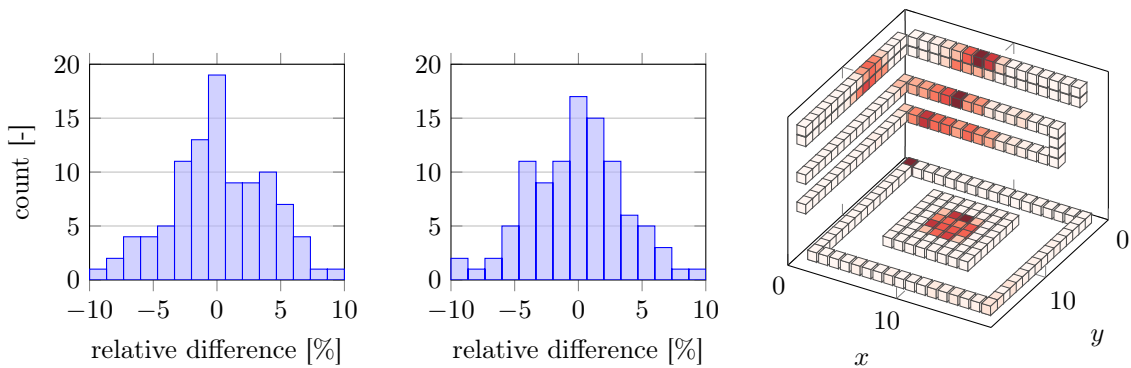
The results in Figure~\ref{fig:robustness} demonstrate a high degree of stability across the different datasets. 
The makespan ($C_{\max}$) observed for the unseen validation sets deviates by no more than $\pm 10\%$ from the result obtained with the training set. This narrow margin indicates that the layout generated by the placement and packing phases is robust and not over-fitted to the particular historic set $\ordersHistory$.

As a further experiment to assess the robustness of the tactical level, the stability of the interface placement is assessed.
The conducted experiment consisted of 100 independent placement optimizations with generated sets mirroring the medicine distribution of historic order \ordersHistory~on line, ring, doubleline, and square
layouts to identify the most occurring interface locations when 2 interfaces are used.
A stable placement should put interface tiles into similar positions when repeated with similar inputs.
The results in Figure~\ref{fig:interface-stability} show the interface heatmap for all 4 considered layouts, with color indicating the frequency of interface placement.
Due to the rotation symmetry of the ring topology, the resulting layouts were canonized so that their first interface is placed at the same coordinates.
The results suggest a strong preference for clustering interfaces at the geometric center of the grid, while the peripheral tiles were used almost exclusively for dispensers. 
While this configuration minimizes theoretical mover travel times, it may pose practical engineering challenges regarding the physical accessibility of the interface for cartridge loading and unloading.

\FloatBarrier

\subsection{Correlation of Placement and Scheduling Objectives}
\label{sec:experiments-correlation}
To validate the link between the proposed objective function for \textsf{Placement} and \textsf{Scheduling}, we analyzed the correlation between the objective values \eqref{eq:placement_problem} obtained during placement optimization via GA and the resulting final scheduling makespan ($C_{\max}$). 
The experimental setup used layouts with 48 tiles (47 for the square topology) and 2 interfaces, with 5 movers fulfilling 50 orders. The scheduling phase was conducted without warm-start initialization, constrained by a 600-second time limit across 5 CPU threads, with a packing configuration of 62 dispensers on the \textit{fast dispensing} setting (see Table~\ref{tab:experimental_settings}).

\begin{figure}[ht]
    \centering
    \begin{tikzpicture}
    \begin{axis}[
        height=0.21\textwidth,
        width=0.7\textwidth,
        grid=both,
        xlabel={placement objective [-]},
        ylabel={$C_{max}$ [-]},
        legend pos={outer north east},
        legend cell align={left}
    ]
        \addplot [
            scatter,
            only marks,
            point meta=explicit symbolic,
            mark options={xscale=0.8, yscale=0.8},
            scatter/classes={
                line={blue!40},
                doubleline={red!40},
                ring={green!40},
                square={brown!50!black!70}% <-- don't add comma
            }
        ] table[meta=topology,x=placer,y=cmax,col sep=comma] {experiments/topology_convergence.csv};

        \addplot [no markers, thick, gray!90] table [col sep=comma,x=placer,y={create col/linear regression={y=cmax}}] {experiments/topology_convergence.csv};
        \legend{Line, Doubleline, Ring, Square}
    \end{axis}
    \end{tikzpicture}
    \caption{Correlation of placement and scheduling objective.}
    \label{fig:placement-sched-correlation}
\end{figure}
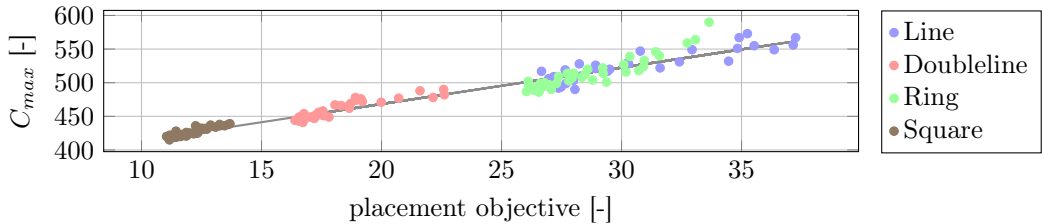 

The results, illustrated in Figure~\ref{fig:placement-sched-correlation}, demonstrate the progression of the GA across different topologies. Each data point represents a placement solution evaluated during the \textsf{Placement} optimization. 
The scatter plot reveals a linear relationship: as the GA minimizes the placement objective \eqref{eq:placement_problem} through successive generations, the corresponding $C_{\max}$ decreases proportionally. 

This consistent correlation confirms that the placement objective \eqref{eq:placement_problem} serves as a reasonable proxy for the actual scheduling. 
By optimizing for this proxy, \textsf{Placement} effectively drives the search toward configurations that minimize the final execution time. Furthermore, the model correctly identifies the performance potential of different structural layouts, accurately placing the square configuration in the highest-efficiency region (lowest $C_{\max}$) and simultaneously highlighting the inherent physical constraints and higher costs associated with the ring and line topologies.

\FloatBarrier

\subsection{Effect of Mover Count on \texorpdfstring{$C_{max}$}{Makespan}}\label{exp:mvc_cmax}

Figure~\ref{fig:routing_cmax_effect_combined} illustrates the makespan trend as the number of movers increases across the four considered layouts. Initially, adding more movers significantly reduces the makespan, since operations can be distributed more evenly across the system. However, the improvement rate diminishes as the number of movers grows. This plateau effect is explained by the limited number of available dispensers per drug type; adding additional movers beyond a certain point does not yield further speed-up, as the dispensers themselves become the throughput bottleneck. The exact makespan values for each configuration are reported in Tables~\ref{table:expRouting:cmax_d10},~\ref{table:expRouting:cmax_d1}.

\begin{figure}[htbp]
  \centering
  \begin{subfigure}{0.495\textwidth}
    \centering
    \begin{tikzpicture}
      \begin{axis}[
          height=0.6\textwidth,
          width=\textwidth,
          xlabel={number of movers [-]},
          ylabel={$C_\text{max}$ [-]},
          grid=both,
          bar width=3.3pt,
          ybar=0pt,
          xmin=0.5,xmax=12.5,
          xtick distance=1,
          scaled y ticks = false,
      ]
       
      \addplot+ plot table [x=movers, y=cmax, col sep=comma] {experiments/movers_cmax_10/line.csv};
      \addplot+ plot table [x=movers, y=cmax, col sep=comma] {experiments/movers_cmax_10/ring.csv};
      \addplot+ plot table [x=movers, y=cmax, col sep=comma] {experiments/movers_cmax_10/doubleline.csv};
      \addplot+ plot table [x=movers, y=cmax, col sep=comma] {experiments/movers_cmax_10/square.csv};
      \legend{Line, Ring, Doubleline, Square}
      \end{axis}
    \end{tikzpicture}
    \caption{Standard dispensing speed.}
    \label{fig:routing_cmax_effect_moderate}
  \end{subfigure}~~
  \begin{subfigure}{0.495\textwidth}
    \centering
    \begin{tikzpicture}
      \begin{axis}[
          height=0.6\textwidth,
          width=\textwidth,
          xlabel={number of movers [-]},
          grid=both,
          bar width=3.3pt,
          ybar=0pt,
          xmin=0.5,xmax=12.5,
          xtick distance=1,
          scaled y ticks = false,
      ]
       
      \addplot+ plot table [x=movers, y=cmax, col sep=comma] {experiments/movers_cmax_1/line.csv};
      \addplot+ plot table [x=movers, y=cmax, col sep=comma] {experiments/movers_cmax_1/ring.csv};
      \addplot+ plot table [x=movers, y=cmax, col sep=comma] {experiments/movers_cmax_1/doubleline.csv};
      \addplot+ plot table [x=movers, y=cmax, col sep=comma] {experiments/movers_cmax_1/square.csv};
      \legend{Line, Ring, Doubleline, Square}
      \end{axis}
    \end{tikzpicture}
    \caption{Fast dispensing speed.}
    \label{fig:routing_cmax_effect_fast}
  \end{subfigure}
  
  \caption{Effect of mover count on $C_{max}$ for different topologies (64 tiles each) with varying dispensing speeds.}
  \label{fig:routing_cmax_effect_combined}
\end{figure}
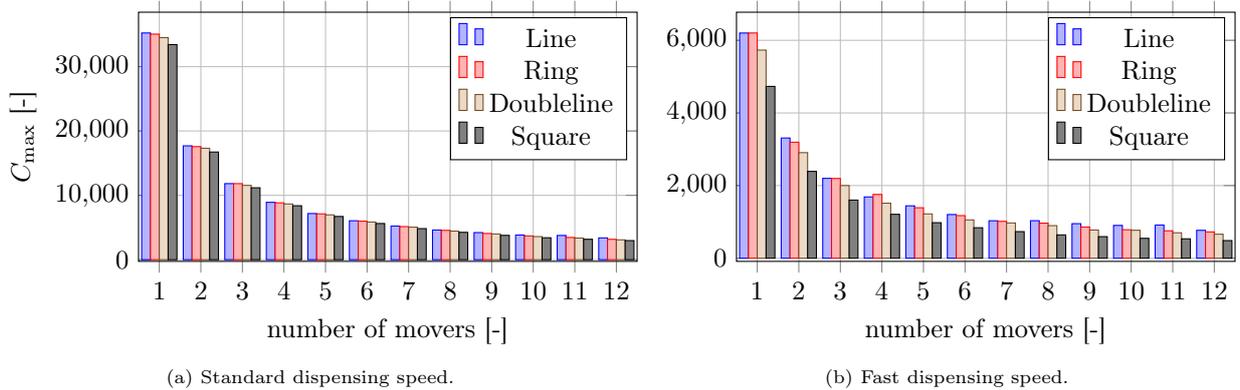

\subsection{Effect of Time Limit on \texorpdfstring{$C_{max}$}{Makespan}}\label{expt:tml_cmax}

The scheduling problem grows exponentially in complexity with the number of orders (if the batching mechanism detailed in Subsection~\ref{batch-heuristic-merging} is not used), making exact optimization infeasible for even moderately sized instances. The aim is therefore to obtain high-quality solutions within practical time limits, raising the question of when to terminate the search.

Figure~\ref{fig:timeexp_be100_m8} evaluates this trade-off by running the scheduler on order sets of sizes 10, 25, 100, and 250 under varying time limits, while Figure~\ref{fig:timeexp_ab100_m8} extends the analysis to sets of 200, 300, and 500 orders.

 \pgfplotsset{
    my distinct colors/.style={
        cycle list={
            {blue, mark=*},               % p10
            {red, mark=square*},          % p20
            {green!60!black, mark=triangle*}, % p50
            {orange, mark=diamond*},      % p100
            {violet, mark=pentagon*},     % p200
            {cyan, mark=otimes*},         % p300
            {magenta, mark=star}          % p500
        }
    }
} 
\begin{figure}[ht]
    \centering
    % --- Left Subfigure ---
    \begin{subfigure}[t]{0.36\textwidth}
        \centering
        \begin{tikzpicture}[baseline=(current axis.north)]
            \begin{axis}[
                height=4.5cm,
                width=\linewidth,
                xlabel={time [s]},
                ylabel={$C_\text{max}$ [-]},
                ymin=0, ymax=7000,    
                grid=both,
                my distinct colors
            ]
                \addplot+ plot table [x=timelimit, y=Cmax, col sep=comma] {experiments/timeexp/timeexp_p10_m8.csv};
                \addplot+ plot table [x=timelimit, y=Cmax, col sep=comma] {experiments/timeexp/timeexp_p20_m8.csv};
                \addplot+ plot table [x=timelimit, y=Cmax, col sep=comma] {experiments/timeexp/timeexp_p50_m8.csv};
                \addplot+ plot table [x=timelimit, y=Cmax, col sep=comma] {experiments/timeexp/timeexp_p100_m8.csv};
            \end{axis}
        \end{tikzpicture}
        \phantomcaption
        \label{fig:timeexp_be100_m8}
    \end{subfigure}
    \hfill
    % --- Right Subfigure ---
    \begin{subfigure}[t]{0.36\textwidth}
        \centering
        \begin{tikzpicture}[baseline=(current axis.north)]
            \begin{axis}[
                height=4.5cm,
                width=\linewidth,
                scaled ticks=false,
                xlabel={time [s]},
                ylabel={},
                ymin=0, ymax=35000,    
                grid=both,
                my distinct colors, 
                cycle list shift=4, 
            ]
                \addplot+ plot table [x=timelimit, y=Cmax, col sep=comma] {experiments/timeexp/timeexp_p200_m8.csv};
                \addplot+ plot table [x=timelimit, y=Cmax, col sep=comma] {experiments/timeexp/timeexp_p300_m8.csv};
                \addplot+ plot table [x=timelimit, y=Cmax, col sep=comma] {experiments/timeexp/timeexp_p500_m8.csv};
            \end{axis}
        \end{tikzpicture}
        \phantomcaption
        \label{fig:timeexp_ab100_m8}
    \end{subfigure}
    \hfill
    % --- Two-Column Legend aligned to the top ---
    \begin{subfigure}[t]{0.24\textwidth}
        \centering
        \begin{tikzpicture}[baseline=(current bounding box.north)]
            \matrix [
                matrix of nodes,
                draw,
                inner sep=0.15cm,
                column sep=1mm,
                row sep=1mm,
                nodes={anchor=mid west, font=\small}
            ] {
                $|\mathcal{P}|:$ & 
                \tikz[baseline=-0.65ex]{\draw[blue, thick] (0,0) -- (0.4,0) plot[mark=*, mark size=1.5pt] coordinates {(0.2,0)};} 10 \\ 
                \tikz[baseline=-0.65ex]{\draw[red, thick] (0,0) -- (0.4,0) plot[mark=square*, mark size=1.5pt] coordinates {(0.2,0)};} 20 &
                \tikz[baseline=-0.65ex]{\draw[green!60!black, thick] (0,0) -- (0.4,0) plot[mark=triangle*, mark size=1.5pt] coordinates {(0.2,0)};} 50 \\ 
                \tikz[baseline=-0.65ex]{\draw[orange, thick] (0,0) -- (0.4,0) plot[mark=diamond*, mark size=1.5pt] coordinates {(0.2,0)};} 100 &
                \tikz[baseline=-0.65ex]{\draw[violet, thick] (0,0) -- (0.4,0) plot[mark=pentagon*, mark size=1.5pt] coordinates {(0.2,0)};} 200 \\ 
                \tikz[baseline=-0.65ex]{\draw[cyan, thick] (0,0) -- (0.4,0) plot[mark=otimes*, mark size=1.5pt] coordinates {(0.2,0)};} 300 &
                \tikz[baseline=-0.65ex]{\draw[magenta, thick] (0,0) -- (0.4,0) plot[mark=star, mark size=1.5pt] coordinates {(0.2,0)};} 500 \\
            };
        \end{tikzpicture}
    \end{subfigure}
    
    \caption{The effect of the time limit on $C_\text{max}$ for different sizes of $\mathcal{P}$ on a square $8\times8\sim2$ topology.}
    \label{fig:timeexp_combined_m8}
\end{figure}
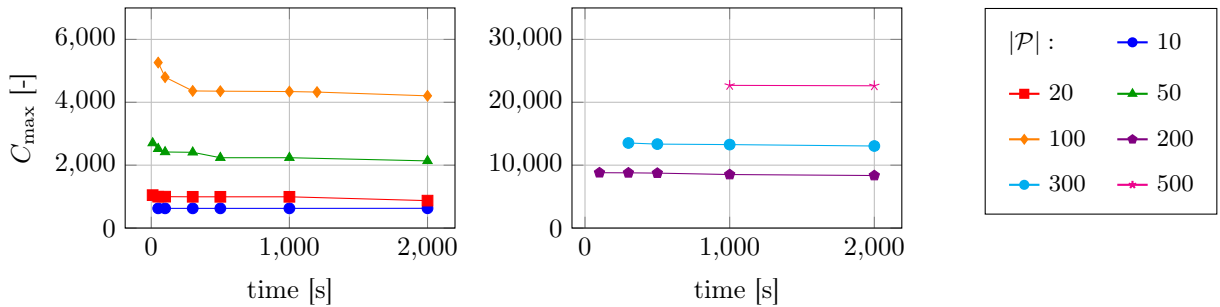
 The results for order sets with 100 or fewer orders show diminishing improvements in makespan with longer run times. Time limits of around 300 seconds appear to be sufficient for practical reasons across all order set sizes.

Figure~\ref{fig:timeexp_ab100_m8} illustrates the impact of instance size on problem complexity. For larger order sets, even identifying a feasible solution requires significantly more time, and subsequent improvements become negligible within any reasonable time limit. This behavior stems from the exponential growth of the solution space that the solver must explore.
\subsection{Comparison Between \texorpdfstring{$C_{max}$}{Makespan}and Lower Bound}\label{exp:lb}

Because the previous analysis focused on identifying practical time limits by comparing makespan across different solver runs, it only provided a relative notion of solution quality. To put these results into a more absolute context, we would ideally compare the obtained schedules against the theoretical optimal makespan. Since computing the exact optimum is computationally intractable, we instead use the lower bound introduced in Section~\ref{sec:lowerbound}. As this bound is derived by relaxing certain problem constraints, the true optimum must lie strictly above it. Consequently, the reported gaps in Table~\ref{table:expFour:compar_wide} are conservative, and the actual distance of our solutions from optimality is even smaller than the numbers suggest.

\begin{table}[htbp]
    \centering
    \setlength{\tabcolsep}{3pt} % Reduces white space between columns to help it fit
    \small % Slightly smaller font to accommodate the width
    \begin{tabular}{c|cccc|cccc|cccc}
        \toprule
        \multirow{2}{*}{\begin{tabular}{@{}c@{}}Mover\\Count\end{tabular}} & 
        \multicolumn{4}{c|}{Order Set Size: 25} & 
        \multicolumn{4}{c|}{Order Set Size: 50} & 
        \multicolumn{4}{c}{Order Set Size: 100} \\
        \cmidrule{2-13}
         & LB & $C_{max}$ & Gap & Gap [\%] & LB & $C_{max}$ & Gap & Gap [\%] & LB & $C_{max}$ & Gap & Gap [\%] \\
        \midrule
        
        2  & 3804 & 4173 & 369 & 8.84\%  & 7541 & 8518 & 977 & 11.47\% & 15304 & 16727 & 1423 & 8.51\% \\
        
        6  & 1270 & 1577 & 307 & 19.47\% & 2514 & 3008 & 494 & 16.42\% & 5101  & 5879  & 778  & 13.23\% \\
        
        10 & 767  & 1010 & 243 & 24.06\% & 1509 & 1896 & 387 & 20.41\% & 3063  & 3692  & 629  & 17.04\% \\
        
        12 & --   & --   & --  & --      & --   & --   & --  & --      & 2555  & 3416  & 861  & 25.20\% \\
        
        \bottomrule
    \end{tabular}
    \caption{Comparison of lower bound vs. actual makespan across different order set sizes ($8\times8\sim2$ layout).}
    \label{table:expFour:compar_wide}
\end{table}

The results show that the $C_{max}$ of our solutions is consistently close to the lower bound, with gaps generally remaining in a moderate range. For larger order sets combined with many movers, the differences increase more noticeably, reinforcing that problem complexity at this scale begins to limit how close the solver can remain to the optimum.

\FloatBarrier
\subsection{Achieved Scalability: Effect of DAG-based Schedule Merging}\label{exp:batching}

Subsection~\ref{batch-heuristic-merging} discusses batching larger order sets and then merging them together. To measure the impact this technique has on the overall $C_{max}$, we test different batch sizes for varying numbers of orders while keeping the number of movers fixed at eight. Each curve shows the evolution of the makespan ($C_{max}$) over time as the solver continues to improve the solution.

\pgfplotsset{
    % 1. Define specific styles for each batch size
    batch50/.style={blue, mark=*},
    batch100/.style={red, mark=square*},
    batch200/.style={green!60!black, mark=triangle*},
    batch300/.style={orange, mark=diamond*},
    batch500/.style={violet, mark=pentagon*},
    % 2. Common axis style for all plots
    pltstyle/.style={
        width=\linewidth,
        height=3.8cm,
        grid=both,
        xlabel={time [s]},
        ylabel={$C_\text{max}$},
        scaled y ticks=false,
        yticklabel style={
            /pgf/number format/fixed,
            /pgf/number format/precision=0,
            text width=3.5em,   
            align=right       
        },
        legend style={font=\footnotesize},
        xlabel style={font=\small},
        ylabel style={font=\small},
    }
}
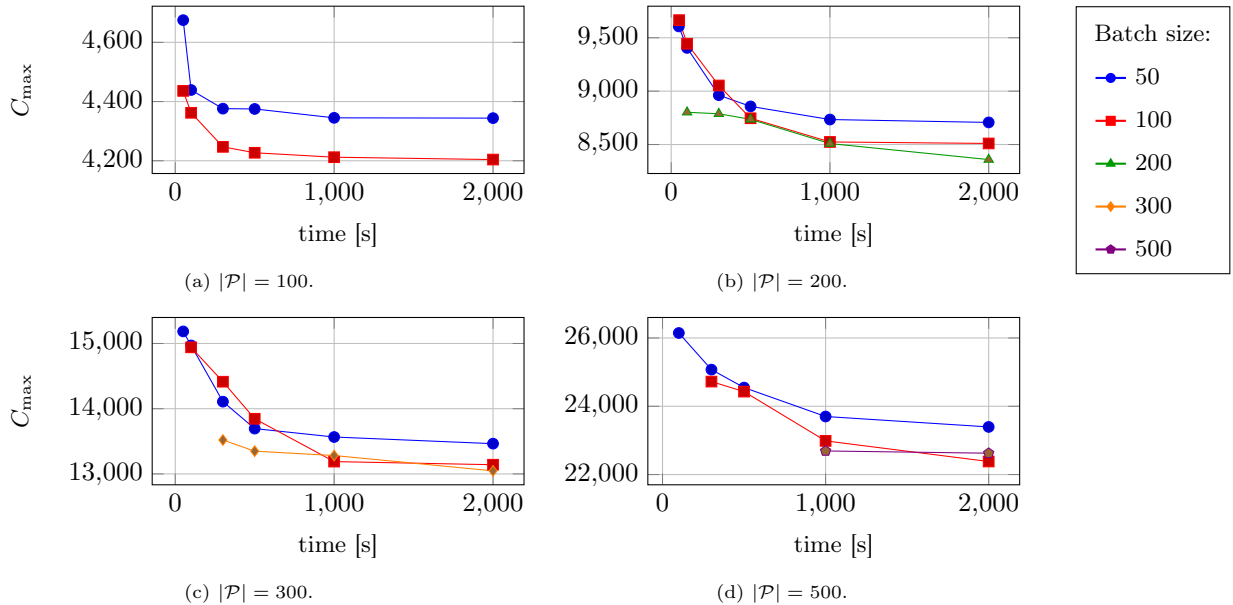
\begin{figure}[!htb]
    \centering
    % --- Left side: 2x2 Grid of Plots ---
    \begin{minipage}[t]{0.82\textwidth}
        \centering
        % --- Row 1 ---
        \begin{subfigure}[t]{0.48\textwidth}
            \centering
            \begin{tikzpicture}[baseline=(current axis.north)]
                \begin{axis}[pltstyle]
                    \addplot+[batch50]  table [x=timelimit,y=Cmax, col sep=comma] {experiments/batch_size_heuristics/tikz_patients100_batch50.csv}; 
                    \addplot+[batch100] table [x=timelimit,y=Cmax, col sep=comma] {experiments/batch_size_heuristics/tikz_patients100_batch100.csv}; 
                \end{axis}
            \end{tikzpicture}
            \caption{$|\orders|=100$.}
            \label{subfig:orders100}
        \end{subfigure}
        \hfill
        \begin{subfigure}[t]{0.48\textwidth}
            \centering
            \begin{tikzpicture}[baseline=(current axis.north)]
                \begin{axis}[pltstyle,ylabel={}]
                    \addplot+[batch50]  table [x=timelimit,y=Cmax, col sep=comma] {experiments/batch_size_heuristics/tikz_patients200_batch50.csv}; 
                    \addplot+[batch100] table [x=timelimit,y=Cmax, col sep=comma] {experiments/batch_size_heuristics/tikz_patients200_batch100.csv}; 
                    \addplot+[batch200] table [x=timelimit,y=Cmax, col sep=comma] {experiments/batch_size_heuristics/tikz_patients200_batch200.csv}; 
                \end{axis}
            \end{tikzpicture}
            \caption{$|\orders|=200$.}
            \label{subfig:orders200}
        \end{subfigure}
    
        \vspace{0.8em}
    
        % --- Row 2 ---
        \begin{subfigure}[t]{0.48\textwidth}
            \centering
            \begin{tikzpicture}[baseline=(current axis.north)]
                \begin{axis}[pltstyle]
                    \addplot+[batch50]  table [x=timelimit,y=Cmax, col sep=comma] {experiments/batch_size_heuristics/tikz_patients300_batch50.csv}; 
                    \addplot+[batch100] table [x=timelimit,y=Cmax, col sep=comma] {experiments/batch_size_heuristics/tikz_patients300_batch100.csv}; 
                    \addplot+[batch300] table [x=timelimit,y=Cmax, col sep=comma] {experiments/batch_size_heuristics/tikz_patients300_batch300.csv}; 
                \end{axis}
            \end{tikzpicture}
            \caption{$|\orders|=300$.}
            \label{subfig:orders300}
        \end{subfigure}
        \hfill
        \begin{subfigure}[t]{0.48\textwidth}
            \centering
            \begin{tikzpicture}[baseline=(current axis.north)]
                \begin{axis}[pltstyle,ylabel={}, ymin=21700, ymax=26600, 
            ytick={ 22000,24000, 26000}]
                    \addplot+[batch50]  table [x=timelimit,y=Cmax, col sep=comma] {experiments/batch_size_heuristics/tikz_patients500_batch50.csv}; 
                    \addplot+[batch100] table [x=timelimit,y=Cmax, col sep=comma] {experiments/batch_size_heuristics/tikz_patients500_batch100.csv}; 
                    \addplot+[batch500] table [x=timelimit,y=Cmax, col sep=comma] {experiments/batch_size_heuristics/tikz_patients500_batch500.csv}; 
                \end{axis}
            \end{tikzpicture}
            \caption{$|\orders|=500$.}
            \label{subfig:orders500}
        \end{subfigure}
    \end{minipage}
    \hfill
    % --- Right side: One-Column Legend ---
    \begin{subfigure}[t]{0.16\textwidth}
        \centering
        \begin{tikzpicture}[baseline=(current bounding box.north)]
            \matrix [
                matrix of nodes,
                draw,
                inner sep=0.12cm,
                row sep=1.25mm,
                nodes={anchor=mid west, font=\small}
            ] {
                Batch size: \\
                \tikz[baseline=-0.6ex]{\draw[blue, thick] (0,0) -- (0.4,0) plot[mark=*, mark size=1.5pt] coordinates {(0.2,0)};} 50 \\
                \tikz[baseline=-0.6ex]{\draw[red, thick] (0,0) -- (0.4,0) plot[mark=square*, mark size=1.5pt] coordinates {(0.2,0)};} 100 \\
                \tikz[baseline=-0.6ex]{\draw[green!60!black, thick] (0,0) -- (0.4,0) plot[mark=triangle*, mark size=1.5pt] coordinates {(0.2,0)};} 200 \\
                \tikz[baseline=-0.6ex]{\draw[orange, thick] (0,0) -- (0.4,0) plot[mark=diamond*, mark size=1.5pt] coordinates {(0.2,0)};} 300 \\
                \tikz[baseline=-0.6ex]{\draw[violet, thick] (0,0) -- (0.4,0) plot[mark=pentagon*, mark size=1.5pt] coordinates {(0.2,0)};} 500 \\
            };
        \end{tikzpicture}
    \end{subfigure}

    \caption{The effect of the time limit on $C_\text{max}$ for different batch sizes. Square $8\times8\sim2$ layout, 8 movers.}
    \label{fig:timeexp_batches}
\end{figure}

The results demonstrate that partitioning the problem into smaller batches is an effective strategy for scaling to larger order sets. Although this decomposition introduces minor suboptimalities for smaller instances, it substantially improves computational efficiency. For larger instances, the method attains solutions of comparable or superior quality (Figure~\ref{subfig:orders500}, stitching batches of size 200 (red line) has better performance than no batching (purple line)) within a significantly shorter time horizon, indicating that the trade-off between optimality and tractability is highly favorable under realistic time constraints.

\FloatBarrier

\subsection{Effect of the Routing Phase on Solution Quality}\label{exp:rting}
To assess the effect of the routing phase on system performance, we report the makespan values obtained from the scheduling model and after routing with conflict resolution.
Tables~\ref{table:expRouting:cmax_d10},~\ref{table:expRouting:cmax_d1} present the results for four layout types (ring, square, doubleline and line) across different numbers of movers and two different drug dispensing speeds. The values illustrate the relative changes in makespan between the scheduled and routed solutions.
\begin{table}[ht]
    \centering
    \resizebox{\textwidth}{!}{%
        \begin{tabular}{c|c|c|c|c|c|c|c|c}
             \toprule
            \multirow{2}{*}{\shortstack{Mover\\count}} &
            \multicolumn{2}{c}{Square $C_{\max}$} &
            \multicolumn{2}{c}{Doubleline $C_{\max}$} &
            \multicolumn{2}{c}{Ring $C_{\max}$} &
            \multicolumn{2}{c}{Line $C_{\max}$} \\
            \cmidrule(lr){2-3}\cmidrule(lr){4-5}\cmidrule(lr){6-7}\cmidrule(lr){8-9}
             & Scheduling & Routing & Scheduling & Routing & Scheduling & Routing & Scheduling & Routing \\
            \midrule
            1 & 33427 & +0~(0.0\%) & 34506 & +0~(0.0\%) & 35056 & +0~(0.0\%) & 35233 & +0~(0.0\%) \\
            4 & 8373 & +10~(0.1\%) & 8638 & +21~(0.2\%) & 8806 & +29~(0.3\%) & 8875 & +38~(0.4\%) \\
            8 & 4224 & +24~(0.6\%) & 4399 & +37~(0.8\%) & 4512 & +46~(1.0\%) & 4560 & +53~(1.2\%) \\
            12 & 2942 & +23~(0.8\%) & 3042 & +42~(1.4\%) & 3108 & +56~(1.8\%) & 3286 & +100~(3.0\%) \\
            \bottomrule
        \end{tabular}
    }
    \caption{Makespan values for different movers, before and after routing; standard speed.}
    % \tn{jeste se da ten post-route napsat jen jako $+xx$ čimž se ušetří tak 2 číslice $\xrightarrow{}$ bude to užší}
    \label{table:expRouting:cmax_d10}
\end{table}
\begin{table}[ht]
    \centering
    \resizebox{\textwidth}{!}{%
        \begin{tabular}{c|c|c|c|c|c|c|c|c}
            \toprule
            \multirow{2}{*}{\shortstack{Mover\\count}} &
            \multicolumn{2}{c}{Square $C_{\max}$} &
            \multicolumn{2}{c}{Doubleline $C_{\max}$} &
            \multicolumn{2}{c}{Ring $C_{\max}$} &
            \multicolumn{2}{c}{Line $C_{\max}$} \\
            \cmidrule(lr){2-3}\cmidrule(lr){4-5}\cmidrule(lr){6-7}\cmidrule(lr){8-9}
             & Scheduling & Routing & Scheduling & Routing & Scheduling & Routing & Scheduling & Routing \\
            \midrule
            1 & 4731 & +0~(0.0\%) & 5731 & +0~(0.0\%) & 6202 & +0~(0.0\%) & 6201 & +0~(0.0\%) \\
            4 & 1199 & +10~(0.8\%) & 1492 & +23~(1.5\%) & 1714 & +42~(2.5\%) & 1638 & +50~(3.1\%) \\
            8 & 619 & +23~(3.7\%) & 858 & +39~(4.5\%) & 914 & +52~(5.7\%) & 980 & +53~(5.4\%) \\
            12 & 448 & +37~(8.3\%) & 605 & +56~(9.3\%) & 665 & +60~(9.0\%) & 689 & +79~(11.5\%) \\
            \bottomrule
        \end{tabular} 
    }
    \caption{Makespan values for different movers, before and after routing; fast dispensing speed.}
    \label{table:expRouting:cmax_d1}
\end{table}

The results indicate that the routing overhead is minimal across all layout types and numbers of movers. This demonstrates that the majority of the routing problem is already effectively handled by the scheduling algorithm, confirming that our decomposition approach is well-structured. Table~\ref{table:expRouting:cmax_d1} demonstrates that increasing the dispensing speed leads to a larger routing overhead; as the movers spend comparatively more time in transit than dispensing, the likelihood of spatial conflicts increases, requiring more resolution time.

\FloatBarrier
\section{Conclusion}\label{sec:conclusion}
This work addresses the shift toward mass personalization in pharmaceutical manufacturing by framing the production process as a complex logistics challenge. Inspired by planar transport technologies like the Beckhoff XPlanar, we designed a framework that splits decision-making into a tactical layer for long-term dispenser layout and an operational layer for daily execution.

At the tactical level, we optimized the physical configuration of the production line by utilizing a Mixed-Integer Quadratic Programming model to solve the Packing problem, which allocates dispensers based on drug co-occurrence patterns in historical patient data. Subsequently, a Genetic Algorithm addressed the Placement problem by arranging these dispensers to minimize expected mover travel distances. Our experiments demonstrated that the objective function used in this phase strongly correlates with the final scheduling makespan, while robustness analysis confirmed that these layouts perform consistently well on unseen validation sets.

At the operational level, we focused on daily order fulfillment by formulating the Scheduling problem through Constraint Programming, modeling movers as take-give reservoir resources. To ensure scalability, we handle larger order sets via batch decomposition. The Scheduling phase is complemented by a Routing phase that uses iterative conflict-resolution and DAG-based logic to generate collision-free paths. Results indicate that routing overhead remains minimal, confirming that the schedule accurately approximates physical movements. Practically, this framework serves as a simulation tool for investors and system designers to estimate throughput and identify diminishing returns on hardware resources before physical implementation.
% \clearpage

\section*{Declaration of generative AI and AI-assisted technologies in the manuscript preparation process.}
During the preparation of this work, the authors used ChatGPT and Gemini to assist with refining linguistic clarity, improving structural flow, and generating TikZ code for scientific diagrams. All diagrams are based on data collected and processed independently by the authors; the AI tools had no access to the underlying research data, ensuring the technical accuracy of the visualizations. Additionally, generative AI was used to produce the text-to-speech audio for the supplemental video. The authors reviewed and edited all outputs and take full responsibility for the content and integrity of the final publication and all associated multimedia.
\section*{Acknowledgments}
This work was supported by the European Union under the ROBOPROX (Robotics and advanced industrial production) project (reg. no. 
CZ.02.01.01/00/22\_008/0004590) and the Grant Agency of the Czech Technical University in
Prague, grant No. SGS26/117/OHK3/1T/37.
\bibliography{references}

\clearpage
\appendix
\counterwithin{figure}{section}
\counterwithin{table}{section}
\section{Notation}
\begin{table}[ht]
\footnotesize
\renewcommand{\arraystretch}{1.1} % <--- Adds 30% more vertical space globally
\centering
\begin{tabularx}{\linewidth}{>{\centering\arraybackslash}m{0.05\linewidth} X >{\centering\arraybackslash}m{0.14\linewidth} X}

\toprule
\text{Symbol} & \text{Description} & \text{Symbol} & \text{Description} \\
\midrule

\label{sym:layout}\layout & Topology layout of the environment. & \label{sym:numTiles}\numTiles & Number of tiles. \\ 

\label{sym:interfaces}\interfaces & Set of interface tiles. & \label{sym:numInterfaces}\numInterfaces & Number of interfaces. \\ 

\label{sym:numDispensers}\numDispensers & Number of dispensers. & \label{sym:patients}\orders & Set of orders. \\ 

\label{sym:movers}\movers & Set of movers. & \label{sym:numMovers}\numMovers & Number of movers. \\ 

\label{sym:maxMovers}\maxMovers & Maximum number movers allowed. & \label{sym:drugs}\drugs & Set of drugs. \\ 

\label{sym:packing}\packedDispensers & Packed dispensers & \label{sym:placementExt}\placementExt & Placement function for dispensers to tiles. \\ 

\label{sym:patientsHistory}\ordersHistory & Set of historical order data. & \label{sym:tasks}\tasks & Set of current orders. \\ 

\label{sym:distMatrix}\distMatrix & Distance matrix over the set of tiles. & \label{sym:fillDuration}\fillDuration & Duration of dispensing of drug $g$ for order $p$. \\ 

\label{sym:tasksForPatient}\tasksForPatient{p} & Set of drugs requested by order $p \in \orders$. & \label{sym:taskAlternatives}\taskAlternatives & Set of alternative operations. \\ 

\label{sym:patientPathTime}\patientPathTime & Time to traverse the path for order $p \in \orders$. & \label{sym:sites}\sites & Set of resting sites. \\ 

\label{sym:trayOperation}\trayOperation & Constant time for cartridge swap. & \label{sym:transits}\transits & Set of transits between sites with idle periods. \\ 
\bottomrule
\end{tabularx}
\caption{List of symbols.} 
\label{sym:layout_table}
\end{table}
\FloatBarrier
\section{Packing: Maximization of Pairwise Correlations}\label{sec:appendix_correlations}

This section describes the additional step in the packing, which aims to maximize the sum of the pairwise correlations.
The motivation comes from the correlation matrix in Figure~\ref{fig:corr-matrix} that we derived from the real order data~(National Center for Health Statistics, 2012).
The matrix displays the pairwise correlation of the individual drugs appearing in the patients' prescriptions.
By manual investigation, we found that, for example, drugs for high blood pressure tend to co-occur with drugs for diabetes, i.e., a positive correlation.
On the other hand, a negative correlation was discovered, for example, for a pair of drugs, both used for cardiac patients that work as alternatives to each other---one is a stronger variant of the other.
% correlation matrix:
\begin{figure}[htbp]
\pgfplotsset{
  colormap={bluewhitebrown}{
    rgb(0cm)=(0.25,0.12,0.10)  % dark brown for negative
    rgb(0.5cm)=(1,1,1)         % 0 -> white
    rgb(1cm)=(0.05,0.20,0.45)  % deep blue for positive
  }
}
    \centering
    \begin{tikzpicture}
          \begin{axis}[
                axis on top,
                axis equal image,
                width=0.4\textwidth,
                xlabel={drug $g$ [-]},
                ylabel={drug $g$ [-]},
                point meta=explicit,
                point meta min=-0.3,
                point meta max=0.3,
                xmin=-1,xmax=40,
                ymin=-1,ymax=40,
                % colormap/viridis,
                colormap name=bluewhitebrown,
                % colormap={whiteviridis}{
                %     rgb255(0cm)=(33,145,140);   
                %     rgb255(0.25cm)=(59,82,139);   % Deep blue from viridis
                %     rgb255(0.5cm)=(255,255,255);  
                %     rgb255(0.75cm)=(94,201,98);   % Green from viridis
                %     rgb255(1cm)=(37,55,250)      % Bright blue accent at max
                % },
                colorbar,
                % colorbar style={
                %     ylabel=expected tile utilization [-]
                % }
            ]
                \addplot [
                    matrix plot*,
                ] table [meta index=2] {experiments/correlation_matrix.csv};
            \end{axis}
        \end{tikzpicture}
    \caption{Example of drug correlation matrix $\m{O}=\left\{o_{g,g^\prime}\right\}_{g,g^\prime\in\drugs}$ derived from real data (with subtracted main diagonal).}
    \label{fig:corr-matrix}
\end{figure}

Therefore, our idea would be to find, among all solutions with the minimal maximum tile utilization, the one that also maximizes the sum of pairwise correlations between drugs when placed on the same tile.
In this way, we can reduce the number of steps, as often co-occurring drugs are filled on the same tile without moving or transitioning to a different tile.
Initially, we experimented with a modification of model (4.1)--(4.11) that would express this secondary objective via the scalarization technique. %(implemented by Gurobi's hierarchical objective functions).
However, the results were largely unsatisfactory.
Thus, we decided to split the solution into two phases: first, we solve model (4.1)--(4.11), and then fix the maximum tile utilization \hatMaxLoad and the number of dispensers \hatDispenserCount for each drug $g$.
%, to re-optimize the occurrence of dispensers on the individual tiles to maximize the sum of pairwise correlations. 

After that, we solve a correlation model that aims to maximize pairwise correlations given the maximum utilization bound for each tile.
The inputs to the correlation model are feasible values from the packing model (4.1)--(4.11), where \hatMaxLoad is achieved maximum tile utilization, the number of tiles that drug $g$ occupies \hatDispenserCount and the expected utilization of a single dispenser \hatDispenserLoad of drug $g$.
These values are fixed in the correlation model and not subject to optimization; the only free variables are tile-drug assignments $y_{k,g}\in\{0,1\}$:
% \textbf{Given:}
% \begin{itemize}
%     \item \drugUtil: total estimated utilization time for medicine $i\in I$
%     \item \hatMaxLoad: maximum dispenser utilization
%     \item \hatDispenserCount: number of copies of drug $i\in I$
%     \item \hatDispenserLoad: expected utilization of a single dispensor of drug $i$
% \end{itemize}
% \textbf{Variables}:
% \begin{itemize}
%     \item $y_{k,i}$ assignment of dispensor of drug $i$ to tile $k$
% \end{itemize}
\begin{align}
   &\max \sum^{\numTiles}_{k=1}\sum_{g,g^\prime\in \drugs, g\neq g^\prime} o_{g,g^\prime}\cdot y_{k,g}\cdot y_{k,g^\prime} \label{eq:pack2-obj}\\
   \text{subject to}&\notag \\
   %&\hatMaxLoad \geq \tileLoad\quad \forall k \in \{1,\ldots, n\}\\
    &\sum^{\numTiles}_{k=1} y_{k,g} = \hatDispenserCount \quad \forall g \in \drugs\\
    % &\dispenserLoad \cdot \hatDispenserCount = \drugUtil \quad \forall i \in I \\
    &1\leq \sum_{g\in \drugs}y_{k,g} \leq d_\text{max} \quad \forall k \in \{1, \ldots, \numTiles\}\\
    &\sum_{g\in \drugs} \hatDispenserLoad\cdot y_{k,g} \leq \hatMaxLoad \quad \forall k \in  \{1, \ldots, \numTiles-\numInterfaces\}\\
    %&\sum^n_{k=1} \sum_{i\in I} y_{k,i} \leq d_{max}\\
    %&\hatDispenserCount \geq 1 \quad \forall i \in I\\
    \text{where}&\notag\\
    &y_{k,g} \in \{0,1\} \quad \forall k \in \{1,\ldots, \numTiles\}, \forall g \in \drugs. \label{eq:pack2-finish}
    %&\dispenserLoad \in \mathbb{R}^+_0 \quad \forall i \in I
    %&\tileLoad \in \mathbb{R}^+_0 \quad \forall k \in \{1,\ldots, n\}
\end{align}
The model \eqref{eq:pack2-obj}--\eqref{eq:pack2-finish} is an MIQP model, which resembles a variant of the quadratic multiple knapsack problem.
%The number of copies of dispenser $i$ is fixed from the previous stage by \hatDispenserCount, as well as the maximal tile load \hatMaxLoad.
%Within that, we group dispensers on tiles such that we maximize total pairwise correlations.
The optimization of pairwise correlations indeed has some merits. 
For example, an example packing produced by the first stage has a total sum of pairwise correlations of $-0.31$, while after optimization, it is $1.88$---with the same maximum tile utilization.

\section{Experiments: Extended Data Tables}\label{sec:extended_tables}
Some of the tables in the manuscript were condensed to highlight key trends and significant values for clarity. This section provides the extended versions of those tables.
\begin{table}[ht]
    \centering
    \resizebox{\textwidth}{!}{%
        \begin{tabular}{c|c|c|c|c|c|c|c|c}
             \toprule
            \multirow{2}{*}{\shortstack{Mover\\count}} &
            \multicolumn{2}{c}{Square $C_{\max}$} &
            \multicolumn{2}{c}{Doubleline $C_{\max}$} &
            \multicolumn{2}{c}{Ring $C_{\max}$} &
            \multicolumn{2}{c}{Line $C_{\max}$} \\
            \cmidrule(lr){2-3}\cmidrule(lr){4-5}\cmidrule(lr){6-7}\cmidrule(lr){8-9}
             & Scheduling & Routing & Scheduling & Routing & Scheduling & Routing & Scheduling & Routing \\
            \midrule
        1 & 33427 & +0~(0.0\%) & 34506 & +0~(0.0\%) & 35056 & +0~(0.0\%) & 35233 & +0~(0.0\%) \\
        2 & 16722 & +7~(0.0\%) & 17295 & +22~(0.1\%) & 17556 & +14~(0.1\%) & 17663 & +28~(0.2\%) \\
        3 & 11163 & +11~(0.1\%) & 11524 & +21~(0.2\%) & 11782 & +35~(0.3\%) & 11788 & +31~(0.3\%) \\
        4 & 8373 & +10~(0.1\%) & 8638 & +21~(0.2\%) & 8806 & +29~(0.3\%) & 8875 & +38~(0.4\%) \\
        5 & 6700 & +14~(0.2\%) & 6943 & +23~(0.3\%) & 7091 & +30~(0.4\%) & 7131 & +39~(0.5\%) \\
        6 & 5602 & +17~(0.3\%) & 5822 & +27~(0.5\%) & 5947 & +41~(0.7\%) & 5992 & +54~(0.9\%) \\
        7 & 4808 & +17~(0.4\%) & 5030 & +24~(0.5\%) & 5092 & +48~(0.9\%) & 5151 & +75~(1.5\%) \\
        8 & 4224 & +24~(0.6\%) & 4399 & +37~(0.8\%) & 4512 & +46~(1.0\%) & 4560 & +53~(1.2\%) \\
        9 & 3762 & +17~(0.5\%) & 3961 & +33~(0.8\%) & 4012 & +58~(1.4\%) & 4137 & +72~(1.7\%) \\
        10 & 3388 & +19~(0.6\%) & 3557 & +31~(0.9\%) & 3623 & +73~(2.0\%) & 3757 & +71~(1.9\%) \\
        11 & 3137 & +33~(1.1\%) & 3310 & +50~(1.5\%) & 3372 & +80~(2.4\%) & 3659 & +104~(2.8\%) \\
        12 & 2942 & +23~(0.8\%) & 3042 & +42~(1.4\%) & 3108 & +56~(1.8\%) & 3286 & +100~(3.0\%) \\
        \bottomrule
    \end{tabular}
    }
    \caption{Full version of Table~4.}
    \label{table:expRouting:cmax_d10_full}
\end{table}
\begin{table}[ht]
    \centering
    \resizebox{\textwidth}{!}{%
        \begin{tabular}{c|c|c|c|c|c|c|c|c}
             \toprule
            \multirow{2}{*}{\shortstack{Mover\\count}} &
            \multicolumn{2}{c}{Square $C_{\max}$} &
            \multicolumn{2}{c}{Doubleline $C_{\max}$} &
            \multicolumn{2}{c}{Ring $C_{\max}$} &
            \multicolumn{2}{c}{Line $C_{\max}$} \\
            \cmidrule(lr){2-3}\cmidrule(lr){4-5}\cmidrule(lr){6-7}\cmidrule(lr){8-9}
             & Scheduling & Routing & Scheduling & Routing & Scheduling & Routing & Scheduling & Routing \\
            \midrule
            1 & 4731 & +0~(0.0\%) & 5731 & +0~(0.0\%) & 6202 & +0~(0.0\%) & 6201 & +0~(0.0\%) \\
            2 & 2386 & +5~(0.2\%) & 2897 & +10~(0.3\%) & 3181 & +13~(0.4\%) & 3296 & +14~(0.4\%) \\
            3 & 1592 & +9~(0.6\%) & 1980 & +21~(1.1\%) & 2171 & +25~(1.2\%) & 2167 & +32~(1.5\%) \\
            4 & 1199 & +10~(0.8\%) & 1492 & +23~(1.5\%) & 1714 & +42~(2.5\%) & 1638 & +50~(3.1\%) \\
            5 & 962 & +16~(1.7\%) & 1187 & +33~(2.8\%) & 1362 & +27~(2.0\%) & 1410 & +32~(2.3\%) \\
            6 & 817 & +22~(2.7\%) & 1012 & +43~(4.2\%) & 1131 & +42~(3.7\%) & 1158 & +45~(3.9\%) \\
            7 & 710 & +22~(3.1\%) & 935 & +36~(3.9\%) & 969 & +46~(4.7\%) & 987 & +44~(4.5\%) \\
            8 & 619 & +23~(3.7\%) & 858 & +39~(4.5\%) & 914 & +52~(5.7\%) & 980 & +53~(5.4\%) \\
            9 & 571 & +23~(4.0\%) & 728 & +48~(6.6\%) & 813 & +46~(5.7\%) & 885 & +66~(7.5\%) \\
            10 & 524 & +24~(4.6\%) & 720 & +50~(6.9\%) & 743 & +41~(5.5\%) & 834 & +69~(8.3\%) \\
            11 & 500 & +34~(6.8\%) & 642 & +58~(9.0\%) & 712 & +43~(6.0\%) & 807 & +107~(13.3\%) \\
            12 & 448 & +37~(8.3\%) & 605 & +56~(9.3\%) & 665 & +60~(9.0\%) & 689 & +79~(11.5\%) \\
            \bottomrule
        \end{tabular}
        }
    \caption{Full version of Table~5.}
    \label{table:expRouting:cmax_d1_full}
\end{table}

\end{document}